\documentclass[11pt,a4paper]{article}

\usepackage[english]{babel}
\usepackage{tikz}
\usepackage{graphicx}
\usepackage{float}
\usepackage{amssymb}
\usepackage{amsmath}
\usepackage{amsthm}
\usepackage{mathrsfs}
\usepackage[T1]{fontenc}
\usepackage{inputenc}
\usepackage{lmodern}
\usepackage{hyperref}
\usepackage{geometry}
\usepackage{changepage}
\changepage{0pt}{}{}{}{}{0pt}{}{0pt}{10pt}
\usepackage[numbers]{natbib}
\hypersetup{
pdfpagemode=none,
pdftoolbar=true,        
pdfmenubar=true,        
pdffitwindow=false,     
pdfstartview={Fit},    
pdftitle={Oscillating behaviour of the spectrum for a plasmonic problem in a domain with a rounded corner},    
pdfauthor={L. Chesnel, X. Claeys, S.A. Nazarov},     
pdfsubject={},  
pdfcreator={L. Chesnel, X. Claeys, S.A. Nazarov},   
pdfproducer={L. Chesnel, X. Claeys, S.A. Nazarov}, 
pdfkeywords={}, 
pdfnewwindow=true,      
colorlinks=true,       
linkcolor=magenta,          
citecolor=red,        
filecolor=cyan,      
urlcolor=blue           
}

\newcommand{\dsp}{\displaystyle}
\newcommand{\lbr}{\lbrack}
\newcommand{\rbr}{\rbrack}
\newcommand{\eps}{\varepsilon}
\newcommand{\om}{\omega}
\newcommand{\Om}{\Omega}
\newcommand{\mrm}[1]{\mathrm{#1}}

\newcommand{\bfx}{\boldsymbol{x}}
\newcommand{\bfxi}{\boldsymbol{\xi}}

\newcommand{\Cplx}{\mathbb{C}}
\newcommand{\N}{\mathbb{N}}
\newcommand{\R}{\mathbb{R}}
\newcommand{\Z}{\mathbb{Z}}
\renewcommand{\div}{\mrm{div}}
\newcommand{\mL}{\mrm{L}}
\newcommand{\mH}{\mrm{H}}
\newcommand{\mV}{\mrm{V}}
\newcommand{\mVc}{\mathcal{V}}

\newcommand{\mA}{\mrm{A}}

\newcommand{\out}{\mrm{out}}

\newcommand{\ind}{\mrm{Ind}}

\newcommand{\Ker}{\mrm{Ker}}
\newcommand{\spec}{\mathfrak{S}}
\newcommand{\sg}{\mathbf{s}}
\newcommand{\sgp}{\mathbf{s}_{+}}
\newcommand{\sgm}{\mathbf{s}_{-}}
\newcommand{\sgpm}{\mathbf{s}_{\pm}}

\newcommand{\Sgp}{\mathbf{S}_{+}}
\newcommand{\Sgm}{\mathbf{S}_{-}}
\newcommand{\Sgpm}{\mathbf{S}_{\pm}}
\newtheorem{theorem}{Theorem}[section]
\newtheorem{lemma}{Lemma}[section]
\newtheorem{remark}{Remark}[section]
\newtheorem{definition}{Definition}[section]

\newtheorem{proposition}{Proposition}[section]

\newtheorem{Assumption}{Assumption}
\makeatletter
\renewenvironment{proof}[1][\proofname]{%
  \par\pushQED{\qed}\normalfont%
  \topsep6\p@\@plus6\p@\relax
  \trivlist\item[\hskip\labelsep\bfseries#1\@addpunct{.}]%
  \ignorespaces
}{%
  \popQED\endtrivlist\@endpefalse
}
\makeatother
\begin{document}

~\vspace{0.1cm}
\begin{center}
{\sc \bf\LARGE 
Oscillating behaviour of the spectrum for a plasmonic\\[4pt]
problem in a domain with a rounded corner
}
\end{center}

\begin{center}
\textsc{Lucas Chesnel}$^1$, \textsc{Xavier Claeys}$^2$, \textsc{Sergei A. Nazarov}$^{3,\,4,\,5}$\\[16pt]
\begin{minipage}{0.82\textwidth}
{\small
$^1$ INRIA/Centre de mathématiques appliquées, \'Ecole Polytechnique, Université Paris-Saclay, Route de Saclay, 91128 Palaiseau, France;\\
$^2$ Laboratory Jacques Louis Lions, University Pierre et Marie Curie, 4 place Jussieu, 75005 Paris, France; \\
$^3$ St. Petersburg State University, Universitetskaya naberezhnaya, 7-9, 199034, St. Petersburg, Russia;\\
$^4$ Peter the Great St. Petersburg Polytechnic University, Polytekhnicheskaya ul, 29, 195251, St. Petersburg, Russia;\\
$^5$ Institute of Problems of Mechanical Engineering, Bolshoy prospekt, 61, 199178, V.O., St. Petersburg, Russia.\\[10pt]
E-mails: \texttt{lucas.chesnel@inria.fr}, \texttt{claeys@ann.jussieu.fr}, \texttt{srgnazarov@yahoo.co.uk}\\[-14pt]
\begin{center}
(\today)
\end{center}
}
\end{minipage}
\end{center}
\vspace{0.4cm}

\noindent\textbf{Abstract.} 
We investigate the eigenvalue problem $-\div(\sigma \nabla u) = \lambda u\ (\mathscr{P})$ in a 2D domain $\Om$ divided into two regions $\Om_{\pm}$. We are interested in situations where $\sigma$ takes positive values on $\Om_{+}$ and negative ones on $\Om_{-}$. Such problems appear in time harmonic electromagnetics in the modeling of plasmonic technologies. In a recent work \cite{ChCNSu}, we highlighted an unusual instability phenomenon for the source term problem associated with $(\mathscr{P})$: for certain configurations, when the interface between the subdomains $\Om_{\pm}$ presents a rounded corner, the solution may depend critically on the value of the rounding parameter. In the present article, we explain this property studying the eigenvalue problem $(\mathscr{P})$. We provide an asymptotic expansion of the eigenvalues and prove error estimates. We establish an oscillatory behaviour of the eigenvalues as the rounding parameter of the corner tends to zero. We end the paper illustrating this phenomenon with numerical experiments.\\


\section{Introduction}
In electromagnetics, it is well-known that the dielectric permittivity $\eps$ of metals has a negative real part at optical wavelength. Because of this property, some waves called \textit{surface plasmon polaritons} can propagate at the interface between a metal and a classical dielectric \cite{BaDE03,ZaSM05}. Physicists seek to use the plasmons in order to propagate information and plasmonic technologies appear a promising solution for the miniaturization of electronic devices. In this context, an important issue is to focus energy in some confined regions of space. To achieve this, one approach consists in using metallic structures with sharp geometries involving corners, tips, edges, ... \cite{Stoc04,BVNMRB08}. 

When losses are neglected, which is often desired for applications and which is reasonable to assume for certain ranges of frequencies, the physical parameters in devices involving negative materials change sign in the domain of interest. In this case, the study of time harmonic Maxwell's equations can not be handled using the classical methods \cite{OlRa08,FeRa09}. New techniques have to be developed \cite{KhPS07,PePu12,BoTr13,Grie14}. Using a variational approach, it has been proved in \cite{BoCZ10,BoCC12} that the scalar problem equivalent to Maxwell's equations in 2D configurations, turns out to be of Fredholm type in the classical functional framework only whenever the contrast (ratio of the values
of $\eps$ across the interface) lies outside some interval, which always contains the value $-1$. Moreover, this interval reduces to $\{-1\}$ if and if only the interface between the positive material and the negative material is smooth (of class $\mathscr{C}^1$). Analogous results have been obtained by techniques of boundary integral equations in \cite{CoSt85} long before the age of plasmonic technologies. The numerical approximation of the solution of this scalar problem, based on classical finite element methods, has been investigated in \cite{BoCZ10,NiVe11,ChCiAc}. Under some assumptions on the meshes, the discretized problem is well-posed and its solution converges to the solution of the continuous problem. The study of Maxwell's equations has been carried out in \cite{BoCCSud}. The influence of corners of the interface, studied in \cite{Meix72,SuSK06,WaKS08,HePe12}, has been clarified in \cite{BoCCSuc,BCCC15} for the scalar problem (see also the previous works \cite{DaTe97,BoDR99,Ramd99} where the general theory \cite{Kond67,MaPl77,NaPl94,KoMR97} is extended to this configuration where the operator is not strongly elliptic). In \cite{BoCCSuc}, following \cite{NaPl94,NaTa08,BaNa09,NaTa11}, the authors prove that when the contrast of the physical parameters lies inside the critical interval, Fredholm property is lost because of the existence of two strongly oscillating singularities at the corner. In such 
a case, Fredholmness can be recovered by adding to the functional framework one of the two singularities, selected by means of a limiting absorption principle, and by working in a special weighted Sobolev setting with weight centered at the corner \cite{BoCC13}. This functional framework amounts to prescribing a radiation condition at the corner.

Such a special functional framework seems an uncomfortable situation though, at least from a physical point of view. Indeed, it leads to working with solutions 
which are not of finite energy (their $\mH^1$-norm is infinite). A possible regularization that may appear natural would consist in considering slightly rounded 
corners, instead of real corners at the interface. In the sequel, we will denote $\delta>0$ the (small) parameter corresponding to the rounding of the corner. In 
a recent work \cite{ChCNSu}, we prove an instability phenomenon for the source term problem set in such a geometry: when the contrast of the physical parameters 
belongs to the critical interval, the solution depends critically on the value of $\delta$ and does not converge, even for very weak norms, as $\delta$ tends to 
zero. In the present article, our goal is to study the properties of the associated eigenvalue problem.

We use asymptotic analysis to carry out this study. We do not derive a complete asymptotic expansion of the eigenvalues though. 
Asymptotic techniques here only stand as an intermediate (yet crucial) tool for the description of the predominant behaviour of the boundary value operator. 
Our analysis leads to the conclusion that this operator and the corresponding eigenvalues asymptotically behave, as $\delta\to 0$, 
as an operator that admits a non-constant periodic dependency with respect to $\ln\delta$.

The outline of this paper is as follows. In Section \ref{Description of the problem}, we describe 
in detail the problem and the geometry that we want to consider, namely an eigenvalue problem for a diffusion equation with a sign-changing 
coefficient in the principal part. The domain is a bounded cavity divided into two regions by an interface containing a rounded corner close to the boundary. 
As above, the rounding of the corner is described by some small parameter $\delta$ ($\delta=0$ corresponds to the geometry with a ``perfect'' corner in the interface). 
In Section \ref{LimitProblem}, we study the spectral properties of $\mA^0$, the natural limit operator for $\delta=0$. More precisely, we recall some results of 
\cite{BoDR99,Ramd99} which indicate that $\mA^0$ is not self-adjoint. Using Kondratiev's theory \cite{Kond67}, and more precisely, the results established in \cite{BoCC13}, 
we then describe all the self-adjoint extensions of $\mA^0$. In Section \ref{FormalAsymp}, we propose a formal asymptotic expansion of the eigenpairs of $\mA^{\delta}$, 
the natural operator set in the geometry with a slightly rounded corner. This expansion is built using matched asymptotics \cite{MaNP81}, \cite[Chap.\,4,\,5]{MaNP00}. 
In particular, in accordance with \cite{KaNa00,Naza99}, we find that the eigenpairs of $\mA^{\delta}$ behave asymptotically as the eigenpairs of some self-adjoint extension 
of $\mA^0$. The originality lies in the fact that the latter self-adjoint extension depends periodically of $-\ln\delta$. In Section \ref{SectionMainThm}, we prove the main 
result of the paper, namely Theorem \ref{thmMajor}. We establish that, asymptotically, all the eigenvalues of $\mA^{\delta}$ are periodic in $\ln\delta$-scale as $\delta$ tends 
to zero. Section \ref{SectionSourceTermpb} is devoted to showing an important intermediate proposition allowing to justify the formal asymptotic analysis and to provide error 
estimates. We conclude the paper with numerical experiments illustrating the results we obtained in the previous section.

\section{Description of the problem}\label{Description of the problem}

Let $\Omega\subset \R^{2}$ be a domain, \textit{i.e.} a bounded and connected open subset of $\R^{2}$, with Lipschitz boundary $\partial\Om$ (see Figure \ref{GeneralGeom} below). We assume that 
$\Omega$ is partitioned into two sub-domains $\Omega_{\pm}^{\delta}$ so that 
$\overline{\Omega} = \overline{\Omega_{+}^{\delta}}\cup \overline{\Omega_{-}^{\delta}}$ with 
$\Omega_{+}^{\delta}\cap \Omega_{-}^{\delta} = \emptyset$. We consider a smooth curve $\Sigma^{0}$ that intersects $\partial\Omega$ at only two points $O$ and $O'$. We assume that $\partial\Om$ and $\Sigma^{0}$ are straight in a neighbourhood of $O$, $O'$, and that at $O'$, $\Sigma^{0}$ is perpendicular to $\partial\Om$. We also assume that the interface $\Sigma^{\delta} := 
\overline{\Omega_{+}^{\delta}}\cap \overline{\Omega_{-}^{\delta}}$ coincides 
with $\Sigma^{0}$ outside the disk $\mrm{D}(O,\delta)$. We denote $\mathscr{S}:=O\cup O'$ and we introduce $\boldsymbol{n}^{\delta}$ the unit outward normal vector to $\Sigma^{\delta}$ directed from $\Omega_{+}^{\delta}$ to $\Omega_{-}^{\delta}$.

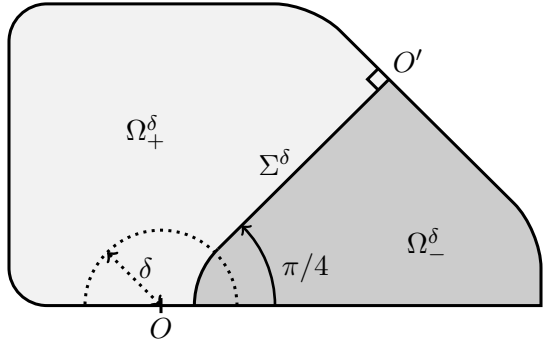
\begin{figure}[!ht]
\centering
\begin{tikzpicture}
\draw[draw=none, line width = 0.4mm,fill=gray!10,rounded corners=5mm] (-4,1)--(3,1)--(3,2)--(0,5)--(-4,5)--cycle;
\fill[draw=none, line width = 0.4mm,fill=gray!40,rounded corners=5mm] [sharp corners](-1.25,1.75) arc (135:180:1.05) -- (3,1) [rounded corners=5mm]-- (3,2) [sharp corners]-- (1,4) -- cycle ;
\draw[line width = 0.4mm] [rounded corners=5mm](-4,1)[sharp corners]--(3,1)[rounded corners=5mm]--(3,2)--(0,5)--(-4,5)--cycle;
\draw[black, line width = 0.4mm] (-1.25,1.75)--(1,4);
\draw[black, line width = 0.4mm] (-1.25,1.75) arc (135:180:1.060660172);
\draw[black, line width = 0.4mm,<->,dotted] (-2,1)--(-2.707106,1.707106);
\node at (-2.2,1.5){$\delta$};
\draw[black, line width = 0.4mm,dotted] (-1,1) arc (0:180:1);
\node at (1.5,1.8){$\Omega_{-}^{\delta}$};
\node at (-2.2,3.3){$\Omega_{+}^{\delta}$};
\node at (-0.5,2.9){$\Sigma^{\delta}$};
\draw[black, line width =1pt,->] (-0.5,1) arc (0:45:1.5);
\node at (-0.1,1.5){$\pi/4$};
\draw[black, line width = 0.4mm](0.8586,4.1414)--(0.7172,4)--(0.8586,3.8586);
\node at (-2,0.7){$O$};
\draw[black, line width = 0.4mm](-2,0.9)--(-2,1.1);
\node at (1.25,4.25){$O'$};
\end{tikzpicture}
\caption{Geometry of the problem.\label{GeneralGeom}}
\end{figure}

\noindent In the sequel, we shall denote by $(r,\theta)$ the polar 
coordinates centered at $O$ such that $\theta=0$ or $\pi$ at the boundary in a neighbourhood of $O$. As $\delta\to 0$, 
the sub-domains $\Omega_{\pm}^{\delta}$ turn into $\Omega_{\pm}^{0}$ and we assume that there exists a disk $\mrm{D}(O,2r_{0})$ 
centered at $O$ such that $\Omega_{-}^{0}\cap \mrm{D}(O,2r_{0}) =  \{ (r\cos\theta, r\sin\theta)\in\R^{2}\;\vert\; 0<r<2r_{0}\;,\; 
0<\theta<\pi/4 \}$ and  $\Omega_{+}^{0}\cap \mrm{D}(O,2r_{0}) 
=  \{ (r\cos\theta, r\sin\theta)\in\R^{2}\;\vert\; 0<r<2r_{0}\;,\; \pi/4<\theta<
\pi \}$. We consider the value $\pi/4$
for the aperture of the corner for a reason which appears in \S \ref{StudySing} (the calculus of $\Lambda$ in (\ref{SingExp}) can be made explicit in this case). However, there is no difficulty to adapt the rest of the forthcoming analysis for other values of this angle (see the discussion in Remark \ref{RmkAngleQconque}). To fix ideas, and without restriction, we assume 
that we can take $r_{0}=1$, \textit{i.e.} we assume that there holds $(\mrm{D}(O,2)\cap \R\times\R^{\ast}_{+})\subset\Om$, where  $\R\times\R^{\ast}_{+}=\{(x,y)\in\R^2\,|\,y>0\}$.\\
\newline
With this geometry, we associate a set of cut-off functions which we will refer to throughout this paper. We introduce  $\psi\in\mathscr{C}^{\infty}(\R,[0;1])$ such that $\psi(r)=1$ for $0\le r \le 1$ and $\psi(r)=0$ for $r\ge 2$. We define $\chi:=1-\psi$. Finally, for $t>0$, we denote $\psi_t$, $\chi_t$ the functions  such that $\psi_t(r)=\psi(r/t)$, $\chi_t(r)=\chi(r/t)$ (see Figure \ref{cut-off functions}).

\begin{figure}[!ht]
\centering\includegraphics{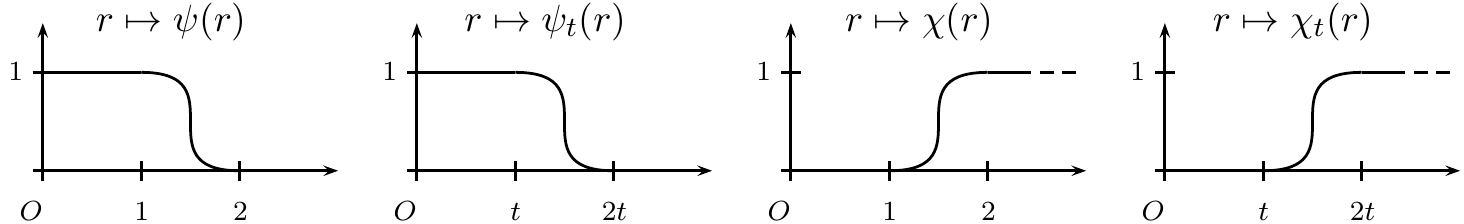}\vspace{-0.2cm}
\caption{Cut-off functions.\label{cut-off functions}}
\end{figure}

\subsection{Geometry of the rounded corner}\label{paragraph geom} The set $\Sigma^{\delta}\cap
\mrm{D}(O,\delta)$ will be defined as follows. Let $ \Xi := \R\times\R^{\ast}_{+}$ 
refer to the upper half plane partitioned by means of two open sets  
$\Xi_{\pm}$ such that $ \overline{\Xi} = \overline{\Xi_{+}}\cup\overline{\Xi_{-}}$ and $\Xi_{+}\cap\Xi_{-}=\emptyset$.
We assume that $\Gamma := \overline{\Xi_{+}}\cap\overline{\Xi_{-}}$ 
is a curve $\Gamma = \{\varphi_{\Gamma}(t),\;t\in\lbr 0;+\infty)\}$
where $\varphi_{\Gamma}$ is a $\mathscr{C}^{\infty}$ function such that 
$\partial_{t}\varphi_{\Gamma}(0)$ is orthogonal to the $x$-axis
and $\varphi_{\Gamma}(t) = (t,t)$ for $t\geq 1$, see Figure \ref{Frozen geometry} below.
In a neighbourhood of the corner,  we assume that 
$\Omega_{\pm}^{\delta}$ can be defined from $\Xi_{\pm}$ by 
self similarity: 
\[
\Omega_{\pm}^{\delta}\cap\mrm{D}(O,\delta) = \{\;\bfx\in \R^{2}\;\vert\;
\bfx/\delta \in \Xi_{\pm}\cap\mrm{D}(O,1)\;\}.
\]

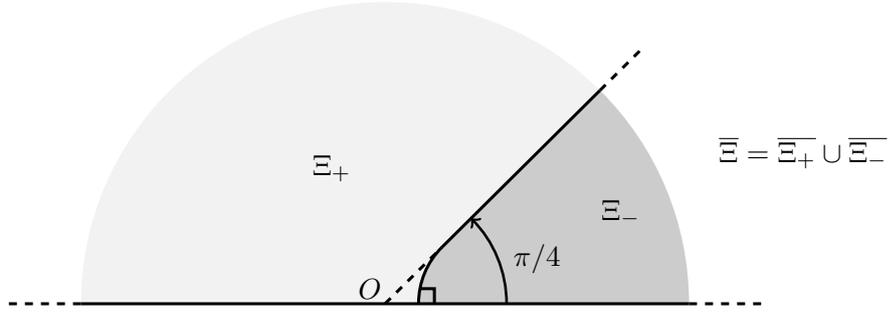
\begin{figure}[!ht]
\centering
\begin{tikzpicture}
\fill[draw=none, line width = 0.4mm,fill=gray!10] (-2,1)--(2,1) arc (0:180:4);
\fill[draw=none, line width = 0.4mm,fill=gray!40] (2,1) arc (0:45:4)-- (-1.25,1.75) arc (135:180:1.05) -- cycle;
\draw[black,line width = 0.4mm,dashed](-6,1)--(-7,1);
\draw[black,line width = 0.4mm,dashed](2,1)--(3,1);
\draw[black,line width = 0.4mm](-6,1)--(2,1);
\draw[black,line width = 0.4mm,dashed](-2,1)--(-1.25,1.75);
\draw[black,line width = 0.4mm](-1.25,1.75)--(0.8284,3.8284);
\draw[black,line width = 0.4mm,dashed](0.8284,3.8284)--(1.4,4.4);
\draw[black,line width = 0.4mm] (-1.25,1.75) arc (135:180:1.05);
\draw[black,line width = 0.4mm](-1.55,1.2)--(-1.35,1.2)--(-1.35,1);
\node at (1.1,2.2){$\Xi_{-}$};
\node at (-2.7,2.8){$\Xi_{+}$};
\node at (3.5,3){$ \overline{\Xi} = \overline{\Xi_{+}}\cup \overline{\Xi_{-}}$};
\node at (-7.5,3){\phantom{$ \overline{\Xi} = \overline{\Xi_{+}}\cup \overline{\Xi_{-}}$}};
\draw[black, line width =1pt,->] (-0.4,1) arc (0:45:1.6);
\node at (0,1.6){$\pi/4$};
\node at (-2.2,1.2){$O$};
\end{tikzpicture}
\caption{Frozen geometry. \label{Frozen geometry}}
\end{figure}

\subsection{The problem under study}\label{The problem under study}
First of all, let us set basic notations. In the sequel, for any open subset 
$\om\subset \R^{d}$ with $d=1,2$, the space  $\mrm{L}^{2}(\om)$ will refer to the Lebesgue 
space of square integrable functions equipped with the scalar product 
$(u,v)_{\mL^{2}(\omega)}:= \int_{\omega}u\, \overline{v}\, d\bfx$.
We denote $\Vert v\Vert_{\mrm{L}^{2}(\omega)}:=\sqrt{(v,v)_{\omega}}$.
We will consider the Sobolev space $\mrm{H}^{1}(\omega) := \{v\in\mrm{L}^{2}(\omega)\;\vert\;
 \nabla v\in \mrm{L}^{2}(\omega)\}$, and define $\mrm{H}^{1}_{0}(\omega) :=\{ v\in 
\mrm{H}^{1}(\omega)\;\vert\;v\vert_{\partial\omega} = 0\}$ equipped with  
\[
(u,v)_{\mH^{1}_{0}(\om)}\;:=\;\int_{\om}\nabla u\cdot\nabla \overline{v}\;d\bfx\;,
\quad\quad \Vert u \Vert_{\mrm{H}^{1}_{0}(\om)} \displaystyle:= \Vert \nabla u \Vert_{\mL^{2}(\om)}.
\]

\noindent 
The present article will focus on a transmission problem with a sign-changing
coefficient. Define the function $\sigma^{\delta}:\Omega\to \R$ such that 
$\sigma^{\delta} = \sigma_{\pm}$ in $\Omega_{\pm}^{\delta}$, where $\sigma_{+}>0$ and 
$\sigma_{-}<0$ are constants. We are interested in the eigenvalue problem
\begin{equation}\label{ExPb}
\begin{array}{|l}
\dsp{ \textrm{Find}\;(\lambda^{\delta},u^{\delta})\in\Cplx\times\mH^{1}_{0}(\Omega)\setminus\{0\}\mbox{ such that}\quad}\\[6pt]
\dsp{ -\div(\sigma^{\delta}\nabla u^{\delta}) = \lambda^{\delta}u^{\delta}\quad\mbox{ in }\Omega.}
\end{array}
\end{equation}
This problem also writes
\begin{equation}\label{pb trans form}
\begin{array}{|rcll}
\multicolumn{4}{|l}{\dsp{ \textrm{Find}\;(\lambda^{\delta},u_+^{\delta},u_-^{\delta})\in\Cplx\times\mH^{1}(\Omega^{\delta}_+)\times\mH^{1}(\Omega^{\delta}_-),\mbox{ with }(u_+^{\delta},u_-^{\delta})\ne(0,0), \mbox{ such that}\quad}}\\[6pt]
- \sigma^{\delta}_{\pm}\Delta u_{\pm}^{\delta} & = &  \lambda^{\delta}u^{\delta} & \quad\textrm{in}\;\;\Omega^{\delta}_{\pm}\\[4pt]
u_+^{\delta}-u_-^{\delta} & = & 0 & \quad\textrm{on}\;\;\Sigma^{\delta}\setminus\mathscr{S}\\[4pt]
\sigma^{\delta}_+\partial_{\boldsymbol{n}} u_+^{\delta}-\sigma^{\delta}_-\partial_{\boldsymbol{n}}u_-^{\delta} & = & 0 & \quad\textrm{on}\;\;\Sigma^{\delta}\setminus\mathscr{S}\\[4pt]
u_{\pm}^{\delta}  & = & 0 & \quad\textrm{on}\;\;\partial\Om^{\delta}_{\pm}\cap\partial\Om.
\end{array}
\end{equation}
As usual, the problem above can be reformulated in terms of operators. Consider
the unbounded operator $\mA^{\delta}:D(\mA^{\delta})\to \mL^{2}(\Omega)$ defined by 
\begin{equation}\label{ExOp}
\begin{array}{|l}
\mA^{\delta}\, v \;=\; -\mrm{div}(\sigma^{\delta}\nabla v)\\[6pt]
D(\mA^{\delta})\;:=\{v\in\mH^{1}_{0}(\Omega)\;\vert\; \mrm{div}(\sigma^{\delta}\nabla v)\in\mL^{2}(\Omega)\}.
\end{array}
\end{equation}
On $\Sigma^{\delta}\setminus\mathscr{S}$, the elements of $D(\mA^{\delta})$ satisfy the transmission 
conditions of (\ref{pb trans form}). Since $\Sigma^{\delta}$ is smooth and intersect $\partial\Omega$ with right angles, we have the following proposition (see \cite[Thm.\,1]{BoDR99} and \cite{CoSt85,DaTe97,Ramd99}).
\begin{proposition}
Assume that the contrast $\kappa_{\sigma} = \sigma_{-}/\sigma_{+}$ satisfies $\kappa_{\sigma}\neq -1$. 
Then for any $\delta>0$, the operator $\mA^{\delta}$ is densely defined, closed, self-adjoint 
and admits compact resolvent.
\end{proposition}

\noindent 
The previous result allows to study, for a fixed $\delta>0$, the spectrum of $\mA^{\delta}$. That 
this spectrum is not semi-bounded is a striking and challenging feature that will make the analysis 
more involved in the remaining of the present article.

\begin{proposition}\label{propoSptDesc}
Assume that the contrast $\kappa_{\sigma}$ satisfies $\kappa_{\sigma}\neq -1$. Then the spectrum of $\mA^{\delta}$, denoted  $\mathfrak{S}(\mA^{\delta})$, consists of two sequences, one nonnegative and one negative, of real eigenvalues of finite multiplicity: 
\[
\dots \lambda_{-m}^{\delta}\leq \dots\leq 
\lambda_{-1}^{\delta}< 0 \leq \lambda_{0}^{\delta}\leq \lambda_{1}^{\delta}\leq \dots 
\leq \lambda_{m}^{\delta} \dots\ .
\]
Moreover, there hold $\ \inf \mathfrak{S}(\mA^{\delta}) = -\infty\ $ and $\ \sup \mathfrak{S}(\mA^{\delta}) = +\infty$.
\end{proposition}
\begin{proof} The operator $\mA^{\delta}$ is self-adjoint. This implies 
that $\spec(\mA^{\delta})\subset \R$. Since it has compact resolvent,
it is a direct application of \cite[Chap.\,III, Thm.\,6.29]{Kato95} that 
$\spec(\mA^{\delta})$ consists of isolated eigenvalues with finite multiplicities. 
Let us show that $\inf \spec(\mA^{\delta}) = -\infty$. According to 
\cite[Cor.\,4.1.5]{BiSo87}, it suffices to exhibit a sequence $(\xi_{m})_{m}$ of elements of $D(\mA^{\delta})$  
which satisfies $\lim_{m\to+\infty}(\mA^{\delta}\xi_{m},\xi_{m})_{\mL^{2}(\Omega)}= -\infty$ 
and $\Vert \xi_{m}\Vert_{\mL^{2}(\Omega)} = 1$. We proceed as in \cite[Prop.\,4.1]{BoRa02}.\\

\noindent 
Let us introduce the function $\xi'$ such that $\xi'(\bfx)=\exp(-1/(1-\vert\bfx\vert^{2}))$ for $\vert \bfx\vert<1$ and 
$\xi'(\bfx)=0 $ for $\vert \bfx\vert\geq 1$. One can prove that $\xi'\in \mathscr{C}_{0}^{\infty}(\mathbb{R}^{2})$. Define $\xi := \xi'/\Vert \xi'\Vert_{\mrm{L}^{2}(\Omega)}$. Now, take any point 
$\bfx_{0}\in \Omega_{-}^{\delta}$ so that $\sigma(\bfx) = \sigma_{-}<0$ in a neighbourhood of 
$\bfx_{0}$. Set $\xi_{m}(\bfx) := m\,\xi(m(\bfx - \bfx_{0}))$. For $m$ large enough, we have 
$\mrm{supp}(\xi_{m})\subset \Omega_{-}^{\delta}$. Elementary calculus shows that 
$\Vert \xi_{m}\Vert_{\mrm{L}^{2}(\Omega)} = 1$ and $\Vert \nabla \xi_{m}\Vert_{\mrm{L}^{2}(\Omega)}^{2} = 
m^{2}\Vert \nabla\xi\Vert_{\mrm{L}^{2}(\Omega)}^{2}\to +\infty$ for $m\to +\infty$. As a consequence, there holds 
\[
(\mA^{\delta}\xi_{m} ,\xi_{m})_{\mL^{2}(\Omega)} = \int_{\Omega}\sigma\vert \nabla \xi_{m}
\vert^{2}d\bfx = -m^{2}\,\vert\sigma_{-}\vert\,\Vert\nabla \xi\Vert_{\mrm{L}^{2}(\Omega)  }^{2}
\mathop{\longrightarrow}_{m\to +\infty} - \infty.
\]
This proves that $\inf \spec(\mA^{\delta}) = -\infty$. We establish similarly that 
$\sup \spec(\mA^{\delta}) = +\infty$ by choosing $\bfx_{0}\in \Omega_{+}^{\delta}$.
Finally, we may assume that the eigenvalues are indexed in increasing order, 
considering a renumbering if necessary. This concludes the proof.\end{proof}

\noindent In the present paper, our goal is to study the behaviour of the spectrum $\spec(\mA^{\delta})$ as $\delta \to 0$. 
We will use asymptotic analysis, providing error estimates. 

\subsection{Problematic}
In order to explain the underlying difficulty of this asymptotic analysis, using the Riesz representation theorem, we define the continuous linear operator $\mathfrak{L}^{\delta}:\mH_{0}^{1}(\Omega)\to \mH^{-1}(\Omega)$ such that 
\begin{equation}\label{DefContOp}
\langle \mathfrak{L}^{\delta} u,v\rangle_{\Omega} = (\sigma^{\delta} \nabla u,\nabla v)_{\mL^{2}(\Omega)},\qquad 
\forall u,v\in\mH_{0}^{1}(\Omega).
\end{equation}
In the above definition, $\langle\cdot,\cdot\rangle_{\Omega}$ refers to the duality pairing between $\mH^{-1}(\Omega)$ and $\mH^{1}_{0}(\Omega)$. As it is known from \cite{BoCC12}, for all $\delta>0$, the operator $\mathfrak{L}^{\delta}$ 
is Fredholm of index 0 (with a possible non trivial kernel) whenever $\kappa_{\sigma}=\sigma_{-}/\sigma_{+}\neq -1$, 
as the interface $\Sigma^{\delta}$ is smooth and meets $\partial\Omega$ 
orthogonally. In \cite[Thm.\,6.2]{BoCC12}, it is also proved that, as soon as $\Sigma^{\delta}$ presents a straight 
section, in the case $\kappa_{\sigma} =\sigma_{-}/\sigma_{+}=-1$, 
the operator $\mathfrak{L}^{\delta}$ is not of Fredholm type. Actually, for this configuration, one can check that ellipticity is lost for Problem (\ref{ExPb}) (see \cite{Sche60,RoSh63} and \cite{LiMa68}). Therefore, the situation $\kappa_{\sigma} = -1$ cannot be studied with the tools we propose. We refer the reader to \cite{Ola95,Nguy15} for more details concerning this case and we discard it from now on.

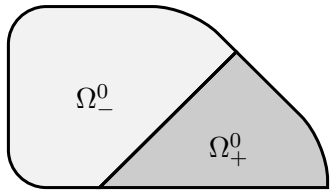
\begin{figure}[!ht]
\centering
\begin{tikzpicture}
\begin{scope}[scale=0.6]
\draw[black,fill=gray!10,line width = 0.4mm][rounded corners=5mm](-4,1) [sharp corners]--(-2,1)--(1,4)[rounded corners=5mm]--(0,5)--(-4,5)--cycle;
\draw[black,fill=gray!40,line width = 0.4mm](-2,1)[sharp corners]--(1,4)[rounded corners=5mm]--(3,2)[sharp corners]--(3,1)--cycle;
\draw[black, line width = 0.4mm](-2,1)--(1,4);
\end{scope}
\node at (0.5,1.1){$\Omega_{+}^{0}$};
\node at (-1.25,1.75){$\Omega_{-}^{0}$};
\end{tikzpicture}
\caption{Geometry for $\delta=0$.\label{GeometryDeltaZero}}
\end{figure}

\noindent Now, note that for $\delta=0$, the interface no longer meets $\partial\Om$ perpendicularly. As shown in \cite{BoCC12} and as mentioned in the introduction, there exist values of the contrasts $\kappa_{\sigma}=\sigma_{-}/\sigma_{+}$ for which the operator $\mathfrak{L}^{0}$ 
fails to be of Fredholm type, because of the existence of two strongly oscillating singularities at the corner point $O$. More precisely, for the present geometrical configuration, $\mathfrak{L}^{0}$ is a Fredholm operator if and only if, 
$\kappa_{\sigma}\in\R^{\ast}_{-}:=(-\infty;0)$ satisfies $\kappa_{\sigma}\notin[-1;-1/3]$. Here, the value $3$ comes from the ratio of the two apertures: 
$3=(\pi-\pi/4)/(\pi/4)$.\\
\newline
When $\mathfrak{L}^{0}$ is of Fredholm type, there is no qualitative difference between 
Problem (\ref{ExPb}) for $\delta>0$, and Problem (\ref{ExPb}) for $\delta=0$. In this case, using the analysis we provide in this article (and which was introduced in \cite{MaNP81}, \cite[Chap.\,4,\,9]{MaNP00}) we can prove that the spectrum of $\mA^{\delta}$ converges to the spectrum of $\mA^{0}$ as $\delta$ tends to zero. Since this result can be obtained from the 
approach we present here, in a more classical way, we have chosen not to present it.\\
\newline
When $\mathfrak{L}^{0}$ is not of Fredholm type, there is a qualitative difference between 
Problem (\ref{ExPb}) for $\delta>0$, and Problem (\ref{ExPb}) for $\delta=0$. The purpose of the present document is to 
study such a qualitative transition. When $\kappa_{\sigma} = -1/3$, the singularities associated to the corner have a more complex structure, with a logarithmic term. In the following, we discard this limit case, and therefore (unless otherwise stated), we assume that
\begin{equation}\label{CritInt}
\kappa_{\sigma} =\sigma_{-}/\sigma_{+}\in (-1;-1/3).
\end{equation}

\section{Limit problem}\label{LimitProblem}

Since we are interested in the behaviour of the spectrum of $\mA^{\delta}$ for $\delta\to 0$,
it seems natural to consider a problem similar to (\ref{ExPb}) with $\delta = 0$. 
To set such a limit problem, we have to choose a relevant functional setting. This point 
is non-trivial because, for $\delta = 0$, the interface $\Sigma^{\delta}$ does not 
intersect $\partial\Omega$ perpendicularly anymore, which prevents the limit problem 
from admitting Fredholm property in a standard Sobolev setting. This is our motivation for introducing a slightly 
different functional setting, based on weighted Sobolev (Kondratiev) spaces, that will be better suited to the present situation.

\subsection{Adapted functional setting}\label{StudySing}
The description of functional spaces adapted to this limit problem was one of the outcomes 
of \cite{BoCC13}. We dedicate this subsection to recalling results already established in 
the latter article. These results will be usefull for the analysis of the present article.\\

\noindent According to Kondratiev's theory, we need first to describe the singularities associated to the corner point $O$. Once singularities at $O$ have been computed, all the results become a consequence of the general theory of \cite{Kond67,MaPl77} (see also \cite{NaPl94,KoMR97}).  Singularities are functions of separate variables in polar coordinates which satisfy the homogeneous problem in the infinite corner. Define the function $\sigma^{0}$ by $\sigma^{0} = \sigma_{\pm}$ in $\Omega^{0}_{\pm}$.  
According to \S 4.1 in \cite{BoCC13}, the problem of finding couples 
$(\lambda,\varphi)\in \Cplx\times \mH^{1}_{0}(0;\pi)$ such that $\div(\sigma^{0}\nabla (r^{\lambda}\varphi(\theta))) = 0$ 
in $\Omega$ has non-trivial solutions only for $\lambda$ belonging to the set of singular exponents $\Lambda$ with 
\begin{equation}\label{SingExp}
\begin{array}{l}
\dsp{ \Lambda\;:=\; \big(\,2\,\Z\setminus\{0\}\,\big)
\cup\{\,i\mu+4\,\Z\,\}\cup\{\,-i\mu+
4\,\Z\,\} },\\[10pt]

\dsp{ \mu := -\frac{2}{\pi}\,\ln\Big\lbr\;\;
\frac{1}{2}\frac{\sigma_{+}-\sigma_{-} }{\sigma_{+}+\sigma_{-}}
+i\sqrt{1-\Big( \frac{1}{2}\frac{\sigma_{+}-\sigma_{-} }{\sigma_{+}
+\sigma_{-}} \Big)^{2}}\;\;\Big\rbr} .
\end{array}
\end{equation}
In the case where $\sigma_{-}/\sigma_{+}\in(-1;-1/3)$, 
we have $\mu\in(0;+\infty)$, so that the set $\Lambda$
contains only two elements in the strip $\vert\Re e\,\lambda\vert<2$, 
namely $\pm i\mu$. For $\lambda = \pm i\mu$, the space of functions
$\varphi\in \mH^{1}_{0}(0;\pi)$ such that $\div(\sigma^{0}\nabla 
(r^{\pm i\mu}\varphi(\theta)) = 0$ is one dimensional. It is generated by some 
$\phi$, both for $+i\mu$ and $-i\mu$ (see \cite[\S 4.1]{BoCC13}), such that
\begin{equation}\label{SingularExponents2}
\phi(\theta) = c_{\phi}\frac{\sinh(\mu\theta)}{\sinh(\mu\,\pi/4)}
\quad\textrm{on}\;\;\lbr0;\pi/4 \rbr\quad \textrm{and}\quad
\phi(\theta) = c_{\phi}\frac{\sinh(\mu(\pi-\theta)\,)}{\sinh(\mu\,3\pi/4)}
\quad\textrm{on}\;\;\lbr\pi/4;\pi \rbr,
\end{equation}
$c_{\phi}$ being a constant of $\R\setminus\{0\}$.
We have $\int_{0}^{\pi}\sigma^{0}(\theta)\, \phi(\theta)^{2} 
d\theta >0$ according to \cite[Lem. A.2]{BoCC13}. 
Hence, adjusting $c_{\phi}$ 
if necessary, we can normalize $\phi$ so that 
$\mu\int_{0}^{\pi}\sigma^{0}(\theta) \phi(\theta)^{2} d\theta = 1$.

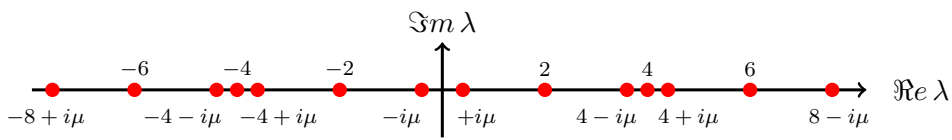
\begin{figure}[!ht]
\centering
\begin{tikzpicture}[scale=0.9]
\draw[black, line width = 0.4mm,->] (-6,0)--(6.2,0);
\draw[black, line width = 0.4mm,->] (0,-0.7)--(0,0.7);
\fill[red](-1.5,0) circle (3pt);
\fill[red](1.5,0) circle (3pt);
\fill[red](-3,0) circle (3pt);
\fill[red](3,0) circle (3pt);
\fill[red](-4.5,0) circle (3pt);
\fill[red](4.5,0) circle (3pt);
\fill[red](-0.3,0) circle (3pt);
\fill[red](-2.7,0) circle (3pt);
\fill[red](-3.3,0) circle (3pt);
\fill[red](-5.7,0) circle (3pt);
\fill[red](0.3,0) circle (3pt);
\fill[red](2.7,0) circle (3pt);
\fill[red](3.3,0) circle (3pt);
\fill[red](5.7,0) circle (3pt);
\node at (1.5,0.3){\scriptsize $2$};
\node at (3,0.3){\scriptsize $4$};
\node at (4.5,0.3){\scriptsize $6$};
\node at (-1.5,0.3){\scriptsize   $-2$};
\node at (-3,0.3){\scriptsize   $-4$};
\node at (-4.5,0.3){\scriptsize   $-6$};
\node at (0.5,-0.41){\scriptsize   $+i\mu$};
\node at (2.4,-0.41){\scriptsize   $4-i\mu$};
\node at (3.6,-0.41){\scriptsize   $4+i\mu$};
\node at (5.8,-0.41){\scriptsize   $8-i\mu$};
\node at (-0.6,-0.41){\scriptsize   $-i\mu$};
\node at (-2.4,-0.41){\scriptsize   $-4+i\mu$};
\node at (-3.8,-0.41){\scriptsize   $-4-i\mu$};
\node at (-5.8,-0.41){\scriptsize   $-8+i\mu$};
\node at (7,0){$\Re e\,\lambda$};
\node at (0,1){$\Im m\,\lambda$};
\end{tikzpicture}
\caption{Set $\Lambda$ for $\kappa_{\sigma} = -1/4$. This configuration is characteristic of the case $\kappa_{\sigma}\in(-\infty;-3)\cup(-1/3;0)$.\label{Expo singu hors inter critique 1}}
\end{figure}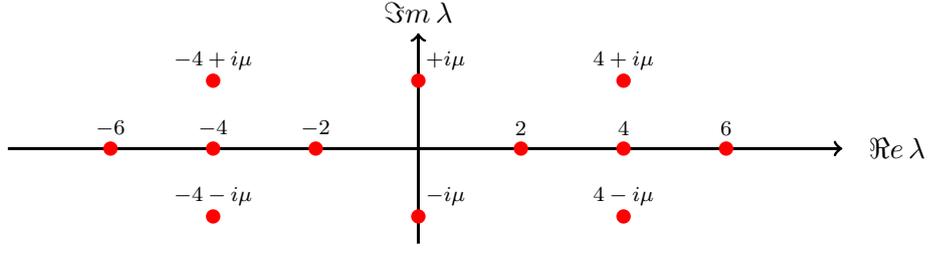
\begin{figure}[!ht]
\centering
\begin{tikzpicture}[scale=0.9]
\draw[black, line width = 0.4mm,->] (-6,0)--(6.2,0);
\draw[black, line width = 0.4mm,->] (0,-1.4)--(0,1.7);
\fill[red](0,1) circle (3pt);
\fill[red](0,-1) circle (3pt);
\fill[red](-1.5,0) circle (3pt);
\fill[red](1.5,0) circle (3pt);
\fill[red](-3,0) circle (3pt);
\fill[red](3,0) circle (3pt);
\fill[red](-3,1) circle (3pt);
\fill[red](3,1) circle (3pt);
\fill[red](-3,-1) circle (3pt);
\fill[red](3,-1) circle (3pt);
\fill[red](-4.5,0) circle (3pt);
\fill[red](4.5,0) circle (3pt);
\node at (1.5,0.3){\scriptsize $2$};
\node at (3,0.3){\scriptsize $4$};
\node at (4.5,0.3){\scriptsize $6$};
\node at (-1.5,0.3){\scriptsize   $-2$};
\node at (-3,0.3){\scriptsize   $-4$};
\node at (-4.5,0.3){\scriptsize   $-6$};
\node at (0.4,1.3){\scriptsize $+i\mu$};
\node at (0.4,-0.7){\scriptsize $-i\mu$};
\node at (3,1.3){\scriptsize $4+i\mu$};
\node at (3,-0.7){\scriptsize $4-i\mu$};
\node at (-3,1.3){\scriptsize $-4+i\mu$};
\node at (-3,-0.7){\scriptsize $-4-i\mu$};
\node at (-5.8,-0.41){\scriptsize   \phantom{$-8+i\mu$}};
\node at (7,0){$\Re e\,\lambda$};
\node at (0,2){$\Im m\,\lambda$};
\end{tikzpicture}
\caption{Set $\Lambda$ for $\kappa_{\sigma} = -1/2$. This configuration is characteristic of the case $\kappa_{\sigma}\in(-1;-1/3)$.}
\end{figure}

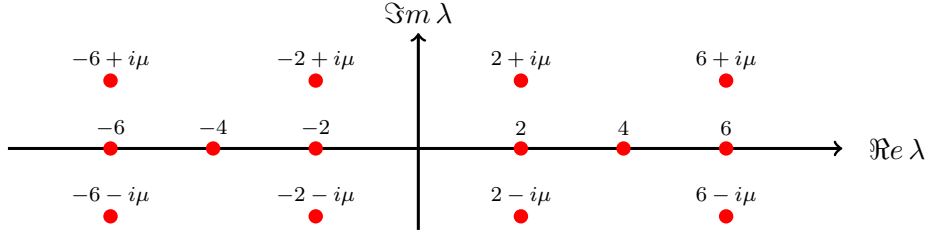
\begin{figure}[!ht]
\centering
\begin{tikzpicture}[scale=0.9]
\draw[black, line width = 0.4mm,->] (-6,0)--(6.2,0);
\draw[black, line width = 0.4mm,->] (0,-1.2)--(0,1.7);
\fill[red](-3,0) circle (3pt);
\fill[red](3,0) circle (3pt);
\fill[red](-1.5,1) circle (3pt);
\fill[red](1.5,1) circle (3pt);
\fill[red](-1.5,-1) circle (3pt);
\fill[red](1.5,-1) circle (3pt);
\fill[red](-1.5,0) circle (3pt);
\fill[red](1.5,0) circle (3pt);
\fill[red](-4.5,1) circle (3pt);
\fill[red](4.5,1) circle (3pt);
\fill[red](-4.5,-1) circle (3pt);
\fill[red](4.5,-1) circle (3pt);
\fill[red](-4.5,0) circle (3pt);
\fill[red](4.5,0) circle (3pt);
\node at (1.5,0.3){\scriptsize $2$};
\node at (3,0.3){\scriptsize $4$};
\node at (4.5,0.3){\scriptsize $6$};
\node at (-1.5,0.3){\scriptsize   $-2$};
\node at (-3,0.3){\scriptsize   $-4$};
\node at (-4.5,0.3){\scriptsize   $-6$};
\node at (1.5,1.3){\scriptsize $2+i\mu$};
\node at (-1.5,1.3){\scriptsize $-2+i\mu$};
\node at (4.5,1.3){\scriptsize $6+i\mu$};
\node at (-4.5,1.3){\scriptsize $-6+i\mu$};
\node at (1.5,-0.7){\scriptsize $2-i\mu$};
\node at (-1.5,-0.7){\scriptsize $-2-i\mu$};
\node at (4.5,-0.7){\scriptsize $6-i\mu$};
\node at (-4.5,-0.7){\scriptsize $-6-i\mu$};
\node at (-5.8,-0.41){\scriptsize   \phantom{$-8+i\mu$}};
\node at (7,0){$\Re e\,\lambda$};
\node at (0,2){$\Im m\,\lambda$};
\end{tikzpicture}
\caption{Set $\Lambda$ for $\kappa_{\sigma} = -2$. This configuration is characteristic of the case $\kappa_{\sigma}\in(-3,-1)$. \label{Expo singu hors inter critique 2}}
\end{figure}

\begin{remark}\label{PropertiesSingu}
$\bullet$ For $\kappa_{\sigma}=-1/3$, one has $0\in\Lambda$. The singularities associated with the singular exponent $0$ are $(r,\theta)\mapsto c\,\varphi(\theta)$ and $(r,\theta)\mapsto c\,\ln r\,\varphi(\theta)$, where $c$ is a constant, $\varphi(\theta)=\theta$ on $\lbr0;\pi/4 \rbr$ and $\varphi(\theta)=(\pi-\theta)/3$ on $\lbr\pi/4;\pi \rbr$. As previously announced, we do not study this limit case here.\\
$\bullet$ For $\kappa_{\sigma}\in(-1;-1/3)$ such that $\kappa_{\sigma}\to -1^{+}$, there holds $\mu\to+\infty$.\\
$\bullet$ Finally, for $\kappa_{\sigma}<0$ such that $\kappa_{\sigma}\notin[-1;-1/3]$, there holds $\Lambda\cap\{\lambda\in\Cplx\,|\,\Re e\,\lambda=0\}=\emptyset$. 
Consequently, in this case, we can prove that the limit problem for $\delta=0$ admits Fredholm property in the standard Sobolev setting $\mH^1_0(\Om)$.
\end{remark}
\begin{remark}\label{RmkAngleQconque}
Let us discuss briefly the situation where the aperture of the corner at $O$ is not $\pi/4$ but a value $\vartheta\in(0;\pi/2)$ (for the case $\vartheta\in(\pi/2;\pi)$ multiply the partial differential equation (\ref{ExPb}) by ``$-$'' to exchange the roles of $\Om_+$ and $\Om_-$). In this case, the critical interval defined in (\ref{CritInt}) is not $(-1;-1/3)$ but $(-1;-\vartheta/(\pi-\vartheta))$ and the set of singular exponents $\Lambda$ can not be computed explicitly as in (\ref{SingExp}). However, as proved in \cite[Lemma 2]{BCCC15}, for all $\kappa_{\sigma}\in(-1;-\vartheta/(\pi-\vartheta))$, we have $\Lambda\cap\{\lambda\in\Cplx\,|\,\Re e\,\lambda=0\}=\{\pm i\mu\}$ for some $\mu>0$ depending on $\kappa_{\sigma}$. Due to this property, phenomena analogous to the ones presented in this work in the case $\kappa_{\sigma}\in(-1;-1/3)$ appear. Mathematically, they can be explained following exactly the analysis developed below.  
\end{remark}

\quad\\
Let $\mathscr{C}^{\infty}_{0}(\overline{\Omega}\setminus\{O\})$ refer to the set of infinitely differentiable functions supported in $\overline{\Omega}\setminus\{O\}$. For $\beta\in \R$ and $k\geq 0$, we define 
the Kondratiev space $\mV^{k}_{\beta}(\Omega)$ as the completion of $\mathscr{C}^{\infty}_0(\overline{\Omega}\setminus\{O\})$
for the norm 
\[
\Vert v\Vert_{\mV^{k}_{\beta}(\Omega)} := \Big(\sum_{|\alpha|\le k}\int_{\Omega}
r^{2(\beta+\vert\alpha\vert-k)}\vert\partial^{\alpha}_{\bfx} v\vert^{2}\;d\bfx\;\Big)^{1/2}.
\]
Observe that $\mV^{0}_{0}(\Omega) = \mL^{2}(\Omega)$. To take into account the Dirichlet boundary condition on $\partial\Om$, for $\beta\in\R$, we introduce the space
\begin{equation}\label{defWeightedSobo}
\mathring{\mV}^{1}_{\beta}(\Omega):=\left\{ v\in 
\mV^{1}_{\beta}(\Omega)\;\vert\;v = 0\mbox{ on }\partial\Omega\setminus\{O\}\right\}.
\end{equation}
One can check that for all $\beta\in\R$, $\mathring{\mV}^{1}_{\beta}(\Omega)$ is equal to the completion of $\mathscr{C}^{\infty}_0(\Omega)$ for the norm $\Vert\cdot\Vert_{\mV^{1}_{\beta}(\Omega)}$. Moreover, using a Poincar\'e inequality on the arc $(0;\pi)$, we can prove the estimate $\|r^{-1}v\|_{\Om}\le c\,\|\nabla v\|_{\Om}$ for all $v\in\mathring{\mV}^{1}_{0}(\Omega)$ (see \S 1.3.1 \cite[Vol.\,1]{MaNP00}). This allows to conclude that $\mH^1_0(\Om)=\mathring{\mV}^{1}_{0}(\Omega)$. The norm in the dual space to $\mathring{\mV}^{1}_{\beta}(\Omega)$ 
is the intrinsic norm
\begin{equation}\label{DualNorm}
\Vert g\Vert_{\mathring{\mV}^{1}_{\beta}(\Omega)^{*}} := \sup_{v\in \mathring{\mV}^{1}_{\beta}(\Omega)\setminus\{0\}}
\frac{\vert \langle g,v\rangle_{\Omega}\vert}{\Vert v\Vert_{\mV^{1}_{\beta}(\Omega)}}\,,
\end{equation}
where $\langle\cdot,\cdot\rangle_{\Omega}$ refers to the duality pairing between 
$\mathring{\mV}^{1}_{\beta}(\Omega)^{*}$ and $\mathring{\mV}^{1}_{\beta}(\Omega)$. Although we adopt 
the same notation for the pairing between $\mH^{-1}(\Omega)$ and 
$\mH^{1}_{0}(\Omega)$, this will not bring further confusion.\\
\newline
For $\beta\in\R$, define the bounded 
linear operator $\mathfrak{L}_{\beta}: \mathring{\mV}_{\beta}^{1}(\Omega)\to \mathring{\mV}_{-\beta}^{1}(\Omega)^{*}$ such that 
\begin{equation}\label{defOpLbeta}
\langle\mathfrak{L}_{\beta} u,v\rangle_{\Omega} = (\sigma^0\nabla u,\nabla v)_{\mL^{2}(\Omega)},\qquad\forall u\in \mathring{\mV}_{\beta}^{1}(\Omega),\ v\in \mathring{\mV}_{-\beta}^{1}(\Omega).
\end{equation}
Note in particular that $\mathfrak{L}_{0}=\mathfrak{L}^{0}$ where $\mathfrak{L}^{0}$ has been introduced in (\ref{DefContOp}). The following proposition is a consequence of Kondratiev's theory (see 
\cite[Thm.\,4.1]{BoCC13}, in particular Estimate (5.15) of the proof of \cite[Thm.\,4.1]{BoCC13}, for the details).

\begin{proposition}\label{FredholmnessFarField}
Assume that $\kappa_{\sigma} = \sigma_{-}/\sigma_{+}\neq -1$. Then, the operator 
$\mathfrak{L}_{\beta}: \mathring{\mV}_{\beta}^{1}(\Omega)\to \mathring{\mV}_{-\beta}^{1}(\Omega)^{*}$ is of Fredholm type if and only if no element $\lambda\in \Lambda$ satisfies $\Re e\,\lambda = \beta$. Moreover, if no element $\lambda\in \Lambda$ satisfies $\Re e\,\lambda = \beta$, there holds the estimate
\begin{equation}\label{estimPropositionFFF}
\|u\|_{\mV_{\beta}^{1}(\Omega)} \le C\,(\|\mathfrak{L}_{\beta} u\|_{\mathring{\mV}_{-\beta}^{1}(\Omega)^{*}}+\|u\|_{\mL^{2}(\Omega\setminus\overline{\mrm{D}(O,1)})}),\qquad\forall u\in\mathring{\mV}_{\beta}^{1}(\Omega).
\end{equation}
Here, $C>0$ is a constant independent of $u\in\mathring{\mV}_{\beta}^{1}(\Omega)$.
\end{proposition}

\noindent 
Note that, according to (\ref{SingExp}),  for $\kappa_{\sigma}\in(-1;-1/3)$ the set 
$\Lambda$ contains two purely imaginary 
singular exponents $\pm i\mu$. Hence the proposition above shows that 
$\mathfrak{L}_{0}: \mathring{\mV}_{0}^{1}(\Omega) = \mH^{1}_{0}(\Omega)\to \mH^{-1}(\Omega)$ is 
not of Fredholm type: this confirms that the standard Sobolev setting is not adapted 
to the limit geometry $\Omega_{\pm}^{0}$ described in Figure \ref{GeometryDeltaZero}. On the other hand, for $\kappa_{\sigma}\in(-1;-1/3)$, there is no element  $\lambda\in\Lambda$ 
satisfying $0<\vert \Re e\,\lambda\vert <2$. As a consequence, for all $\beta\in(-2;0)\cup(0;2)$, the operator $\mathfrak{L}_{\beta}$ is of Fredholm type whenever $\kappa_{\sigma}\in(-1;-1/3)$. Let us define
\begin{equation}\label{DefSing}
\sgpm(r,\theta)\;:=\;\psi(r)r^{\pm i\mu}\phi(\theta),
\end{equation}
where $\psi\in\mathscr{C}^{\infty}(\R,[0;1])$ is the cut-off function introduced right before \S\ref{paragraph geom}. In particular, we recall it satisfies $\psi(r) = 1$ 
for $r\leq 1$ and  $\psi(r) = 0$ for $r\geq 2$. Observe that in a neighbourhood of $O$, $\sgpm$ has the same behaviour as the singularity associated with the purely imaginary singular exponent $\pm i\mu$. Moreover, the multiplication by the cut-off function $\psi$ ensures that $\sgpm=0$ on $\partial\Om\setminus\{O\}$. A direct computation shows that $\sgpm \in \mathring{\mV}^{1}_{\beta}(\Omega)$ for all $\beta>0$. Application of 
Kondratiev's calculus to the operators $\mathfrak{L}_{\beta}$ yields the following decomposition result (\cite[Thm.\,5.2]{BoCC13}):
\begin{proposition}\label{WeightExp}
Assume that $\kappa_{\sigma}\in(-1;-1/3)$, so that the only elements $\lambda\in\Lambda$
satisfying $-2<\Re e\,\lambda < +2$ are $\lambda = \pm i\mu$. Let $\beta\in(0;2)$ be given and $v$ be an element of $\mathring{\mV}^{1}_{\beta}(\Omega)$ such that $\mathfrak{L}_{\beta}v\in 
\mathring{\mV}^{1}_{\beta}(\Omega)^{*}$ (the important point here is that $\mathring{\mV}^{1}_{\beta}(\Omega)^{*}$ is embedded in $\mathring{\mV}^{1}_{-\beta}(\Omega)^{*}$ since $-\beta<\beta$). Then, there holds the following representation
\[
v = c_{+}\sgp+c_{-}\sgm+\tilde{v},\qquad\mbox{ with }c_{\pm}\in \Cplx,\,\tilde{v}\in\mathring{\mV}^{1}_{-\beta}(\Omega).
\] 
\end{proposition}

\subsection{Limit self-adjoint operators}
We are interested in finding an operator that would be the ``limit'' of $\mA^{\delta}$ as $\delta\to 0$. A first candidate 
may perhaps consist in the unbounded operator $\mathfrak{A}:D(\mathfrak{A})
\to \mL^{2}(\Omega)$ defined by
\[
\begin{array}{|l}
\mathfrak{A} v\;:=\; -\mrm{div}(\sigma^{0}\nabla v)\\[6pt]
D(\mathfrak{A}) := \{ v\in \mH^{1}_{0}(\Omega)\;\vert\; \mrm{div}(\sigma^{0}\nabla v)
\in\mL^{2}(\Omega)\}.
\end{array}
\]
Using the weighted Sobolev spaces introduced in (\ref{defWeightedSobo}), the regularity at $O$ of the elements of $D(\mathfrak{A})$ can be determined more precisely.
\begin{proposition}\label{propoDcpDA}
There holds $D(\mathfrak{A}) := \{ v\in \mathring{\mV}^{1}_{-1}(\Omega)\;\vert\; \mrm{div}(\sigma^{0}\nabla v)\in\mL^{2}(\Omega)\}$. Moreover, there exists a constant $C>0$ such that
\begin{equation}\label{estimAPdomainReg}
\|v\|_{\mV_{-1}^{1}(\Omega)} \le C\,(\|\mathfrak{A}v\|_{\mL^{2}(\Omega)}+\|v\|_{\mL^{2}(\Omega\setminus\overline{\mrm{D}(O,1)})}),\qquad\forall v\in D(\mathfrak{A}).
\end{equation}
\end{proposition}
\begin{proof} Let $v\in \mH^{1}_{0}(\Omega)=\mathring{\mV}^{1}_{0}(\Omega)\subset\mathring{\mV}^{1}_{1}(\Omega)$ be a function of $D(\mathfrak{A})$. We denote $f:=\mathfrak{A}v\in\mL^{2}(\Omega)$. Observing that $|(f,w)_{\mrm{L}^2(\Om)}|\le \|f\|_{\mL^{2}(\Omega)}\|w\|_{\mL^{2}(\Omega)}\le \|f\|_{\mL^{2}(\Omega)}\|w\|_{\mV^{1}_{1}(\Omega)}$, we deduce that 
$f\in\mathring{\mV}^{1}_{1}(\Omega)^{*}$ with $\|f\|_{\mathring{\mV}^{1}_{1}(\Omega)^{*}}\le \|f\|_{\mL^{2}(\Omega)}$. Therefore, the equation $\mathfrak{A}v=f$ also writes $\mathfrak{L}_{+1}v=f$. The application of Proposition \ref{WeightExp} for $\beta=1$ implies 
that there exist $c_{\pm}\in \Cplx$ and $\tilde{v}\in\mathring{\mV}^{1}_{-1}(\Omega)$ such that $v = c_{+}\sgp+c_{-}\sgm+\tilde{v}$. Since $\sgp$ and $\sgm$ are two linearly independent elements of 
$\mathring{\mV}^{1}_{1}(\Omega)\setminus\mathring{\mV}^{1}_{0}(\Omega)$ and since $v\in\mathring{\mV}^{1}_{0}(\Omega)$, we deduce successively that $c_{\pm}=0$ and $v\in\mathring{\mV}^{1}_{-1}(\Omega)$. Finally, we obtain (\ref{estimAPdomainReg}) applying (\ref{estimPropositionFFF}) for $\beta=-1$. \end{proof}

\noindent The operator $\mathfrak{A}$ is symmetric. From the point of view 
of spectral analysis, it would be desirable to determine whether it is self-adjoint. The following result is established in \cite[Chap.\,7]{Ramd99}.

\begin{proposition}\label{Adjoint}
The domain of the operator $\mathfrak{A}^{*}$ is given by $D(\mathfrak{A}^{*}) = \mrm{span}\{\sgp,\sgm\}\oplus D(\mathfrak{A})$. 
\end{proposition}

\noindent 
Since $D(\mathfrak{A}^{*}) \ne D(\mathfrak{A})$, this shows that $\mathfrak{A}$ cannot be self-adjoint. Actually, it is possible to completely describe all the self-adjoint extensions of this operator (for the general theory, see e.g. \cite[Sect.\,X.1]{ReSi75}). To do so, we begin by introducing a classical tool (see \cite[Chap.\,5]{NaPl94}), namely the sesquilinear form $q(\cdot,\cdot)$ such that 
\begin{equation}\label{Simplectic}
q(u,v) := (\mathfrak{A}^{*}u,v)_{\mL^{2}(\Omega)} - (u,\mathfrak{A}^{*}v)_{\mL^{2}(\Omega)},\quad 
\forall u,v\in D(\mathfrak{A}^{*}).
\end{equation}
The form $q(\cdot,\cdot)$ is said to be anti-hermitian, or equivalently, symplectic, because it verifies $q(u,v)=-\overline{q(v,u)}$ for all $u,v\in D(\mathfrak{A}^{*})$.
\begin{proposition}\label{PropositionPrtiesFormeq}
The anti-hermitian form $q(\cdot,\cdot)$ satisfies the following properties:\\[3pt]
\begin{tabular}{ll}
i) & $q(u,v)=q(v,u)=0$, for all $u\in D(\mathfrak{A}^{\ast})$, $v\in D(\mathfrak{A})$;\\
ii) & $q(\sgp,\sgm)=q(\sgm,\sgp)=0$;\\
iii) & $q(\sgp,\sgp)=-q(\sgm,\sgm)\ne0$.
\end{tabular}
\end{proposition}
\begin{proof} Item $i)$ comes from the very definition of $\mathfrak{A}^{*}$, and the fact that $\mathfrak{A}\subset \mathfrak{A}^{*}$. One establishes $ii)$ noticing that $\overline{\sgp} = \sgm$. Finally, one observes that, necessarily, there holds $q(\sgp,\sgp)=-q(\sgm,\sgm)\neq 0$, 
otherwise we would have $q(w,\tilde{w}) = 0$, for all  $w,\tilde{w}\in D(\mathfrak{A}^{*})$, 
\textit{i.e.} $\mathfrak{A}^{*}$ would be self-adjoint which is impossible according to 
Proposition \ref{Adjoint}.
\end{proof}

\noindent Now, we are able to demonstrate the following result.
\begin{proposition}\label{DescribingSAE}
The self-adjoint extensions of  $\mathfrak{A}$ are the unbounded operators $\mathfrak{A}(\tau)$, $\tau\in\R$, such that $\mathfrak{A}(\tau):D(\mathfrak{A}(\tau))\to\mL^{2}(\Omega)$ is defined by 
\begin{equation}\label{definitionOperateurAtau}
\begin{array}{|l}
\mathfrak{A}(\tau)v\;=\; -\mrm{div}(\sigma^{0}\nabla v)\\[6pt]
D(\mathfrak{A}(\tau))  = \mrm{span}\{
\sgp + e^{i\tau}\sgm \}\oplus D(\mathfrak{A}).
\end{array}
\end{equation}
\end{proposition}
\begin{proof} Suppose first that  $\mathscr{A}:D(\mathscr{A})\to \mL^{2}(\Omega)$ is a self-adjoint extension 
of $\mathfrak{A}$. We have $\mathfrak{A}\subset \mathscr{A}\subset \mathfrak{A}^{*}$ and the 
inclusions are strict, otherwise $\mathscr{A} = \mathfrak{A}$ or $\mathscr{A} = \mathfrak{A}^{*}$  
(since $\mrm{dim}(D(\mathfrak{A}^{*})/D(\mathfrak{A})) = 2$) and then $\mathscr{A}$ would not 
be self-adjoint. Due to Proposition \ref{Adjoint}, there exist fixed constants $\alpha_{\pm}\in \Cplx$ 
such that $D(\mathscr{A})  = \mrm{span}\{\alpha_{+}\sgp + \alpha_{-}\sgm\}\oplus D(\mathfrak{A})$. Set for a moment $\sg_{\alpha} = \alpha_{+}\sgp + \alpha_{-}\sgm$. If the operator $\mathscr{A}$ is self-adjoint, it is in particular symmetric. Since $\mathscr{A}\subset \mathfrak{A}^{*}$, a necessary and sufficient condition for $\mathscr{A}$ to be symmetric is that $q(u,v) = 0$, for all $u,v\in 
D(\mathscr{A})$. Take two arbitrary elements 
$u,v\in D(\mathscr{A})$, so that there exist $\tilde{u}, \tilde{v}\in D(\mathfrak{A})$, and 
$c_{u},c_{v}\in\Cplx$ such that $u = c_{u}\,\sg_{\alpha} + \tilde{u}$ and $v = c_{v}\,\sg_{\alpha} + \tilde{v}$. The symmetry of $\mathscr{A}$ and Proposition \ref{PropositionPrtiesFormeq} impose
\[
0=q(u,v) = c_{u}\,\overline{c}_{v}\;q(\sg_{\alpha},\sg_{\alpha}),\qquad 
\forall u,v\in D(\mathscr{A}).
\]
This is true if and only if $0 = q(\sg_{\alpha},\sg_{\alpha}) = q(\sgp, \sgp)(\vert \alpha_{+}\vert^{2} - \vert \alpha_{-}\vert^{2})\Leftrightarrow\vert \alpha_{+}\vert = \vert \alpha_{-}\vert$ ($\neq 0$). To briefly sum up, 
if $\mathscr{A}$ is a self-adjoint extension of $\mathfrak{A}$, then we have
\begin{equation}\label{CondSelfAdjoint}
\begin{array}{|l}
\mathscr{A} v = -\mrm{div}(\sigma^{0}\nabla v)\\[6pt]
D(\mathscr{A}) = \mrm{span}\{\alpha_{+}\sgp +\alpha_{-}\sgm\}\oplus D(\mathfrak{A}),\qquad\mbox{ where }\vert \alpha_{+}\vert =  \vert \alpha_{-}\vert\;\neq 0.
\end{array}
\end{equation}
Now, let us consider an operator $\mathscr{A}$ which satisfies (\ref{CondSelfAdjoint}). Let us prove that 
$\mathscr{A}$ is a self-adjoint extension of $\mathfrak{A}$. What precedes shows that 
the condition $\vert\alpha_{-}\vert = \vert\alpha_{+}\vert$ implies symmetry of $\mathscr{A}$, so we 
only have to establish that $D(\mathscr{A}^{*}) = D(\mathscr{A})$. Take $u\in D(\mathscr{A}^{*})$. 
Since $D(\mathscr{A}^{*})\subset D(\mathfrak{A}^{*})$, there exist $\tilde{u}\in D(\mathfrak{A})$ and 
$c_{\pm}\in \Cplx$ such that $u = c_{+}\sgp + c_{-}\sgm + \tilde{u}$. The symmetry of $\mathscr{A}$ allows to write
\[
0 = q(u,\alpha_{+}\sgp+\alpha_{-}\sgm) = q(\sgp,\sgp)(c_{+}\overline{\alpha}_{+} - c_{-}\overline{\alpha}_{-}).
\]
We deduce that $c_{-} = c_{+}\overline{\alpha}_{+}/\overline{\alpha}_{-} = 
c_{+} \alpha_{-}/\alpha_{+}$ (remember that $\vert \alpha_{-}/\alpha_{+}\vert = 1$). We finally conclude that
$u = c_{+}/\alpha_{+}(\alpha_{+}\sgp+\alpha_{-}\sgm) + \tilde{u}\in D(\mathscr{A})$. This shows that 
(\ref{CondSelfAdjoint}) are actually necessary and sufficient conditions of self-adjointness.
It takes elementary calculus to check that this is equivalent to what is announced in the statement 
of the proposition. \end{proof}

\noindent The anti-hermitian form $q(\cdot,\cdot):
D(\mathfrak{A}^{*})^{2}\to \Cplx$ introduced in (\ref{Simplectic}) will play an important role in the sequel. We shall also
consider the linear forms $\pi_{\pm}:D(\mathfrak{A}^{*})\to \Cplx$ such that
\begin{equation}\label{defintionPipm}
\pi_{+}(v) := \frac{q(v,\sgp)}{q(\sgp,\sgp)}\;,\quad 
\pi_{-}(v) := \frac{q(v,\sgm)}{q(\sgm,\sgm)},\qquad \forall v\in D(\mathfrak{A}^{*}).
\end{equation}
Clearly, the forms $\pi_{\pm}$ are continuous with respect to the norm of the graph of $\mathfrak{A}^{*}$: 
for all $v\in D(\mathfrak{A}^{*})$, $\vert  \pi_{\pm}(v)\vert\leq C\,(\Vert v\Vert_{\mL^{2}(\Omega)}+\Vert 
\mathfrak{A}^{*}v\Vert_{\mL^{2}(\Omega)})$. With these 
maps, we have the decomposition
\begin{equation}\label{DecSgPart}
v - (\pi_{+}(v)\sgp +\pi_{-}(v)\sgm)\;\in D(\mathfrak{A})
,\qquad\forall v\in D(\mathfrak{A}^{*}).
\end{equation}
To prove (\ref{DecSgPart}), it suffices to notice that if $v=c_+\sgp+c_-\sgm+\tilde{v}$, with $c_{\pm}\in\Cplx$, $\tilde{v}\in D(\mathfrak{A})$, is an element of $D(\mathfrak{A}^{*})$, then there holds, according to Proposition \ref{PropositionPrtiesFormeq}, $q(v,\sgpm)=c_+q(\sgp,\sgpm)+c_-q(\sgm,\sgpm)+q(\tilde{v},\sgpm)=c_{\pm}q(\sgpm,\sgpm)$. The functionals $\pi_{\pm}$ can be exploited to express in a convenient manner  
a result established in \cite{BoCC13} concerning the kernels of operators $\mathfrak{L}_{\pm 1}$ (the operators $\mathfrak{L}_{\beta }$ defined in (\ref{defOpLbeta}) with $\beta=\pm1$). 

\begin{proposition}\label{KernelGen}
There exists a unique (modulo $\Ker\,\mathfrak{L}_{-1}$) element $\zeta\in\Ker\,\mathfrak{L}_{+1}\setminus \Ker\,\mathfrak{L}_{-1}$
satisfying $\pi_{-}(\zeta) = 1$. This function is such that $\Ker\,\mathfrak{L}_{+1} = 
\mrm{span}\{\zeta\}\oplus\Ker\,\mathfrak{L}_{-1}$ and there holds $\vert\pi_{+}(\zeta)\vert = 1$.
\end{proposition}
\begin{proof} First of all, by virtue of Proposition \ref{Adjoint} and Proposition \ref{WeightExp}, it is clear that $\Ker\,\mathfrak{L}_{-1}\subset D(\mathfrak{A})$ and $\Ker\,\mathfrak{L}_{+1}\subset D(\mathfrak{A}^{*})$. As shown in step 2 of the proof of \cite[Thm.\,4.4]{BoCC13}, 
necessarily we have $\Ker\,\mathfrak{L}_{+1}\neq \Ker\,\mathfrak{L}_{-1}$.
If $\zeta\in \Ker\,\mathfrak{L}_{+1}\setminus \Ker\,\mathfrak{L}_{-1}$ then $\zeta\in D(\mathfrak{A}^{*})$ 
and, applying the same calculus as in the proof of Proposition \ref{DescribingSAE} above, we have 
$0 = q(\zeta,\zeta) = q(\sgp,\sgp)(\vert\pi_{+}(\zeta)\vert^{2} - \vert\pi_{-}(\zeta)\vert^{2})$.
Since $q(\sgp,\sgp)\neq 0$, this implies $\vert\pi_{+}(\zeta)\vert = \vert\pi_{-}(\zeta)\vert$, and 
$\vert\pi_{-}(\zeta)\vert\neq 0$ since $\zeta\notin \Ker\,\mathfrak{L}_{-1}$.
Hence, dividing by $\pi_{-}(\zeta)$ if necessary, we can assume that $\pi_{-}(\zeta)=1$. In this case, there holds $\vert \pi_{+}(\zeta)\vert = 1$.

\quad\\
Now, take another element $\zeta'\in \Ker\,\mathfrak{L}_{+1}\setminus \Ker\,\mathfrak{L}_{-1}$.
We also have $\vert\pi_{+}(\zeta')\vert = \vert\pi_{-}(\zeta')\vert$. Moreover, there holds
$0 = q(\zeta',\zeta) = q(\sgp,\sgp)\big(\, \pi_{+}(\zeta')\overline{\pi_{+}(\zeta)} -  
\pi_{-}(\zeta')\overline{\pi_{-}(\zeta)}\,\big)$. We deduce 
\[
\begin{array}{l}
\dsp{ \pi_{+}(\zeta') =  \pi_{-}(\zeta')\pi_{+}(\zeta)  }
\quad\Rightarrow\quad 
\dsp{ \zeta' - \pi_{-}(\zeta')\big(\,\pi_{+}(\zeta)\sgp + 
\sgm\,\big)\;\in\;\mathring{\mV}^{1}_{-1}(\Omega) }\\[10pt]

\phantom{\dsp{ \pi_{+}(\zeta') =  \pi_{-}(\zeta')\pi_{+}(\zeta)}}\quad\Rightarrow\quad 
\dsp{ \zeta' - \pi_{-}(\zeta')\zeta\;\in\;\Ker\,\,\mathfrak{L}_{-1}. }
\end{array}
\]
This ends to prove that $\Ker\,\mathfrak{L}_{+1} = 
\mrm{span}\{\zeta\}\oplus\Ker\,\mathfrak{L}_{-1}$.\end{proof}
\subsection{Spectrum of the self-adjoint extensions}
In this section, we study the features of the spectrum $\mathfrak{S}(\mathfrak{A}(\tau))$ of the operator $\mathfrak{A}(\tau)$, $\tau\in\R$, defined in (\ref{definitionOperateurAtau}). First, we prove it admits 
the same qualitative properties as $\mathfrak{S}(\mrm{A}^{\delta})$ for $\delta>0$.

\begin{proposition}
Pick $\tau\in\R$. The operator 
$\mathfrak{A}(\tau)$ is closed, densely defined, self-adjoint and admits 
compact resolvent. Its spectrum consists of two sequences, one nonnegative and one negative, of real eigenvalues of finite multiplicity: 
\[
\dots \eta_{-m}(\tau)\leq \dots\leq 
\eta_{-1}(\tau)< 0 \leq \eta_{0}(\tau)\leq \eta_{1}(\tau)\leq \dots 
\leq \eta_{m}(\tau) \dots\ .
\]
Moreover, there hold $\ \inf \mathfrak{S}(\mathfrak{A}(\tau)) = -\infty\ $ and $\ \sup \mathfrak{S}(\mathfrak{A}(\tau)) = +\infty$.
\end{proposition}
\begin{proof} From Proposition \ref{DescribingSAE}, we know that for all $\tau\in\R$, $\mathfrak{A}(\tau)$ is a self-adjoint operator. In addition, for $z\in\Cplx\setminus\R$, using \cite[Thm.\,4.4]{BoCC13}, we can prove the estimate
\begin{equation}\label{estimationResolv}
|\pi_+(v)|+ \Vert v - \pi_{+}(v)\,(\sgp + e^{i\tau}\sgm)\Vert_{\mV_{-1}^{1}(\Omega)}\leq C\,\Vert (\mathfrak{A}(\tau)-z\mrm{Id}) v\Vert_{\mL^{2}(\Omega)},\qquad\forall v\in\mL^{2}(\Omega).
\end{equation}
Since the embedding $\mathring{\mV}_{-1}^{1}(\Omega)\subset\mV_{0}^{0}(\Omega)= \mL^{2}(\Omega)$ is compact (see \cite[Lem.\,6.2.1]{KoMR97}), (\ref{estimationResolv}) allows to prove that $\mathfrak{A}(\tau)$ has compact resolvent. The second part of the statement can be obtained working like in the proof  of Proposition \ref{propoSptDesc}.
\end{proof}

\noindent Note that, because we define $(\eta_{j}(\tau))_{j\ge0}$ and $(\eta_{j}(\tau))_{j<0}$ as ordered series, and because we impose that $\eta_{0}(\tau)$ is the smallest non negative eigenvalue, we cannot hope that the function $\tau\mapsto \eta_{j}(\tau)$ be continuous.

\quad\\
Yet there is a possibility to state useful results about the dependency of 
$\mathfrak{S}(\mathfrak{A}(\tau))$ with respect to $\tau$. Indeed there 
is continuous dependency of $\mathfrak{A}(\tau)$ with respect to $\tau$ 
in the sense of ``generalized convergence''. The later is a notion of convergence defined on the set of closed possibly unbounded operators and is considered 
in \cite[Chap.\,IV, \S 2]{Kato95}. Let us briefly recall what it consists in.
Define the gap functional 
\[
\mathfrak{d}(T,S)\:=\;\sup_{u\in D(T)\setminus\{0\}}\;
\mathop{\inf\phantom{p}}_{v\in D(S)}\frac{\Vert u-v\Vert_{\mL^{2}(\Omega)} +
\Vert Tu-Sv\Vert_{\mL^{2}(\Omega)}  }{ 
\Vert u\Vert_{\mL^{2}(\Omega)} +\Vert Tu\Vert_{\mL^{2}(\Omega)}}.
\]
A sequence of closed operators $T_{n}:D(T_{n})\to \mL^{2}(\Omega)$ is said 
to converge to a closed operator $T:D(T)\to \mL^{2}(\Omega)$ if and only if
$\lim_{n\to+\infty}\mrm{max}(\mathfrak{d}(T,T_{n}), \mathfrak{d}(T_{n},T)) = 0$.
A result of continuous dependency in this sense actually holds.

\begin{proposition}\label{propositionGeneralContinuity}
There is a constant $C>0$ such that 
$\mathfrak{d}(\mathfrak{A}(\tau_{1}),\mathfrak{A}(\tau_{2}))\leq C
\vert e^{i\tau_{1}} - e^{i\tau_{2}}\vert,$ $\forall \tau_{1},\tau_{2}\in\R$.
\end{proposition}
\begin{proof} Take two $\tau_{1},\tau_{2}\in\R$. Consider an arbitrary $u\in 
D(\mathfrak{A}(\tau_{1}))\setminus\{0\}$ such that $u = \pi_{+}(u) \,(\sgp+e^{i\tau_{1}}\sgm) + \tilde{u}$, with $\tilde{u}\in D(\mathfrak{A})$. Let us define 
$v := \pi_{+}(u) \,(\sgp+e^{i\tau_{2}}\sgm) + \tilde{u}\in D(\mathfrak{A}(\tau_{2}))$.
Straightforward calculus yields
\[
\Vert u-v\Vert_{\mL^{2}(\Omega)} + \Vert \mathfrak{A}(\tau_{1})u-\mathfrak{A}(\tau_{2})v\Vert_{\mL^{2}(\Omega)}\;=\;
\vert \pi_{+}(u) \vert\;\vert e^{i\tau_{1}} - e^{i\tau_{2}}\vert\;(\Vert \sgm\Vert_{\mL^{2}(\Omega)} +
\Vert \mathfrak{A}^{*}\sgm\Vert_{\mL^{2}(\Omega)}).
\]
According to the continuity of $\pi_{+}$, there exists $C>0$ such that 
$\vert \pi_{+}(u) \vert\leq C\,(\Vert u\Vert_{\mL^{2}(\Omega)}+ \Vert \mathfrak{A}^{*}u\Vert_{\mL^{2}(\Omega)})$. 
Since $\mathfrak{A}^{*}u = \mathfrak{A}(\tau_{1})u$ and $\Vert \sgm\Vert_{\mL^{2}(\Omega)}\neq 0$, this allows to obtain the result of Proposition \ref{propositionGeneralContinuity}.
\end{proof}
\begin{proposition}\label{BoundedEigenvalues}
For any fixed $j\in\Z$, the function $\tau\mapsto\eta_{j}(\tau)$ is bounded over $\R$.
\end{proposition}
\begin{proof} Let us consider a fixed $j\geq 0$. The proof below  can be straightforwardly adapted 
to the case $j<0$. Since the operator valued function $\tau\mapsto \mathfrak{A}(\tau)$ defined in (\ref{definitionOperateurAtau}) is $2\pi$-periodic, the map $\tau\mapsto \eta_{j}(\tau)$ is $2\pi$-periodic.
Therefore, it suffices to show that $\tau\mapsto \eta_{j}(\tau)$ is bounded over $\lbr 0;2\pi\rbr$.

Pick an arbitrary $\tau_{0}\in \lbr 0;2\pi\rbr$. Take, at least, $j+1$ eigenvalues such that $0<\eta_{n_{0}}(\tau_{0})< \dots< \eta_{n_{j}}(\tau_{0})$. 
Consider a closed smooth curve $\Upsilon_{\tau_{0}}\subset \Cplx\setminus \mathfrak{S}(\mathfrak{A}(\tau_{0}))$ dividing the complex plane in two 
connected components: a bounded region $\mathcal{O}_{\tau_{0}}$ and $\Cplx\setminus\overline{\mathcal{O}_{\tau_{0}}}$. We choose $\Upsilon_{\tau_{0}}$ so that $\mathcal{O}_{\tau_{0}}$ contains the eigenvalues $\{\eta_{n_{k}}(\tau_{0})\}_{k=0}^{j}$ and so that there holds $\mathcal{O}_{\tau_{0}}\subset\{\lambda\in\Cplx\,|\,\Re e\,\lambda>0\}$. Since $\tau\mapsto \mathfrak{A}(\tau)$ 
is continuous in the sense of generalized convergence, we can apply \cite[Chap.\,IV, Thm.\,3.16]{Kato95}. This theorem ensures that there exists $\eps_{\tau_{0}}>0$ such that for $\vert \tau -\tau_{0}
\vert< \eps_{\tau_{0}}$, the domain $\mathcal{O}_{\tau_{0}}$ contains $j+1$ strictly positive eigenvalues 
of $\mathfrak{A}(\tau)$ (counted with their multiplicities). Thus, according to the definition of $\eta_{j}(\tau)$,
we have 
\[
\forall \tau\in ( \tau_{0} - \eps_{\tau_{0}}; \tau_{0} + \eps_{\tau_{0}}),\qquad 0\leq \eta_{j}(\tau)\leq m_{\tau_{0}} = 
\max\{\;\Re e\,\lambda\;\vert\;\lambda\in\mathfrak{S}(\mathfrak{A}(\tau_{0}))\cap \mathcal{O}_{\tau_{0}} \;\} <+\infty.
\]
The set $\lbr0;2\pi\rbr$ being compact, it can be covered by a finite number of intervals of this type: 
$\lbr0;2\pi\rbr\subset \cup_{k=1}^{p}( \tau_{k} - \eps_{\tau_{k}}; \tau_{k} + \eps_{\tau_{k}})$.
In conclusion, we have $\eta_{j}(\tau)\leq \mrm{max}_{k=1\dots p} m_{\tau_{k}} <+\infty$ for all $\tau\in 
\lbr 0;2\pi\rbr$. This ends to prove that $\tau\mapsto\eta_{j}(\tau)$ is bounded over $\R$.
\end{proof}

\section{Formal asymptotics of the eigenvectors}\label{FormalAsymp}
We dedicate this section to the description of the asymptotics of the eigenvalues of the operator $\mA^{\delta}$ defined in (\ref{ExOp}) as $\delta\to 0$. In the present problem, there appears a boundary layer in the neighbourhood of the rounded corner. Therefore, we propose a far field expansion and a near field expansion that we will match. This kind of procedure has been thoroughly described in \cite{MR0416240,MR1182791} and \cite[Chap.\,2]{MaNP00} and we refer the reader to these reference books for more details. The ansatz we define here will be validated by error estimates derived in the next sections. Let us consider a fixed $m\in\Z$ and an eigenpair $(\lambda_m^{\delta},u^{\delta})\in\Cplx\times 
\mH^{1}_{0}(\Omega)\setminus\{0\}$ such that 
\begin{equation}\label{EqnsPbVp}
\mA^{\delta}u^{\delta}=-\div(\sigma^{\delta}\nabla u^{\delta}) = \lambda_m^{\delta} u^{\delta}\quad\mbox{ in }\Omega\qquad\mbox{ and }\qquad\Vert u^{\delta}\Vert_{\mL^{2}(\Omega)} = 1,\quad\forall \delta\in(0;1].
\end{equation}
We will also assume that
\[
\limsup_{\delta\to 0}\vert\lambda^{\delta}_m\vert <+\infty.
\]
Admittedly, it is not obvious that such an assumption holds true. However, it will be justified \textit{a posteriori} (see (\ref{MainEstimate})). As previously announced and as it is common in asymptotic analysis, we distinguish expansions 
far from the rounded corner, and close to the rounded corner.

\subsection{Far field expansion.} In the far field region, \textit{i.e.} far from $O$, as $\delta\to0$, 
we look for an expansion of $(\lambda^{\delta}_m,u^{\delta})$ of the form 
\begin{equation}\label{ansatzFF}
\lambda^{\delta}_m = \eta_m^{\delta}+\cdots,\qquad\qquad u^{\delta}=\breve{u}^{\delta}+\cdots,
\end{equation}
for some pair $(\eta_m^{\delta},\breve{u}^{\delta})$ which has to be determined. We should make a few comments concerning this ansatz. Observe that, in this notation, it seems that we allow some $\delta$-dependence for the first term in the asymptotic expansion of $(\lambda^{\delta}_m,u^{\delta})$. The reason comes from the results we obtained in \cite{ChCNSu}. In this paper, where we studied the source term problem associated with (\ref{ExPb}), we proved that the solution, when it is well-defined, is not stable with respect to $\delta$, and depends on the small rounding parameter even at the first order. Nevertheless, $(\eta_m^{\delta},\breve{u}^{\delta})$ will be an asymptotic expansion of $(\lambda^{\delta}_m,u^{\delta})$ in the sense that it will be the solution of some problem defined in the limit geometry (\ref{GeometryDeltaZero}) obtained taking $\delta=0$. Plugging (\ref{ansatzFF}) into (\ref{EqnsPbVp}), we see that the following equations have to hold
\begin{equation}\label{LimitPb}
-\div(\sigma^{0}\nabla \breve{u}^{\delta}) = \eta_m^{\delta}\breve{u}^{\delta}\ \mbox{ in }\Omega\qquad\mbox{ and }\qquad\breve{u}^{\delta}=0\ \mbox{ on }\partial\Omega\setminus\{O\}.
\end{equation}
Let us choose to impose that $\breve{u}^{\delta}$ be an element of $\mrm{L}^2(\Om)$. In this case, (\ref{LimitPb}) leads us to look for a $\breve{u}^{\delta}$ which belongs to $D(\mathfrak{A}^{*})$ and which satisfies the equation $\mathfrak{A}^{*}\breve{u}^{\delta} = \eta_m^{\delta}\,\breve{u}^{\delta}$. Let us recall that according to Proposition \ref{WeightExp}, we know there exist constants 
$c^{\delta}_{\pm}\in\Cplx$  such that 
\begin{equation}\label{NearFieldApprox}
\breve{u}^{\delta} -( c^{\delta}_{+}\,\sgp +c^{\delta}_{-}\,\sgm )\in D(\mathfrak{A})\subset \mathring{\mV}^{1}_{-1}(\Omega).
\end{equation}
Let us emphasize that at this point, $\eta_m^{\delta}$ and $\breve{u}^{\delta}$ are not yet completely determined (in particular the constants $c^{\delta}_{\pm}$ are not known). However, from a formal point of view, up to some remainder with respect to $\delta$, we have the following asymptotic behaviour as $r\to0$:
\begin{equation}\label{dvptFFE}
\begin{array}{lcl}
\breve{u}^{\delta}(r,\theta) & = &  c^{\delta}_{+}\,\sgp(r,\theta) +c^{\delta}_{-}\,\sgm(r,\theta) +\cdots\\[4pt]
  & = &  c^{\delta}_{+}\,r^{+i\mu}\phi(\theta) +c^{\delta}_{-}\,r^{-i\mu}\phi(\theta) +\cdots \ .
\end{array}
\end{equation}

\subsection{Near field expansion.} In the near field region, \textit{i.e.} close to $O$, we consider the change of coordinates $\bfxi=\bfx/\delta$ ($\bfxi$ 
is the fast variable), we set $U^{\delta}(\bfxi)=u^{\delta}(\delta\bfxi)$, and we look for an expansion of $U^{\delta}$ of the form 
\[
U^{\delta} = \breve{U}^{\delta}+\cdots,
\]
where the function $\breve{U}^{\delta}$ has to be determined. Again, in accordance with \cite{ChCNSu}, we allow some $\delta$-dependence for the first term in the asymptotic expansion of $U^{\delta}$. However, $\breve{U}^{\delta}$ will be defined as the solution of some problem set in a geometry whose features do not depend on $\delta$. With the above change of variables, letting $\delta\to 0$, formally we obtain $0 = \mrm{div}_{\bfx}(\sigma^{\delta}\nabla_{\bfx} u^{\delta}) + 
\lambda^{\delta}_m u^{\delta} = \delta^{-2}( \mrm{div}_{\bfxi}(\sigma^{\infty}\nabla_{\bfxi} U^{\delta})+ 
\delta^{2}\lambda^{\delta}_m U^{\delta} )$, where $\sigma^{\infty}$ is the function such that $\sigma^{\infty}(\xi)=\sigma_{\pm}$ in $\Xi_{\pm}$ (see the definition of $\Xi_{\pm}$ on Figure \ref{Frozen geometry}). Since we assumed that $(\lambda^{\delta}_m)_{\delta}$ stays bounded as $\delta$ tends to zero, this leads us to impose the equations 
\begin{equation}\label{NearFieldEq}
-\mrm{div}(\sigma^{\infty}\nabla \breve{U}^{\delta}) = 0\ \mbox{ in } \Xi\qquad\mbox{ and }\qquad \breve{U}^{\delta}=0\ \mbox{ on } \partial\Xi\setminus\{O\}.
\end{equation}
Now, we set a functional framework for $\breve{U}^{\delta}$. Let us denote $(\rho,\theta)$ the polar coordinates in the geometry $\Xi$. For $\beta\in\R$, $k\geq 0$, 
we introduce the space $\mVc^{k}_{\beta}( \Xi)$ defined as the completion 
of the set $\{v|_{\Xi},\,v\in\mathscr{C}^{\infty}_0(\R^2)\}$ for the norm
\begin{equation}\label{WeightedNorms}
\Vert v\Vert_{\mVc^{k}_{\beta}( \Xi)} := \Big(\sum_{\vert\alpha\vert\le k}
\int_{ \Xi}(1+\rho^2)^{\beta+\vert\alpha\vert-k}\vert\partial^{\alpha}_{\bfxi} v
\vert^{2}\;d\bfxi\;\Big)^{1/2}.
\end{equation}
With these spaces, our goal is to discriminate behaviours of functions at infinity only and not at $O$. In the sequel, we will impose the Dirichlet boundary condition on $\partial\Xi$ working with functions belonging to 
\[
\mathring{\mVc}^{1}_{\beta}(\Xi):=\left\{ v\in 
\mVc^{1}_{\beta}(\Xi)\;\vert\;v = 0\mbox{ on }\partial\Xi\setminus\{O\}\right\}.
\]
For all $\beta\in\R$, this space coincides with the completion of $\mathscr{C}^{\infty}_0(\Xi)$ for the norm $\Vert\cdot\Vert_{\mVc^{1}_{\beta}( \Xi)}$. \\
\newline
Let us describe the structure of the 
solutions to problems of the form (\ref{NearFieldEq}) in the 
weighted Sobolev setting corresponding to norms (\ref{WeightedNorms}). We denote $\mathring{\mVc}_{\beta}^{1}(\Xi)^{*}$ the dual space to $\mathring{\mVc}_{\beta}^{1}(\Xi)$ 
equipped with the canonical norm defined similarly to (\ref{DualNorm}), and we let 
$\langle\cdot,\cdot\rangle_{\Xi}$ refer to the duality pairing between these two spaces.
First of all, we have a Fredholmness result.

\begin{proposition}\label{FredholmnessNearField}
Assume that $\kappa_{\sigma} = \sigma_{-}/\sigma_{+}\neq -1$. For $\beta\in\R$, introduce the bounded linear operator $\mathfrak{B}_{\beta}: \mathring{\mVc}_{\beta}^{1}(\Xi)\to \mathring{\mVc}_{-\beta}^{1}(\Xi)^{*}$ such that $\langle\mathfrak{B}_{\beta} u,v\rangle_{\Xi} = (\sigma^{\infty}\nabla u,\nabla v)_{\mrm{L}^2(\Xi)}$ for $u\in \mathring{\mVc}_{\beta}^{1}(\Xi)$, $v\in \mathring{\mVc}_{-\beta}^{1}(\Xi)$. Then,  $\mathfrak{B}_{\beta}$ is of Fredholm type if and only if no element $\lambda\in\Lambda$ satisfies $\Re e\,\lambda = \beta$ (for the definition of $\Lambda$, see (\ref{SingExp})).
\end{proposition}

\noindent
The proof of this proposition actually boils down to the proof of Proposition 
\ref{FredholmnessFarField}, after applying the transformation $\bfx
\mapsto \bfx/\vert\bfx\vert^{2}$. Like in \S\ref{StudySing}, the operators 
$\mathfrak{B}_{\beta}$ can be studied by means of Kondratiev's theory. 
Let us set
\begin{equation}\label{SingChPr}
\Sgpm(\rho,\theta)\;=\;\chi(\rho)\rho^{\pm i \mu}\phi(\theta).
\end{equation}
Here, $(\pm i\mu,\phi)\in\Cplx\times\mH^{1}_{0}(0;\pi)$ are the same pairs as in (\ref{DefSing}) and  $\chi$ is introduced with Figure \ref{cut-off functions} (let us remind that $\chi(\rho)=0$ for $\rho\le 1$ and  $\chi(\rho)=1$ for $\rho\ge 2$). Note that  $\Sgpm\in \mathring{\mVc}^{1}_{-1}(\Xi)$. 
Once again, transforming problem considered in Proposition 
\ref{FredholmnessNearField} by means of the map $\bfx\mapsto \bfx/\vert\bfx\vert^{2}$, we can prove the following result using the same arguments as for Proposition \ref{KernelGen}.

\begin{proposition}\label{DiffKernels}
There exists a unique (modulo $\Ker\,\mathfrak{B}_{+1}$) element $Z\in \Ker\,\mathfrak{B}_{-1}\setminus \Ker\,\mathfrak{B}_{+1}$
admitting the decomposition $Z=\Sgp + c_{Z}\Sgm + \tilde{Z}$ for some constant $c_{Z}\in\Cplx$ and some $\tilde{Z}\in \mathring{\mVc}^{1}_{\beta}(\Xi)$ for all $\beta\in(0;2)$. Moreover, we have $|c_{Z}|=1$. 
\end{proposition}
\noindent  From the above result, we infer that $\Ker\,\mathfrak{B}_{-1} = \mrm{span}\{
Z\}\oplus \Ker\,\mathfrak{B}_{+1}$. In the sequel, the following definition will be convenient. 
\begin{definition}
We call $\delta_{\bullet}$ the largest element of $(0;1]$ such that $\delta_{\bullet}^{-2i\mu}=c_{Z}$.
\end{definition}

\noindent  Depending on the parameters $\sigma_{\pm}$ and on the domains $\Xi_{\pm}$, it can happen that the operator $\mathfrak{B}_{+1}$ gets a non trivial 
(finite dimensional) kernel. We discard this possibility, considering an additional 
assumption
\begin{Assumption}\label{assumption1}
The operator $\mathfrak{B}_{+1}$ is injective.
\end{Assumption}

\noindent For a concrete case where this assumption is satisfied, one may for example consider the situation of Section \ref{Numerics}. It is not clear under which precise circumstances Assumption \ref{assumption1} is satisfied, and this point is already the subject of 
contributions in the current literature, see \cite{Grie14}. Another (open) question is wether or not the subsequent analysis can be 
carried out without this assumption. Nevertheless, as is shown in Section \ref{Numerics}, there are relevant geometries which comply with this assumption.\\
\newline
\noindent We will choose the function $\breve{U}^{\delta}$, which must satisfy equations (\ref{NearFieldEq}), equal to $c^{\delta}_{N}Z$. Here, $c^{\delta}_{N}$ is constant depending on $\delta$. Therefore, up to some remainder with respect to $\delta$, the field in the inner region admits the following behaviour as $\rho=r/\delta\to+\infty$:
\begin{equation}\label{dvptFOF}
\begin{array}{lcl}
\breve{U}^{\delta} & = & c^{\delta}_{N}\Sgp +c^{\delta}_{N}c_{Z}\Sgm +\cdots\\[4pt]
  & = &\dsp  c^{\delta}_{N}\left(\frac{r}{\delta}\right)^{+i\mu}\phi(\theta) +c^{\delta}_{N}\delta_{\bullet}^{-2i\mu}\left(\frac{r}{\delta}\right)^{-i\mu}\phi(\theta) +\cdots \ .
\end{array}
\end{equation}
\subsection{Matching of asymptotics and definition of the model operator}\label{paragraphMatching}
To conclude the construction of the first terms of the asymptotic expansion of the eigenpair $(\lambda^{\delta}_m,u^{\delta})$, we apply the matching procedure. In the present case, it consists in equating expansion (\ref{dvptFFE}), (\ref{dvptFOF}) as $r\to0$ and $r/\delta\to+\infty$:
\[
c^{\delta}_{+}\,r^{+i\mu}\phi(\theta) +c^{\delta}_{-}\,r^{-i\mu}\phi(\theta)+\cdots\quad = \quad\dsp  c^{\delta}_{N}\left(\frac{r}{\delta}\right)^{+i\mu}\phi(\theta) +c^{\delta}_{N}\delta_{\bullet}^{-2i\mu}\left(\frac{r}{\delta}\right)^{-i\mu}\phi(\theta)+\cdots\ .
\]
Since $r\mapsto r^{+i\mu}$ and $r\mapsto r^{-i\mu}$ are independent functions and since $\phi\neq 0$, 
this yields the identities $c^{\delta}_{+} = c^{\delta}_{N}\delta^{-i\mu}$, 
$c^{\delta}_{-} = c^{\delta}_{N}\delta_{\bullet}^{-2i\mu}\delta^{+i\mu}$, and, as a byproduct, the relation  $c^{\delta}_{-} = (\delta/\delta_{\bullet})^{2i\mu}c^{\delta}_{+}$. Therefore,  according to (\ref{NearFieldApprox}), $\breve{u}^{\delta}$ satisfies
\begin{equation}\label{relationMatching}
\breve{u}^{\delta} -c^{\delta}_{N}\delta^{-i\mu}(\sgp +(\delta/\delta_{\bullet})^{2i\mu}\sgm)\in D(\mathfrak{A}).
\end{equation}
This leads us to introduce the model operator $\mathfrak{A}^{\delta}\subset\mathfrak{A}^{*}$ such that 
\begin{equation}\label{defModelOp}
\begin{array}{|l}
\mathfrak{A}^{\delta}v\;=\; -\mrm{div}(\sigma^{0}\nabla v)\\[6pt]
D(\mathfrak{A}^{\delta}) = \mrm{span}\{\,\sgp + (\delta/\delta_{\bullet})^{2i\mu}\sgm\,\}\oplus D(\mathfrak{A}).
\end{array}
\end{equation}
With this definition, we observe from (\ref{relationMatching}) that $\breve{u}^{\delta}\in D(\mathfrak{A}^{\delta})$. Remembering that $(\eta^{\delta}_m,\breve{u}^{\delta})$ has to satisfy (\ref{LimitPb}), we deduce that the matching procedure imposes that $(\eta^{\delta}_m,\breve{u}^{\delta})$ be an eigenpair of the operator $\mathfrak{A}^{\delta}$. Since $|(\delta/\delta_{\bullet})^{2i\mu}|=1$, according to Proposition \ref{DescribingSAE}, we know that $\mathfrak{A}^{\delta}$ is self-adjoint for all $\delta\in(0;1]$ (note that with Definition (\ref{definitionOperateurAtau}), we have $\mathfrak{A}^{\delta}=\mathfrak{A}(2\mu\ln(\delta/\delta_{\bullet})$). We shall denote $\{\eta^{\delta}_{j}\}_{j\in\Z}$ the set of ordered eigenvalues (counted with their multiplicities) of $\mathfrak{A}^{\delta}$. As a consequence of Proposition 
\ref{BoundedEigenvalues}, for each fixed $j\in\Z$, the function $\delta\mapsto \eta_{j}^{\delta}$ is bounded
over $(0;1]$.
\begin{remark}\label{spectrumPeriodic}
When $\delta_1$, $\delta_2>0$ are such that $(\delta_1/\delta_{\bullet})^{2i\mu}=(\delta_2/\delta_{\bullet})^{2i\mu} \Leftrightarrow \ln\delta_1=\ln\delta_2+k\pi/\mu$ for some $k\in\Z$, we have $D(\mathfrak{A}^{\delta_1})=D(\mathfrak{A}^{\delta_2})$, so that $\mathfrak{A}^{\delta_1}=\mathfrak{A}^{\delta_2}$ and $\spec(\mathfrak{A}^{\delta_1})=\spec(\mathfrak{A}^{\delta_2})$. Therefore, for each fixed $j\in\Z$, the map $\delta\mapsto \eta_{j}^{\delta}$ is $\pi/\mu$-periodic in $\ln\delta$-scale.
\end{remark}
\begin{remark}
According to Proposition \ref{KernelGen}, there is a unique (modulo $\Ker\,\mathfrak{L}_{-1}\subset D(\mathfrak{A})$) element $\zeta\in D(\mathfrak{A}^{\ast})\setminus D(\mathfrak{A})$ admitting the decomposition $\zeta=\sgm+c_{\zeta}\,\sgp+\tilde{\zeta}$ with $c_{\zeta}\in\Cplx$ such that $|c_{\zeta}|=|\pi_{+}(\zeta)|=1$, and $\tilde{\zeta}\in D(\mathfrak{A})$. When $\delta>0$ is such that $(\delta/\delta_{\bullet})^{2i\mu}=c_{\zeta}$, the function $\zeta$ belongs to $D(\mathfrak{A}^{\delta})$. In this case, we have $0\in\spec(\mathfrak{A}^{\delta})$. Together with Remark \ref{spectrumPeriodic}, this shows that the set of values of $\delta$ such that $\mathfrak{A}^{\delta}$ is not injective accumulates at zero. Figure \ref{UnitCircleRun} illustrates the phenomenon. 
\begin{figure}[!ht]
\centering
\begin{tikzpicture}[scale=0.7]
\draw[line width=0.2mm] (0,0) circle (2cm);
\draw[dotted,->](-2.3,0)--(2.5,0);
\draw[dotted,->](0,-2.3)--(0,2.4);
\draw[dotted,<->](0,0)--(-1.41421356237,-1.41421356237);
\node at(1.25,2){\small$c_{\zeta}$};
\node at(-0.5,-0.85){\small$1$};
\node at(3,-1.2){\small $(\delta/\delta_{\bullet})^{2i\mu}$};
\draw[dashed,<-] (0,0) ++(320:2.2) arc (320:340:2.2);
\draw[fill=gray,draw=none] (0,0) ++(330:2) circle (0.6mm);
\draw[rotate around={60:(0,0)}] (1.9,0)--(2.1,0);
\end{tikzpicture}
\caption{As $\delta$ goes to zero, $(\delta/\delta_{\bullet})^{2i\mu}$ runs on the unit circle and hits $c_{\zeta}$ infinitely many times.\label{UnitCircleRun}}
\end{figure}
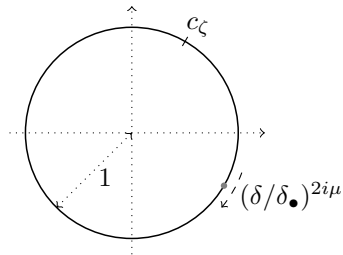
\end{remark}

\section{Main theorem}\label{SectionMainThm}

The formal asymptotic expansion above suggests that $\mathfrak{S}(\mA^{\delta})$, the spectrum of the original operator in the geometry with a rounded corner, behaves as $\mathfrak{S}(\mathfrak{A}^{\delta})$, the spectrum of the model operator, as $\delta$ goes to zero. To prove this result, we would like to show estimates on the inverses of these operators.  
However, this is not possible because the set of $\delta$ such that $\mathfrak{A}^{\delta}$ is not invertible accumulates at zero. To circumvent this problem, we will shift the spectrum in the complex plane working with $\mA^{\delta}+i\mrm{Id}$, $\mathfrak{A}^{\delta}+i\mrm{Id}$ instead of $\mA^{\delta}$, $\mathfrak{A}^{\delta}$. To begin with, we state an important result whose technical proof is postponed to the next section.
\begin{theorem}\label{propoErrorEtsimate}
Under Assumption \ref{assumption1}, for any $\eps>0$, there is a constant $C_{\eps}$ independent of $\delta\in(0;1]$ such that 
\begin{equation}\label{error estimate result L2}
\sup_{f\in\mL^{2}(\Omega)\setminus\{0\}}\frac{\Vert (\mA^{\delta}+i\mrm{Id})^{-1}f - (\mathfrak{A}^{\delta}+i\mrm{Id})^{-1}f
\Vert_{\mL^{2}(\Omega)}}{\Vert f\Vert_{\mL^{2}(\Omega)}}\;\leq \;C_{\eps}\,\delta^{1-\eps}.
\end{equation}
\end{theorem}
\noindent The result of this theorem can be rephrased as $(\mA^{\delta}+i\mrm{Id})^{-1} = (\mathfrak{A}^{\delta}+i\mrm{Id})^{-1} + O(\delta^{1-\eps})$ considering $(\mA^{\delta}+i\mrm{Id})^{-1},(\mathfrak{A}^{\delta}+i\mrm{Id})^{-1}$ as operators mapping $\mL^{2}(\Omega)$ to $\mL^{2}(\Omega)$. Since these two operators are normal, we deduce in the next proposition (see the proof in Appendix) that the spectra of $(\mA^{\delta}+i\mrm{Id})^{-1}$ and $(\mathfrak{A}^{\delta}+i\mrm{Id})^{-1}$ are closed to each other. 
\begin{proposition}\label{SpectralConvergence}
Under Assumption \ref{assumption1}, for any $\eps>0$, there is a constant $C_{\eps}$ independent of $\delta\in(0;1]$ such that 
\begin{equation}\label{distanceSpectra}
\sup_{\eta\in \mathfrak{S}(\mathfrak{A}^{\delta})}\mathop{\inf\phantom{p}}_{\lambda\in \mathfrak{S}(\mA^{\delta})}\Big\vert \frac{1}{\lambda+i}-\frac{1}{\eta+i}\Big\vert\; +\; 
\sup_{\lambda\in \mathfrak{S}(\mA^{\delta})}\mathop{\inf\phantom{p}}_{\eta\in \mathfrak{S}(\mathfrak{A}^{\delta})}\Big\vert \frac{1}{\lambda+i}-\frac{1}{\eta+i}\Big\vert\;\leq\; 
C_{\eps}\;\delta^{1-\eps}.
\end{equation}
\end{proposition}
\noindent Now, we are ready to establish the main theorem of the paper. 
\begin{theorem}\label{thmMajor}
Assume that $\kappa_{\sigma}\in(-1;-1/3)$.\\
Let $\{\lambda^{\delta}_{j}\}_{j\in\Z}$ refer to the set of ordered eigenvalues of the original operator $\mA^{\delta}$ defined in (\ref{ExOp}). \\
Let $\{\eta^{\delta}_{j}\}_{j\in\Z}$ refer to the set of ordered eigenvalues of the model operator $\mathfrak{A}^{\delta}$ defined in (\ref{defModelOp}).\\
Under Assumption \ref{assumption1}, for any $j\in\mathbb{Z},\,\eps>0$, there is a constant $C_{j,\,\eps}>0$ independent of $\delta\in(0;1]$ such that 
\begin{equation}\label{MainEstimate}
\mathop{\inf}_{\lambda\in\spec(\mA^{\delta})}\vert \eta^{\delta}_{j}-\lambda\vert+\mathop{\inf}_{\eta\in\spec(\mathfrak{A}^{\delta})}\vert \lambda^{\delta}_{j}-\eta\vert\;
\leq\; C_{j,\,\eps}\,\delta^{1-\eps}.
\end{equation}
\end{theorem}
\begin{proof}
Pick $j\in\Z$ and $\eps\in(0;1)$. According to Proposition \ref{BoundedEigenvalues}, we know that $\delta\mapsto\eta^{\delta}_{j}$ remains bounded as $\delta$ goes to zero. From (\ref{distanceSpectra}), we infer 
\[
\mathop{\inf}_{\lambda\in\spec(\mA^{\delta})}\vert \eta^{\delta}_{j}-\lambda\vert
\leq\; C_{j,\,\eps}\,\delta^{1-\eps},\qquad\forall \delta\in(0;1].
\]
This estimate, together with simple considerations based on the counting of the eigenvalues of $\mA^{\delta}$ allow to show that for all $j\in\Z$, $\delta\mapsto\lambda^{\delta}_{j}$ is bounded as $\delta\to0$. Using again (\ref{distanceSpectra}), finally we obtain (\ref{MainEstimate}).    
\end{proof}
\noindent Theorem \ref{thmMajor} guarantees that all the eigenvalues of the original operator $\mA^{\delta}$ behave as the eigenvalues of the model operator $\mathfrak{A}^{\delta}$ as $\delta$ goes to zero. Since the eigenvalues of $\mathfrak{A}^{\delta}$ are periodic in $\ln\delta$ - scale, this demonstrates that, asymptotically, all the eigenvalues of $\mA^{\delta}$ are periodic in $\ln\delta$ - scale as $\delta\to0$.

\section{Asymptotic analysis for the source term problem}\label{SectionSourceTermpb}

\noindent The goal of this section is to demonstrate Theorem \ref{propoErrorEtsimate}. Consider the source term problem 
\begin{equation}\label{PbAprioriEstimate}
\begin{array}{|l}
\mbox{Find }u^{\delta}\in\mH^1_{0}(\Om)\mbox{ such that }\\[4pt]
-\div(\sigma^{\delta}\nabla u^{\delta})+iu^{\delta} = f\quad\mbox{ in }\Om,
\end{array}
\end{equation}
where $f$ is a given function of $\mrm{L}^2(\Om)$. For a fixed $\delta\in(0;1]$, Problem (\ref{PbAprioriEstimate}) has a unique solution $u^{\delta}=(\mA^{\delta}+i\mrm{Id})^{-1}f$. Define $v^{\delta}:=(\mathfrak{A}^{\delta}+i\mrm{Id})^{-1}f$. Proving Theorem \ref{propoErrorEtsimate} boils down to show that $v^{\delta}$ is a good approximation of $u^{\delta}$ in the $\mrm{L}^2$-norm as $\delta$ goes to zero. Because the sign of $\sigma^{\delta}$ changes on $\Om$, proving such a result is a delicate task and the variational approach developed in \cite{BoCZ10,BoCC12} to establish Fredholm property for (\ref{PbAprioriEstimate}) seems useless here. Instead, we will employ a method introduced in \cite[Chap.\,2]{MaNP00}, \cite{Naza99} (see also \cite{ChCNSu,ChCN15} for examples of application in the context of negative materials). Our strategy is as follows. First in \S\ref{paragraphFirstAsymptotic}, we construct an almost inverse of $\mA^{\delta}+i\mrm{Id}$ that we denote $\hat{\mrm{R}}^{\delta}$. Then, working with $\hat{\mrm{R}}^{\delta}$, we prove a uniform stability estimate for $(\mA^{\delta}+i\mrm{Id})^{-1}$ as $\delta$ goes to zero. In a third step, we define a second asymptotic expansion of $u^{\delta}$, denoted  $\breve{\mrm{R}}^{\delta}f$, which involves directly the function $v^{\delta}$. We show that it yields a good approximation of $u^{\delta}$ using the stability estimate derived in \S\ref{paragraphStability}. Finally, we establish that $\breve{\mrm{R}}^{\delta}f$ and $v^{\delta}$  are closed to each other when $\delta\to0$.

\subsection{First asymptotic expansion}\label{paragraphFirstAsymptotic}

Let us construct a first asymptotic expansion of $u^{\delta}$, the solution of Problem (\ref{PbAprioriEstimate}). Decompose the source term $f\in\mL^{2}(\Omega)$ in an inner and an outer contribution,
\[
f(\bfx)=g(\bfx)+G(\bfx/\delta)\qquad\mbox{ with }\quad g(\bfx)=\chi_{\sqrt{\delta}}(r)f(\bfx)\quad\mbox{ and }\quad G(\bfxi)=\psi_{\sqrt{\delta}}(\delta\rho)f(\delta\bfxi).
\]
Let us emphasize that according to the definition of $\chi$, $\psi$ (see (\ref{cut-off functions})), 
there holds $\chi_{\sqrt{\delta}}+\psi_{\sqrt{\delta}}=1$. Moreover, since $\chi_{\sqrt{\delta}}=0$ for $r\le\sqrt{\delta}$ and 
$\psi_{\sqrt{\delta}}=0$ for $r\ge 2\sqrt{\delta}$, we know that $g$ vanishes in a neighbourhood of the origin and that $\mbox{supp}(G)$ (the support of $G$) is bounded in $\Xi$. As a consequence, for all $\beta\in\R$ we have $g\in\mV^0_{\beta}(\Om)$ and $G\in\mVc^{0}_{\beta}(\Xi)$. To define a far field expansion for $u^{\delta}$, we use the following result. 
\begin{proposition}\label{defFarField}
Assume that $\kappa_{\sigma}\in(-1;-1/3)$. For all $\beta\in(0;2)$, there is a unique $v$ admitting the decomposition $v=c\,(\sgp + (\delta/\delta_{\bullet})^{2i\mu}\sgm)+\tilde{v}$, with $c\in\Cplx$, $\tilde{v}\in\mathring{\mV}^1_{-\beta}(\Om)$, which satisfies $-\div(\sigma^{0}\nabla v)+i v = g$ in $\Om$. Moreover, we have $c=\pi_{+}(v)$ and 
\begin{equation}\label{eqnUnifContinuityPropo}
|\pi_+(v)|+ \Vert v - \pi_{+}(v)\,(\sgp + (\delta/\delta_{\bullet})^{2i\mu}\sgm)\Vert_{\mV_{-\beta}^{1}(\Omega)}\leq C\,\Vert g\Vert_{\mathring{\mV}_{\beta}^{1}(\Omega)^{*}}
\end{equation}
where $C$ depends on $\beta$ but not on $\delta\in(0;1]$.
\end{proposition}
\begin{proof}
For $\beta\in(0;2)$, define the space $\mV^{\out}_{\beta}(\Omega) := \mrm{span}\{\sgp + (\delta/\delta_{\bullet})^{2i\mu}\sgm\}
\oplus\mathring{\mV}^{1}_{-\beta}(\Om)$ endowed with the norm  
\[
\Vert v \Vert_{\mV^{\out}_{\beta}(\Omega)} := \vert c\vert + \Vert \tilde{v}\Vert_{\mV^{1}_{-\beta}(\Omega)}
\qquad\mbox{for }v =c\,(\sgp + (\delta/\delta_{\bullet})^{2i\mu}\sgm)+ \tilde{v},\,c\in\Cplx,\,\tilde{v}\in\mathring{\mV}^{1}_{-\beta}(\Omega).   
\]
Introduce the operator $\mathfrak{J}^{\out}_{\beta}: \mV^{\out}_{\beta}(\Omega)\to \mathring{\mV}_{\beta}^{1}(\Omega)^{*}$ such that $\langle\mathfrak{J}^{\out}_{\beta} u,v\rangle_{\Omega} = (\sigma^0\nabla u,\nabla v)_{\mL^{2}(\Omega)}+i(u,v)_{\mL^{2}(\Omega)}$, for all $u\in \mV^{\out}_{\beta}(\Omega)$, $v\in\mathscr{C}^{\infty}_0(\Omega)$. Using \cite[Thm.\,4.4]{BoCC13}, we can prove that for all $\delta\in(0;1]$, $\beta\in(0;2)$, $\mathfrak{J}^{\out}_{\beta}$ is Fredholm of index zero. If $u$ belongs to $\ker\,\mathfrak{J}^{\out}_{\beta}$, we have $0=\Im m\,\langle\mathfrak{J}^{\out}_{\beta} u,u\rangle_{\Omega}=\|u\|^2_{\mrm{L}^2(\Om)}$. This shows that $\mathfrak{J}^{\out}_{\beta}$ is an isomorphism. Finally, since the map $\delta\mapsto(\delta/\delta_{\bullet})^{2i\mu}$ is periodic in $\ln\delta$-scale, we can demonstrate that the constant $C$ in (\ref{eqnUnifContinuityPropo}) can be chosen independently of $\delta$.
\end{proof}
\noindent Now, we wish to define a near field expansion for $u^{\delta}$. To proceed, we use the following result.
\begin{proposition}\label{PropoFarField}
Assume that $\kappa_{\sigma}\in(-1;-1/3)$. Under Assumption \ref{assumption1}, for all $\beta\in(0;2)$, the operator $\mathfrak{B}_{-\beta}: \mathring{\mVc}_{-\beta}^{1}(\Xi)\to \mathring{\mVc}_{\beta}^{1}(\Xi)^{*}$ defined in Proposition \ref{FredholmnessNearField} is onto and $\ker\,\mathfrak{B}_{-\beta}=\mrm{span}\{Z\}$ where $Z$ is introduced in Proposition  \ref{DiffKernels}. 
\end{proposition}
\begin{proof}
When $\kappa_{\sigma}\in(-1;-1/3)$, as seen in \S\ref{StudySing}, we have $
\{\lambda\in\Lambda\,|\,-2<\Re e\,\lambda<2\}=\{\pm i\mu\}$. Therefore, Proposition \ref{FredholmnessNearField} guarantees that for all $\beta\in(0;2)$, $\mathfrak{B}_{-\beta}$ and $\mathfrak{B}_{\beta}$ are Fredholm operators. Since $\mathfrak{B}_{-\beta}$ is the adjoint of $\mathfrak{B}_{\beta}$, we have $\ind\,\mathfrak{B}_{-\beta}=-\ind\,\mathfrak{B}_{\beta}$ (here $\ind$ stands for the index). On the other hand, the equality 
$\{\lambda\in\Lambda\,|\,-2<\Re e\,\lambda<2\}=\{\pm i\mu\}$, together with \cite[Chap.\,4, Prop.\,3.1]{NaPl94}, implies that $\ind\,\mathfrak{B}_{-\beta}-\ind\,\mathfrak{B}_{\beta}=2$. From the two previous relations, we infer $\ind\,\mathfrak{B}_{-\beta}=-\ind\,\mathfrak{B}_{\beta}=1$. Applying the Kondratiev theory, since $\{\lambda\in\Lambda\,|\,0<|\Re e\,\lambda|<2\}=\emptyset$, we can prove that for all $\beta\in(0;2)$, $\ker\,\mathfrak{B}_{-\beta}=\ker\,\mathfrak{B}_{-1}$ and $\ker\,\mathfrak{B}_{\beta}=\ker\,\mathfrak{B}_{+1}$. Under Assumption \ref{assumption1}, we deduce that $\ker\,\mathfrak{B}_{-\beta}=\mrm{span}\{Z\}$ and that $\mathfrak{B}_{-\beta}$ is onto (because $\mathfrak{B}_{-\beta}$ is the adjoint of $\mathfrak{B}_{\beta}$ which is injective).
\end{proof}
\noindent Proposition \ref{PropoFarField} ensures that for all $\beta\in(0;2)$, there is a unique function $U\in\mathring{\mathcal{V}}^1_{-\beta}(\Xi)$ satisfying $-\div(\sigma^{\infty}\nabla U)=G$ in $\Xi$ and the orthogonality condition $\int_{\Xi\cap \mrm{D}(0,R)}U(\bfxi)\,Z(\bfxi)\,d\bfxi=0$ for some given $R>0$. Moreover, we have the continuity estimate
\begin{equation}\label{defNearField}
\|U\|_{\mathcal{V}^1_{-\beta}(\Xi)} \le C\,\|G\|_{\mathring{\mathcal{V}}^1_{\beta}(\Xi)^{\ast}}.
\end{equation}
Then we set $\hat{V}_{\delta}(\bfx)=\delta^{2}U_{\delta}(\bfx)+\delta^{i\mu}\pi_{+}(v)Z_{\delta}(\bfx)$ with $U_{\delta}$, $Z_{\delta}$ such that $U_{\delta}(\bfx)=U(\bfx/\delta)$, $Z_{\delta}(\bfx)=Z(\bfx/\delta)$. In the definition of $\hat{V}_{\delta}$, the multiplicative term in front of $Z_{\delta}$ is added so that the behaviour of $\hat{V}_{\delta}$ matches with the one of $v$ when $r\to0$, $r/\delta\to+\infty$ (exactly as in \S\ref{paragraphMatching}). Finally, we define the linear map $\hat{\mrm{R}}^{\delta}:\mrm{L}^2(\Om)\to\mrm{L}^2(\Om)$ such that $\hat{\mrm{R}}^{\delta}f=\hat{u}^{\delta}$ ($f$ is the source term appearing in (\ref{PbAprioriEstimate})) with
\begin{equation}\label{constructionUhat}
\hat{u}^{\delta} = \chi_{\delta}\,v + \psi\,\hat{V}_{\delta} - \psi\chi_{\delta}\,\hat{m}_{\delta}. 
\end{equation} 
In (\ref{constructionUhat}), we take $\hat{m}_{\delta}$ such that $\hat{m}_{\delta}(\bfx)=\pi_{+}(v)\big( r^{+i\mu}+ 
(\delta/\delta_{\bullet})^{2i\mu}r^{-i\mu}\big)\phi(\theta)$. This function represents the predominant behaviour of $v$ 
(resp. $\hat{V}_{\delta}$) as $r\to0$ (resp. $r/\delta\to+\infty$).

\subsection{Stability estimate}\label{paragraphStability}
Now, with the operator $\hat{\mrm{R}}^{\delta}$, we prove a uniform stability estimate for $(\mA^{\delta}+i\mrm{Id})^{-1}$ when $\delta$ goes to zero. To proceed, as mentioned above, we will work with specific norms using a method introduced in \cite[Chap.\,2]{MaNP00}, \cite{Naza99}. To implement the technique, 
we need to introduce the Hilbert spaces $\mV^0_{\beta,\, \delta}(\Om)$, $\beta\in\R$.
These spaces are defined as the completions of $\mathscr{C}^{\infty}(\overline{\Om})$ for the weighted norms
\begin{equation}\label{weighted space bounded}
 \|v\|_{\mV^0_{\beta,\delta}(\Om)}  := 
\|(r+\delta)^{\beta}v\|_{\mrm{L}^2(\Om)}.
\end{equation}
Observe that for any $\beta\in\R$ and $\delta>0$, the space $\mV^0_{\beta,\, \delta}(\Om)$ coincides with  $\mrm{L}^2(\Om)$ because the norm (\ref{weighted space bounded}) is equivalent to $\|\cdot\|_{\mrm{L}^2(\Om)}$. However, the constants coming into play in this equivalence depend on $\delta$ which is a crucial feature.

\begin{proposition}\label{proposition stability estimate}
Under Assumption \ref{assumption1}, for any $\beta\in(0;1)$, there is $C_{\beta}>0$ independent of $\delta$ such that 
\begin{equation}\label{lemma stability estimate result}
\sup_{f\in\mrm{L}^2(\Om)\setminus\{0\}}\frac{
\|(\mA^{\delta}+i\mrm{Id})^{-1}f\|_{\mL^{2}(\Omega)}}{
\|f\|_{\mV^0_{1-\beta,\delta}(\Om)}}\le C_{\beta},\qquad\forall \delta\in(0;1].
\end{equation}
\end{proposition}
\begin{proof}
As above, we denote $\hat{u}^{\delta}=\hat{\mrm{R}}^{\delta}f$. The general scheme, whose steps will be justified hereafter, is the following. We calculate 
\begin{equation}\label{ToProve1}
(\mA^{\delta}+i\mrm{Id})(\hat{\mrm{R}}^{\delta}f)= -\div(\sigma^{\delta}\nabla \hat{u}^{\delta})+i\hat{u}^{\delta} = f+\hat{\mrm{K}}^{\delta}f
\end{equation}
where $\hat{\mrm{K}}^{\delta}:\mrm{L}^2(\Om)\to\mrm{L}^2(\Om)$ is a linear operator defined in (\ref{defExplicitOpK}). We prove that there holds 
\begin{equation}\label{ToProve2}
\|\hat{\mrm{K}}^{\delta}f\|_{\mV^0_{1-\beta,\delta}(\Om)}\le C\,\delta^{\gamma/2}\,\|f\|_{\mV^0_{1-\beta,\delta}(\Om)},
\end{equation}
where $\beta$ is chosen in $(0;1)$ and where $\gamma>0$ is chosen so that $\gamma<\beta$. Here and in the following, $C>0$ denotes a constant, 
which can change from one line to another and depends on $\beta$, but which does not depend on $\delta$. We infer that for $\delta$ small enough, 
$\mrm{Id}+\hat{\mrm{K}}^{\delta}$ is invertible as an operator of $\mrm{L}^2(\Om)$ endowed with the norm $\|\cdot\|_{\mV^0_{1-\beta,\delta}(\Om)}$. On 
the other hand, a direct computation yields
\begin{equation}\label{ToProve3}
\|\hat{\mrm{R}}^{\delta}f\|_{\mrm{L}^2(\Om)}=\|\hat{u}^{\delta}\|_{\mrm{L}^2(\Om)} \le C\,\|f\|_{\mV^0_{1-\beta,\delta}(\Om)}.
\end{equation}
Using (\ref{ToProve1}), (\ref{ToProve3}), then we can write
\begin{equation}\label{estimStab}
\Vert (\mA^{\delta}+i\mrm{Id})^{-1}f\Vert_{\mL^{2}(\Omega)} = \|\hat{\mrm{R}}^{\delta}(\mrm{Id}+\hat{\mrm{K}}^{\delta})^{-1}f \|_{\mrm{L}^2(\Om)}
\le C\,\|(\mrm{Id}+\hat{\mrm{K}}^{\delta})^{-1}f\|_{\mV^0_{1-\beta,\delta}(\Om)}\le C\,\|f\|_{\mV^0_{1-\beta,\delta}(\Om)}.
\end{equation}
Taking the supremum over all $f\in\mL^{2}(\Omega)$, we obtain $(\ref{lemma stability estimate result})$. To end the proof, it remains to establish (\ref{ToProve1})--(\ref{ToProve3}). We first write estimates which will be useful in the analysis. Pick some $\beta\in(0;1)$, $\gamma\in(0;\beta)$ and apply (\ref{eqnUnifContinuityPropo}) with $\beta$ replaced by $\beta+\gamma$. We get
\begin{equation}\label{eqnUnifContinuity}
|\pi_+(v)|+ \Vert v - \pi_{+}(v)\,(\sgp + (\delta/\delta_{\bullet})^{2i\mu}\sgm)\Vert_{\mV_{-\beta-\gamma}^{1}(\Omega)}\leq C\,\Vert g\Vert_{\mathring{\mV}_{\beta+\gamma}^{1}(\Omega)^{*}}.
\end{equation}
Again, we emphasize that $C$ depends on $\beta$, $\gamma$ but not on $\delta$. Moreover, we have
\begin{equation}\label{estimation new source term}
\begin{array}{lcl}
\Vert g\Vert_{\mathring{\mV}_{\beta+\gamma}^{1}(\Omega)^{*}} \le \Vert r^{-\beta-\gamma+1}g\Vert_{\mrm{L}^2(\Om)} & \le & \Vert r^{-\beta-\gamma+1}g\Vert_{\mrm{L}^2(\Om\setminus\overline{\mrm{D}(O,\sqrt{\delta})})} \\[6pt]
 & \le & C\delta^{-\gamma/2}\,\Vert r^{-\beta+1}g\Vert_{\mrm{L}^2(\Om\setminus\overline{\mrm{D}(O,\sqrt{\delta})})} \\[6pt]
 & \le & C\delta^{-\gamma/2}\,\Vert (r+\delta)^{-\beta+1}g\Vert_{\mrm{L}^2(\Om\setminus\overline{\mrm{D}(O,\sqrt{\delta})})} \\[6pt]
 & \le & C\delta^{-\gamma/2}\,\|f\|_{\mV^0_{1-\beta,\delta}(\Om)}.
\end{array}
\end{equation}
Note that taking $\gamma=0$ in (\ref{eqnUnifContinuity}), (\ref{estimation new source term}) gives 
\begin{equation}\label{estimMoinsUn}
|\pi_+(v)|+ \Vert v - \pi_{+}(v)\,(\sgp + (\delta/\delta_{\bullet})^{2i\mu}\sgm)\Vert_{\mV_{-\beta}^{1}(\Omega)} \leq  C\,\Vert g\Vert_{\mathring{\mV}_{\beta}^{1}(\Omega)^{*}} \le  C\,\|f\|_{\mV^0_{1-\beta,\delta}}.
\end{equation}
For the near field term $U$, apply (\ref{defNearField}) with $-\beta$ replaced by $-\beta+\gamma$. This yields 
\begin{equation}\label{eqnUnifContinuityNear}
\|U\|_{\mathcal{V}^1_{-\beta+\gamma}(\Xi)} \le C\,\|G\|_{\mathring{\mathcal{V}}^1_{\beta-\gamma}(\Xi)^{\ast}}.
\end{equation}
Observe that  
\begin{equation}\label{estimation new source termNF}
\begin{array}{lcl}
 \|G\|_{\mathring{\mathcal{V}}^1_{\beta-\gamma}(\Xi)^{\ast}} &\le &\|(1+\rho)^{-\beta+\gamma+1}G\|_{\mrm{L}^2(\Xi)} \\[6pt]
 & \le &\|(1+\rho)^{-\beta+\gamma+1}G\|_{\mrm{L}^2(\Xi\setminus\overline{\mrm{D}(O,2/\sqrt{\delta})}}  \\[6pt]
  & \le & C\,\delta^{-\gamma/2}\|(1+\rho)^{-\beta+1}G\|_{\mrm{L}^2(\Xi\setminus\overline{\mrm{D}(O,2/\sqrt{\delta})}}\  \le\   C\,\delta^{\beta-2-\gamma/2}\,\|f\|_{\mV^0_{1-\beta,\delta}(\Om)}.
\end{array}
\end{equation}
The last inequality in (\ref{estimation new source termNF}) has been obtained making the change of 
variables $\bfx=\delta\bfxi$. This explains the appearance of the term $\delta^{\beta-2}$. Now, we show (\ref{ToProve1})--(\ref{ToProve3}).\\
\newline
$\star$ Proof of (\ref{ToProve1}). A direct computation provides 

\begin{equation}\label{commutator hat}
\begin{array}{lcl}
  -\div(\sigma^{\delta}\nabla \hat{u}^{\delta})+i\hat{u}^{\delta} 
 &= &  -\chi_{\delta}\,\div(\sigma^{\delta}\nabla v)-\psi\,\div(\sigma^{\delta}\nabla \hat{V}_{\delta})+\psi\chi_{\delta}\,\div(\sigma^{\delta}\nabla \hat{m}_{\delta})\\[4pt]
 & & -\big[\div(\sigma^{\delta}\nabla\cdot),\chi_{\delta}\big]v-
\big[\div(\sigma^{\delta}\nabla\cdot),\psi\big]\hat{V}_{\delta}+
\big[\div(\sigma^{\delta}\nabla\cdot),\psi\chi_{\delta}\big]\hat{m}_{\delta}\\[4pt]
& & +i\chi_{\delta}\,v +i\psi\,\hat{V}_{\delta} - i\psi\chi_{\delta}\,\hat{m}_{\delta}\\[4pt]
&= & \chi_{\delta}\,\chi_{\sqrt{\delta}}\,f+\psi\,\psi_{\sqrt{\delta}}\,f\\[4pt]
&  & -\big[\div(\sigma^{\delta}\nabla\cdot),\chi_{\delta}\big](v-\hat{m}_{\delta})-
\big[\div(\sigma^{\delta}\nabla\cdot),\psi\big](\hat{V}_{\delta}-\hat{m}_{\delta})\\[4pt]
& & +i\psi\,(\hat{V}_{\delta} - i\chi_{\delta}\,\hat{m}_{\delta})\\[4pt]
& = & f -\big[\div(\sigma^{\delta}\nabla\cdot),\chi_{\delta}\big](v-\hat{m}_{\delta})-\big[\div(\sigma^{\delta}\nabla\cdot),\psi\big](\hat{V}_{\delta}-\hat{m}_{\delta})\\[4pt]
& & +i\psi\,(\hat{V}_{\delta} - i\chi_{\delta}\,\hat{m}_{\delta}).
\end{array}
\end{equation}
In the above equalities, the commutator $[A,B]$ is defined by $[A,B]=AB-BA$. In particular, observing that $\nabla(\psi\chi_{\delta})=\nabla\psi+\nabla\chi_{\delta}$, we find 
\begin{equation}\label{explicitCommutator}
\begin{array}{lcl}
\big[\div(\sigma^{\delta}\nabla\cdot),\psi\chi_{\delta}\big]\hat{m}_{\delta} \hspace{-0.1cm}&\hspace{-0.1cm}=&\hspace{-0.1cm} \div(\sigma^{\delta}\nabla(\psi\chi_{\delta}\hat{m}_{\delta}))-\psi\chi_{\delta}\,\div(\sigma^{\delta}\nabla \hat{m}_{\delta})\\[4pt]
\hspace{-0.1cm}&\hspace{-0.1cm}=&\hspace{-0.1cm} 2\sigma^{\delta}\nabla(\psi\chi_{\delta})\cdot\nabla \hat{m}_{\delta}+\hat{m}_{\delta}\,\div(\sigma^{\delta}\nabla(\psi\chi_{\delta}))\\[4pt]
\hspace{-0.1cm}&\hspace{-0.1cm}=&\hspace{-0.1cm} 2\sigma^{\delta}\nabla\chi_{\delta}\cdot\nabla \hat{m}_{\delta}+\hat{m}_{\delta}\,\div(\sigma^{\delta}\nabla\chi_{\delta}) +2\sigma^{\delta}\nabla\psi\cdot\nabla \hat{m}_{\delta}+\hat{m}_{\delta}\,\div(\sigma^{\delta}\nabla\psi)\\[4pt]
\hspace{-0.1cm}&\hspace{-0.1cm}=&\hspace{-0.1cm}\big[\div(\sigma^{\delta}\nabla\cdot),\chi_{\delta}\big]\hat{m}_{\delta}+\big[\div(\sigma^{\delta}\nabla\cdot),\psi\big]\hat{m}_{\delta}.
\end{array}
\end{equation}
In (\ref{commutator hat}), we also use that $\chi_{\delta}\,\chi_{\sqrt{\delta}}=\chi_{\sqrt{\delta}}$, $\psi_{\delta}\,\psi_{\sqrt{\delta}}=\psi_{\sqrt{\delta}}$ and $\chi_{\sqrt{\delta}}+\psi_{\sqrt{\delta}}=1$. From (\ref{commutator hat}), we infer that the operator $\hat{\mrm{K}}$ introduced in (\ref{ToProve1}) is defined by 
\begin{equation}\label{defExplicitOpK}
\hat{\mrm{K}}f=-\big[\div(\sigma^{\delta}\nabla\cdot),\chi_{\delta}\big](v-\hat{m}_{\delta})-\big[\div(\sigma^{\delta}\nabla\cdot),\psi\big](\hat{V}_{\delta}-\hat{m}_{\delta}) +i\psi\,(\hat{V}_{\delta} - \chi_{\delta}\,\hat{m}_{\delta}).
\end{equation}
$\star$ Proof of (\ref{ToProve2}). To compute the norm of $\hat{\mrm{K}}$, we will assess each of the terms of the right hand side of (\ref{defExplicitOpK}). For the first one, working as in (\ref{explicitCommutator}), we find
\[
\big[\div(\sigma^{\delta}\nabla\cdot),\chi_{\delta}\big](v-\hat{m}_{\delta}) = 2\sigma^{\delta}\nabla\chi_{\delta}\cdot\nabla(v-\hat{m}_{\delta})+(v-\hat{m}_{\delta})\,\div(\sigma^{\delta}\nabla\chi_{\delta})
\]
Define, for $t>0$, $\mathbb{Q}^{t}:=\{\bfx\in\Xi\,|\,t<|\bfx|<2t\}$. Noticing that $|\nabla\chi_{\delta}|\le C\,\delta^{-1}$ and $|\div(\sigma^{\delta}\nabla\chi_{\delta})|\le C\,\delta^{-2}$, we can write 
\begin{equation}\label{estimateReminder1}
\begin{array}{ll}
& \|\big[\div(\sigma^{\delta}\nabla\cdot),\chi_{\delta}\big](v-\hat{m}_{\delta})\|_{\mV^0_{1-\beta,\delta}(\Om)}\\[4pt] 
 \le &
\|(r+\delta)^{1-\beta}\sigma^{\delta}\nabla\chi_{\delta}\cdot\nabla(v-\hat{m}_{\delta})\|_{\mrm{L}^2(\Om)}+\|(r+\delta)^{1-\beta}(v-\hat{m}_{\delta})\,\div(\sigma^{\delta}\nabla\chi_{\delta})\|_{\mrm{L}^2(\Om)}\\[4pt]
 \le & C\,(\delta^{-1}\|(r+\delta)^{1-\beta}\nabla(v-\hat{m}_{\delta})\|_{\mrm{L}^2(\mathbb{Q}^{\delta})}+\delta^{-2}\|(r+\delta)^{1-\beta}(v-\hat{m}_{\delta})\|_{\mrm{L}^2(\mathbb{Q}^{\delta})})\\[4pt]
 \le & C\,\delta^{\gamma}\,(\|r^{-\beta-\gamma}\nabla(v-\hat{m}_{\delta})\|_{\mrm{L}^2(\mathbb{Q}^{\delta})}+\|r^{-\beta-\gamma-1}(v-\hat{m}_{\delta})\|_{\mrm{L}^2(\mathbb{Q}^{\delta})})\\[4pt]
 \le & C\,\delta^{\gamma}\,\|v-\hat{m}_{\delta}\|_{\mV^1_{-\beta-\gamma}(\Om)} \ \le\ C\,\delta^{\gamma/2}\,\|f\|_{\mV^0_{1-\beta,\delta}(\Om)}.
\end{array}
\end{equation}
In (\ref{estimateReminder1}), we use that we have $|\bfx|\le |\bfx|+\delta \le 2|\bfx|$ in $\mathbb{Q}^{\delta}$. Moreover, the last inequality comes from (\ref{eqnUnifContinuity}), (\ref{estimation new source term}). Proceeding similarly, we find 
$\|\big[\div(\sigma^{\delta}\nabla\cdot),\psi\big](\hat{V}_{\delta}-\hat{m}_{\delta})\|_{\mV^0_{1-\beta,\delta}(\Om)} \le C\,\delta^{\gamma/2}\,\|f\|_{\mV^0_{1-\beta,\delta}(\Om)}$. Now, we bound the third term of the right hand side of (\ref{defExplicitOpK}). Triangular inequality implies 
\begin{equation}\label{tiroir1}
 \|(r+\delta)^{1-\beta}\psi\,(\hat{V}_{\delta} - \chi_{\delta}\,\hat{m}_{\delta})\|_{\mrm{L}^2(\Om)} \le  \|(r+\delta)^{1-\beta}\psi\,(\hat{V}_{\delta} - \hat{m}_{\delta})\|_{\mrm{L}^2(\Om)}+\|(r+\delta)^{1-\beta}\psi_{\delta}\,\hat{m}_{\delta}\|_{\mrm{L}^2(\Om)}.
\end{equation}
On the one hand, using (\ref{estimMoinsUn}), (\ref{eqnUnifContinuityNear}) and (\ref{estimation new source termNF}), we find 
\begin{equation}\label{tiroir2}
\begin{array}{ll}
&\|(r+\delta)^{1-\beta}\psi\,(\hat{V}_{\delta} - \hat{m}_{\delta})\|_{\mrm{L}^2(\Om)} \\[4pt]
\le & \|(r+\delta)^{1-\beta}\psi\,\delta^2U_{\delta}\|_{\mrm{L}^2(\Om)}  +\|(r+\delta)^{1-\beta}\psi\,(\delta^{i\mu}\pi_{+}(v)Z_{\delta}- \hat{m}_{\delta})\|_{\mrm{L}^2(\Om)} \\[4pt]
\le & \delta^{4-\beta}\|(1+\rho)^{1-\beta}\psi_{1/\delta}\,U\|_{\mrm{L}^2(\Xi)}  +C\,\delta^{2-\beta}\,|\pi_+(v)|\,\|(1+\rho)^{1-\beta}\psi_{1/\delta}\,\tilde{Z}\|_{\mrm{L}^2(\Xi)}\\[4pt] 
\le & C\,\delta^{2-\beta+\gamma}\|(1+\rho)^{-\beta+\gamma-1}\,U\|_{\mrm{L}^2(\Xi)}  +C\,\delta^{2-\beta}\,\|\tilde{Z}\|_{\mathcal{V}^1_{2-\beta}(\Xi)}\,\|f\|_{\mV^0_{1-\beta,\delta}(\Om)}\\[4pt] 
\le & C\,\delta^{\gamma/2}\,\|f\|_{\mV^0_{1-\beta,\delta}(\Om)}.
\end{array}
\end{equation} 
On the other hand, with (\ref{estimMoinsUn}), we obtain
\begin{equation}\label{tiroir3}
\begin{array}{lcl}
\|(r+\delta)^{1-\beta}\psi_{\delta}\,\hat{m}_{\delta}\|_{\mrm{L}^2(\Om)} & \le & |\pi_+(v)|\,\|(r+\delta)^{1-\beta}\psi_{\delta}\|_{\mrm{L}^{2}(\Om)}\,\|\big(r^{+i\mu}+ 
(\delta/\delta_{\bullet})^{2i\mu}r^{-i\mu}\big)\phi\|_{\mrm{L}^{\infty}(\Om)}\\[3pt]
 & \le &  C\,\|(r+\delta)^{1-\beta}\psi_{\delta}\|_{\mrm{L}^2(\Om)}\,\|f\|_{\mV^0_{1-\beta,\delta}(\Om)} \ \le  \ 
C\,\delta^{2-\beta}\,\|f\|_{\mV^0_{1-\beta,\delta}(\Om)}.
\end{array}
\end{equation}
Plugging (\ref{tiroir2}) and (\ref{tiroir3}) in (\ref{tiroir1}) provides a good estimate for the third term of the right hand side of (\ref{defExplicitOpK}).\\
\newline
$\star$ Proof of (\ref{ToProve3}). We have 
\begin{equation}\label{eqnExpansion}
\|\hat{u}^{\delta}\|_{\mrm{L}^2(\Om)} \le \|\chi_{\delta}\,v\|_{\mrm{L}^2(\Om)} + \|\psi\,(\hat{V}_{\delta}-\chi_{\delta}\,\hat{m}_{\delta})\|_{\mrm{L}^2(\Om)}.
\end{equation}
Exploiting (\ref{estimMoinsUn}), we get
\begin{equation}\label{estimInterFin1}
\begin{array}{lcl}
\|\chi_{\delta}\,v\|_{\mrm{L}^2(\Om)} & \le & |\pi_+(v)|\,\|\sgp + (\delta/\delta_{\bullet})^{2i\mu}\sgm\|_{\mrm{L}^2(\Om)}+\Vert v - \pi_{+}(v)\,(\sgp + (\delta/\delta_{\bullet})^{2i\mu}\sgm)\Vert_{\mrm{L}^2(\Omega)}\\[4pt]
& \le &  C\,(\|f\|_{\mV^0_{1-\beta,\delta}(\Om)}+\Vert v - \pi_{+}(v)\,(\sgp + (\delta/\delta_{\bullet})^{2i\mu}\sgm)\Vert_{\mV_{-\beta}^{1}(\Omega)}) \ \le\  C\,\|f\|_{\mV^0_{1-\beta,\delta}(\Om)}.
\end{array}
\end{equation}
For the second term of the right hand side of (\ref{eqnExpansion}), adapting (\ref{tiroir1})--(\ref{tiroir3}), we find
\begin{equation}\label{estimInterFin2}
\|\psi\,(\hat{V}_{\delta}-\chi_{\delta}\,\hat{m}_{\delta})\|_{\mrm{L}^2(\Om)} \le C\,\|f\|_{\mV^0_{1-\beta,\delta}(\Om)}.
\end{equation}
Plugging (\ref{estimInterFin1}), (\ref{estimInterFin2}) in (\ref{eqnExpansion}) gives $\|\hat{u}^{\delta}\|_{\mrm{L}^2(\Om)}\le C\,\|f\|_{\mV^0_{1-\beta,\delta}(\Om)}$, which is exactly Estimate (\ref{ToProve3}).
\end{proof}

\subsection{Second asymptotic expansion}

In this section, we construct a second asymptotic expansion of $u^{\delta}$ the solution to Problem (\ref{PbAprioriEstimate}). This expansion will be a bit different from the one derived in \S\ref{paragraphFirstAsymptotic}. In particular, it will involve directly the function $v^{\delta}:=(\mathfrak{A}^{\delta}+i\mrm{Id})^{-1}f$. This feature will be very useful to prove in \S\ref{paragraphFinalProof} that $v^{\delta}$ is a good approximation of $u^{\delta}$, which is our final goal. \\
\newline
Set $\breve{V}_{\delta}(\bfx)=\delta^{i\mu}\pi_{+}(v^{\delta})Z_{\delta}(\bfx)$, with $Z_{\delta}$ as in (\ref{constructionUhat}), and  $\breve{m}_{\delta}(\bfx)=\pi_{+}(v^{\delta})\big( r^{+i\mu}+ 
(\delta/\delta_{\bullet})^{2i\mu}r^{-i\mu}\big)\phi(\theta)$. Then, define the linear map $\breve{\mrm{R}}^{\delta}:\mrm{L}^2(\Om)\to\mrm{L}^2(\Om)$ such that $\breve{\mrm{R}}^{\delta}f=\breve{u}^{\delta}$ with
\begin{equation}\label{constructionUtilde}
\breve{u}^{\delta} = \chi_{\delta}\,v^{\delta} + \psi\,\breve{V}_{\delta} - \psi\chi_{\delta}\,\breve{m}_{\delta}. 
\end{equation} 
Note that $\breve{m}_{\delta}$ represents the main contribution of $v^{\delta}$ (resp. $\breve{V}_{\delta}$) as $r\to0$ (resp. $r/\delta\to+\infty$). Using the uniform stability estimate for $(\mA^{\delta}+i\mrm{Id})^{-1}$, we will show that $\breve{\mrm{R}}^{\delta}f=\breve{u}^{\delta}$ is a good approximation of $(\mA^{\delta}+i\mrm{Id})^{-1}f=u^{\delta}$ as $\delta$ goes to zero.
\begin{proposition}\label{proposition error estimate}
Under Assumption \ref{assumption1}, for any $\eps>0$, there is $C_{\eps}>0$ independent of $\delta$ such that 
\begin{equation}\label{lemma stability estimate result tilde}
\sup_{f\in\mrm{L}^2(\Om)\setminus\{0\}}\frac{
\|(\mA^{\delta}+i\mrm{Id})^{-1}f - \breve{\mrm{R}}^{\delta}f\|_{\mL^{2}(\Omega)}}{
\|f\|_{\mrm{L}^2(\Om)}}\le C_{\eps}\,\delta^{1-\eps},\qquad\forall \delta\in(0;1].
\end{equation}
\end{proposition}
\begin{proof}
Working as in (\ref{commutator hat}), we find $(\mA^{\delta}+i\mrm{Id})(\breve{u}^{\delta}-u^{\delta})= -\div(\sigma^{\delta}\nabla (\breve{u}^{\delta}-u^{\delta}))+i (\breve{u}^{\delta}-u^{\delta})=\breve{\mrm{K}}^{\delta}f$ with 
\begin{equation}\label{commutator tilde}
\breve{\mrm{K}}^{\delta}f= -\psi_{\delta}f -\big[\div(\sigma^{\delta}\nabla\cdot),\chi_{\delta}\big](v^{\delta}-\breve{m}_{\delta})-\big[\div(\sigma^{\delta}\nabla\cdot),\psi\big](\breve{V}_{\delta}-\breve{m}_{\delta}) +i\psi\,(\breve{V}_{\delta} - i\chi_{\delta}\,\breve{m}_{\delta}).
\end{equation}
According to Proposition \ref{proposition error estimate}, we have 
\begin{equation}\label{UtilisationStability}
\|(\mA^{\delta}+i\mrm{Id})^{-1}f - \breve{\mrm{R}}^{\delta}f\|_{\mL^{2}(\Omega)}=\|u^{\delta}-\breve{u}^{\delta}\|_{\mL^{2}(\Omega)} = \|(\mA^{\delta}+i\mrm{Id})^{-1}(\breve{\mrm{K}}^{\delta}f)\|_{\mL^{2}(\Omega)} \le C\,\|\breve{\mrm{K}}^{\delta}f\|_{\mV^0_{1-\beta,\delta}(\Om)},
\end{equation}
where $C$ is a constant independent from $\delta$ and where $\beta$ is set in $(0;1)$. Therefore, we see it is sufficient to examine each of the terms of the right hand side of (\ref{commutator tilde}). For the first one, we can write
\[
\|\psi_{\delta}f\|_{\mV^0_{1-\beta,\delta}(\Om)} = \|(r+\delta)^{1-\beta}\psi_{\delta}f\|_{\mrm{L}^2(\Om)} \le C\,\delta^{1-\beta}\,\|f\|_{\mrm{L}^2(\Om)}.
\]
For the second one, using the estimate
\begin{equation}\label{eqnUnifContinuityPropoModel}
|\pi_+(v^{\delta})|+ \Vert v^{\delta} - \pi_{+^{\delta}}(v)\,(\sgp + (\delta/\delta_{\bullet})^{2i\mu}\sgm)\Vert_{\mV_{-1}^{1}(\Omega)}\leq C\,\Vert f\Vert_{\mathring{\mV}_{1}^{1}(\Omega)^{*}} \leq C\,\|f\|_{\mrm{L}^2(\Om)},
\end{equation}
(the same as (\ref{eqnUnifContinuityPropo}) with $\beta=1$) and mimicking (\ref{estimateReminder1}), we find
\begin{equation}\label{estimateReminder1tilde}
\begin{array}{ll}
& \|\big[\div(\sigma^{\delta}\nabla\cdot),\chi_{\delta}\big](v^{\delta}-\breve{m}_{\delta})\|_{\mV^0_{1-\beta,\delta}(\Om)}\\[4pt] 
 \le &
\|(r+\delta)^{1-\beta}\sigma^{\delta}\nabla\chi_{\delta}\cdot\nabla(v^{\delta}-\breve{m}_{\delta})\|_{\mrm{L}^2(\Om)}+\|(r+\delta)^{1-\beta}(v^{\delta}-\breve{m}_{\delta})\,\div(\sigma^{\delta}\nabla\chi_{\delta})\|_{\mrm{L}^2(\Om)}\\[4pt]
 \le & C\,(\delta^{-1}\|(r+\delta)^{1-\beta}\nabla(v^{\delta}-\breve{m}_{\delta})\|_{\mrm{L}^2(\mathbb{Q}^{\delta})}+\delta^{-2}\|(r+\delta)^{1-\beta}(v^{\delta}-\breve{m}_{\delta})\|_{\mrm{L}^2(\mathbb{Q}^{\delta})})\\[4pt]
 \le & C\,\delta^{1-\beta}\,(\|r^{-1}\nabla(v^{\delta}-\breve{m}_{\delta})\|_{\mrm{L}^2(\mathbb{Q}^{\delta})}+\|r^{-2}(v^{\delta}-\breve{m}_{\delta})\|_{\mrm{L}^2(\mathbb{Q}^{\delta})})\\[4pt]
 \le & C\,\delta^{1-\beta}\,\|v^{\delta}-\breve{m}_{\delta}\|_{\mV^1_{-1}(\Om)} \ \le\ C\,\delta^{1-\beta}\,\|f\|_{\mrm{L}^2(\Om)}.
\end{array}
\end{equation}
Analogously, we obtain 
$\|\big[\div(\sigma^{\delta}\nabla\cdot),\psi\big](\breve{V}_{\delta}-\breve{m}_{\delta})\|_{\mV^0_{1-\beta,\delta}(\Om)} \le C\,\delta^{1-\beta}\,\|f\|_{\mrm{L}^2(\Om)}$. Now, we work on the fourth term of the right hand side of (\ref{commutator tilde}). We have
\begin{equation}\label{tiroir1tilde}
 \|\psi\,(\breve{V}_{\delta} - \chi_{\delta}\,\breve{m}_{\delta})\|_{_{\mV^0_{1-\beta,\delta}(\Om)}} \le  \|(r+\delta)^{1-\beta}\psi\,(\breve{V}_{\delta} - \breve{m}_{\delta})\|_{\mrm{L}^2(\Om)}+\|(r+\delta)^{1-\beta}\psi_{\delta}\,\breve{m}_{\delta}\|_{\mrm{L}^2(\Om)}.
\end{equation}
On the one hand, using (\ref{eqnUnifContinuityPropoModel}), we find 
\begin{equation}\label{tiroir2tilde}
\begin{array}{lcl}
\|(r+\delta)^{1-\beta}\psi\,(\breve{V}_{\delta} - \breve{m}_{\delta})\|_{\mrm{L}^2(\Om)}& \hspace{-0.2cm}= & \hspace{-0.2cm}\|(r+\delta)^{1-\beta}\psi\,(\delta^{i\mu}\pi_{+}(v^{\delta})Z_{\delta}- \breve{m}_{\delta})\|_{\mrm{L}^2(\Om)} \\[4pt]
&\hspace{-0.2cm}\le & \hspace{-0.2cm} C\,\delta^{2-\beta}\,\|(1+\rho)^{1-\beta}\psi_{1/\delta}\,\tilde{Z}\|_{\mrm{L}^2(\Xi)}\,\|f\|_{\mrm{L}^2(\Om)} \le  C\,\delta^{2-\beta}\,\|f\|_{\mrm{L}^2(\Om)}.
\end{array}
\end{equation} 
On the other hand, again with (\ref{eqnUnifContinuityPropoModel}), we get
\begin{equation}\label{tiroir3tilde}
\begin{array}{lcl}
\|(r+\delta)^{1-\beta}\psi_{\delta}\,\breve{m}_{\delta}\|_{\mrm{L}^2(\Om)} & \le & |\pi_+(v)|\,\|(r+\delta)^{1-\beta}\psi_{\delta}\|_{\mrm{L}^{2}(\Om)}\,\|r^{+i\mu}+ 
(\delta/\delta_{\bullet})^{2i\mu}r^{-i\mu}\big)\phi\|_{\mrm{L}^{\infty}(\Om)}\\[3pt]
 & \le &  C\,\|(r+\delta)^{1-\beta}\psi_{\delta}\|_{\mrm{L}^2(\Om)}\,\|f\|_{\mrm{L}^2(\Om)} \ \le  \ 
C\,\delta^{2-\beta}\,\|f\|_{\mrm{L}^2(\Om)}.
\end{array}
\end{equation}
Plugging (\ref{tiroir2tilde}) and (\ref{tiroir3tilde}) in (\ref{tiroir1tilde}) furnishes a good estimate for the fourth term of the right hand side of (\ref{commutator tilde}).\\
\newline
Gathering (\ref{estimateReminder1tilde})-(\ref{tiroir3tilde}), we deduce $\|\breve{\mrm{K}}^{\delta}f\|_{\mV^0_{1-\beta,\delta}(\Om)} \le C\,\delta^{1-\beta}\,\|f\|_{\mrm{L}^2(\Om)}$. Together with (\ref{UtilisationStability}), this yields $\|(\mA^{\delta}+i\mrm{Id})^{-1}f - \breve{\mrm{R}}^{\delta}f\|_{\mL^{2}(\Omega)} \le C\,\delta^{1-\beta}\,\|f\|_{\mrm{L}^2(\Om)}$. Taking the supremum over all $f\in\mL^{2}(\Omega)$, since this inequality is true for all $\beta\in(0;1)$, finally we obtain (\ref{lemma stability estimate result tilde}).
\end{proof}

\subsection{Proof of Theorem \ref{propoErrorEtsimate}}\label{paragraphFinalProof}

Consider some given $f\in\mrm{L}^2(\Om)$ and again, set $u^{\delta}=(\mA^{\delta}+i\mrm{Id})^{-1}f$, $\breve{u}^{\delta}=\breve{\mrm{R}}^{\delta}f$, $v^{\delta}=(\mathfrak{A}^{\delta}+i\mrm{Id})^{-1}f$.
Using triangular inequality and the result of Proposition \ref{proposition error estimate}, we can write 
\begin{equation}\label{ImportantTriangular}
\begin{array}{ll}
 & \Vert (\mA^{\delta}+i\mrm{Id})^{-1}f - (\mathfrak{A}^{\delta}+i\mrm{Id})^{-1}f \Vert_{\mL^{2}(\Omega)} \\[5pt] 
  \le & \Vert (\mA^{\delta}+i\mrm{Id})^{-1}f - \breve{\mrm{R}}^{\delta}f \Vert_{\mL^{2}(\Omega)}+\Vert \breve{\mrm{R}}^{\delta}f - (\mathfrak{A}^{\delta}+i\mrm{Id})^{-1}f \Vert_{\mL^{2}(\Omega)}\\[5pt]
    \le & C\,\delta^{1-\eps}\,\|f\|_{\mrm{L}^2(\Om)}+\Vert \breve{u}^{\delta} - v^{\delta} \Vert_{\mL^{2}(\Omega)}.
\end{array}
\end{equation}
Let us assess the term $\Vert\breve{u}^{\delta} - v^{\delta}\Vert_{\mrm{L}^2(\Om)}$. Observing that $\psi\chi_{\delta} = \psi(1-\psi_{\delta}) = \psi-\psi_{\delta}$, we find
\begin{equation}\label{termsToEstimate1}
\begin{array}{l}
\breve{u}^{\delta} - v^{\delta} = -\psi_{\delta}\,( v^{\delta} - \breve{m}_{\delta}) +\psi\,(\breve{V}_{\delta} - \breve{m}_{\delta}).
\end{array}
\end{equation}
We need to derive a proper upper bound for the $\mL^{2}$-norm of each contribution in the right-hand side above. First of all, note that $v^{\delta}-\breve{m}_{\delta}\in\mV^{1}_{-1}(\Omega)$. Observing that the support of $\psi_{\delta}$ is included in the disk $\overline{\mrm{D}(O,2\delta )}$, with (\ref{eqnUnifContinuityPropoModel}), we obtain
\begin{equation}\label{termsToEstimate2}
\Vert \psi_{\delta}\,(v^{\delta} - \breve{m}_{\delta})\Vert_{\mL^{2}(\Omega)}  \le  C\,\delta^2\,\Vert v^{\delta} - \breve{m}_{\delta}\Vert_{\mV^{0}_{-2}(\Omega)}  \le   C\,\delta^2\,\Vert v^{\delta} - \breve{m}_{\delta}\Vert_{\mV^{1}_{-1}(\Omega)} \le   C\,\delta^2\,\Vert f\Vert_{\mrm{L}^2(\Omega)},
\end{equation}
where $C$ is independent of $\delta$. To deal with the second term of the right-hand side of (\ref{termsToEstimate1}), we simply use (\ref{tiroir2tilde}) which gives 
\begin{equation}\label{termsToEstimate3}
\begin{array}{lcl}
\|\psi\,(\breve{V}_{\delta} - \breve{m}_{\delta})\|_{\mrm{L}^2(\Om)} & = & \delta\,|\pi_+(v)|\,\|\psi_{1/\delta}\,\tilde{Z}\|_{\mrm{L}^2(\Xi)}\\[4pt]
 & \le & C\,\delta^{2-\eps}\,\|(1+\rho)^{1-\eps}\psi_{1/\delta}\,\tilde{Z}\|_{\mrm{L}^2(\Xi)}\,\|f\|_{\mrm{L}^2(\Om)} \\[4pt]
 &\le & C\,\delta^{2-\eps}\,\|\tilde{Z}\|_{\mathcal{V}^1_{2-\eps}(\Xi)}\,\|f\|_{\mrm{L}^2(\Om)}\le C\,\delta^{2-\eps}\,\|f\|_{\mrm{L}^2(\Om)}.
\end{array}
\end{equation}
From (\ref{termsToEstimate1})--(\ref{termsToEstimate3}), we infer
\begin{equation}\label{estimImportante2}
\Vert\breve{u}^{\delta} - v^{\delta}\Vert_{\mrm{L}^2(\Om)}\le C\,\delta^{2-\eps}\,\|f\|_{\mrm{L}^2(\Om)}.
\end{equation}
Finally, plugging (\ref{estimImportante2}) in (\ref{ImportantTriangular}) leads to the result of Theorem \ref{propoErrorEtsimate}.

\section{Numerical illustrations}\label{Numerics}

To illustrate the results we proved in the two previous sections, we approximate numerically the spectrum of Problem (\ref{ExPb}) in a canonical geometry. The geometry will be chosen so that we can separate variables and thus, proceed to explicit computations. The framework (see Figure \ref{geom part}) will be slightly different from the 
one introduced in Section \ref{Description of the problem} because $\overline{\Omega_{+}^{\delta}}\cup \overline{\Omega_{-}^{\delta}}$ will not be a fixed domain. 
However, the analysis we provided all along this paper could be extended without difficulty to the geometry studied here and we would obtain analogous results. 

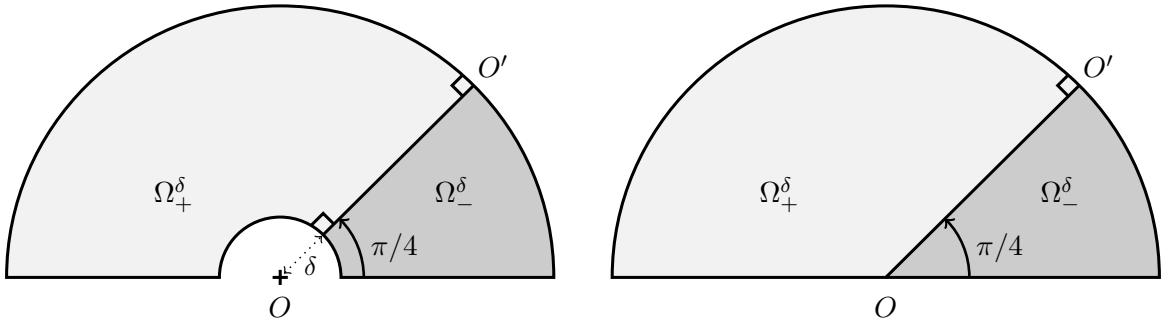
\begin{figure}[!ht]
\centering
\begin{tikzpicture}
\fill[draw=none,fill=gray!10] (3.6,0) arc (0:180:3.6) -- (-0.8,0) arc (180:0:0.8)--cycle;
\fill[draw=none,fill=gray!40] (3.6,0) arc (0:45:3.6) -- (0.5656,0.5656) arc (45:0:0.8)--cycle;
\draw[line width = 0.4mm] (3.6,0) arc (0:180:3.6) -- (-0.8,0) arc (180:0:0.8)--cycle;
\draw[black, line width = 0.4mm](-0.1,0)--(0.1,0);
\draw[black, line width = 0.4mm](0.565685425,0.565685425)--(2.5445584412,2.5445584412);
\draw[black, line width = 0.4mm](0.414264069,0.697106781)--(0.565685425,0.848528138)--(0.707106781,0.707106781);
\draw[black, line width = 0.4mm](2.403137085,2.685979797)--(2.261715729,2.5445584412)--(2.403137085,2.403137085);
\draw[black, line width =1pt,->] (1.1,0) arc (0:45:1.1);
\node at (1.5,0.4){$\pi/4$};
\draw[black, line width = 0.2mm,<->,dotted](0.05,0.05)--(0.55,0.55);
\node at (0.4,0.15){\small $\delta$};
\draw[black, line width = 0.4mm](0,-0.1)--(0,0.1);
\node at (0,-0.4){$O$};
\node at (2.8,2.8){$O'$};
\node at (2.3,1.1){$\Om_{-}^{\delta}$};
\node at (-1.4,1.1){$\Om_{+}^{\delta}$};
\end{tikzpicture}\hspace{0.6cm}
\begin{tikzpicture}
\fill[draw=none,fill=gray!10] (3.6,0) arc (0:180:3.6);
\fill[draw=none,fill=gray!40] (0,0) -- (3.6,0) arc (0:45:3.6);
\draw[line width = 0.4mm] (3.6,0) arc (0:180:3.6) -- cycle;
\draw[black, line width = 0.4mm](0,0)--(2.5445584412,2.5445584412);
\draw[black, line width = 0.4mm](2.403137085,2.685979797)--(2.261715729,2.5445584412)--(2.403137085,2.403137085);
\draw[black, line width =1pt,->] (1.1,0) arc (0:45:1.1);
\node at (1.5,0.4){$\pi/4$};
\node at (0,-0.4){$O$};
\node at (2.8,2.8){$O'$};
\node at (2.3,1.1){$\Om_{-}^{\delta}$};
\node at (-1.4,1.1){$\Om_{+}^{\delta}$};
\end{tikzpicture}
\caption{Domains $\Om^{\delta}$ and $\Om^{0}$.\label{geom part}}
\end{figure}

\noindent Let us first describe the geometry. Consider $\delta\in(0;1)$ and define (see Figure \ref{geom part})
\[
\begin{array}{lcl}
\Om_{+}^{\delta} & := & \{\;(r\cos\theta,r\sin\theta)\ \vert\;\delta<r<1,\;\;\pi/4<\theta<\pi \;\};\\
\Om_{-}^{\delta} & := & \{\;(r\cos\theta,r\sin\theta)\ \vert\;\delta<r<1,\;\;0<\theta<\pi/4 \;\};\\
\Om^{\delta} & := & \{\;(r\cos\theta,r\sin\theta)\ \vert\;\delta<r<1,\;\;0<\theta<\pi \;\}.
\end{array}
\]
Introduce the function $\sigma^{\delta}:\Omega^{\delta}\to \R$ such that $\sigma^{\delta} = \sigma_{\pm}$ in $\Omega_{\pm}^{\delta}$, 
where $\sigma_{+}>0$ and $\sigma_{-}<0$ are constants. We are interested in the eigenvalue problem
\begin{equation}\label{ExPb canon}
\begin{array}{|l}
\dsp{ \textrm{Find}\;(\lambda^{\delta},u^{\delta})\in\Cplx\times\mH^{1}_{0}(\Omega^{\delta})\setminus\{0\}\mbox{ such that}\quad}\\[6pt]
\dsp{ -\div(\sigma^{\delta}\nabla u^{\delta}) = \lambda^{\delta}u^{\delta}\quad\mbox{ in }\Omega^{\delta}.}
\end{array}
\end{equation}
We define the unbounded operator $\mA^{\delta}:D(\mA^{\delta})\to \mL^{2}(\Omega^{\delta})$ such that 
\begin{equation}\label{ExOp canon}
\begin{array}{|l}
\mA^{\delta}\, v \;=\; -\mrm{div}(\sigma^{\delta}\nabla v)\\[6pt]
D(\mA^{\delta})\;:=\{v\in\mH^{1}_{0}(\Omega^{\delta})\;\vert\; \mrm{div}(\sigma^{\delta}\nabla v)\in\mL^{2}(\Omega^{\delta})\}.
\end{array}
\end{equation}
Using an explicit computation relying on the separation of variables, we proved in \cite{ChCNSu} that for $\kappa_{\sigma}=\sigma_-/\sigma_+\in(-\infty,-1)\cup(-1/3;0)$, the operator $\mA^{\delta}$ is injective for all $\delta\in(0;1)$. Moreover, for $\kappa_{\sigma}\in(-1;-1/3)$, $\mA^{\delta}$ is injective  if and only if $\delta\in(0;1)\backslash\cup_{n\in\N^{\ast}}\{\delta^n\}$ with 
\begin{equation}\label{sequence delta n}
\delta^n = \exp{\left(-\cfrac{n\pi^2}{2\,\mrm{acosh}(\frac{1-\kappa_{\sigma}}{2(1+\kappa_{\sigma})})}\right)}
\underset{n\to\infty}{\longrightarrow}0.
\end{equation}

\noindent Now, we discretize Problem (\ref{ExPb canon}). For details concerning the process, we refer the reader to \cite{BoCZ10,NiVe11,ChCiAc}. We impose $\sigma_+=1$. Let us consider $(\mathcal{T}^{\delta}_h)_h$ a shape regular family of triangulations of $\overline{\Omega^{\delta}}$, 
made of triangles. Moreover, we assume that, for any triangle $\tau$, one has either $\tau\subset \overline{\Omega^{\delta}_1}$ 
or $\tau\subset \overline{\Omega^{\delta}_2}$. Define the family of finite element spaces
\[
\mV^{\delta}_h := \left\{v\in \mH_{0}^{1}(\Om^{\delta})\mbox{ such that }v|_{\tau}\in \mathbb{P}_1(\tau)\mbox{ for all }\tau\in\mathcal{T}^{\delta}_h\right\},
\]
where $\mathbb{P}_1(\tau)$ is the space of polynomials of degree at most $1$ on the triangle $\tau$. Let us consider the problem
\begin{equation}\label{Pb h}
\begin{array}{|l}
\dsp{ \textrm{Find}\;(\lambda^{\delta}_h,u^{\delta}_h)\in\Cplx\times\mV^{\delta}_h\setminus\{0\}\mbox{ such that}\quad}\\[6pt]
(\sigma^{\delta} \nabla u^{\delta}_h,\nabla v^{\delta}_h)_{\mrm{L}^2(\Om^{\delta})} = \lambda^{\delta}_h(u^{\delta}_h,v^{\delta}_h)_{\mrm{L}^2(\Om^{\delta})},\quad
\forall v^{\delta}_h\in\mV^{\delta}_h.
\end{array}
\end{equation}
\quad\\
\textsc{$\bullet$ Outside the critical interval (-1;-1/3)}\quad\\ 
In Figure \ref{spectreHI}, we display the ten eigenvalues of smallest modulus of Problem (\ref{Pb h}) with respect to $-\ln\delta$ for a contrast $\kappa_{\sigma}=\sigma_{-}/\sigma_{+}=-1-10^{-4}\notin(-1;-1/3)$. In this case, it is proved in \cite{BoCC12} that the limit problem for $\delta=0$ (see Figure \ref{geom part}, on right) is well-posed in the Fredholm sense in $\mH^1_0(\Om)$. The operator $\mA^0:D(\mA^0)\to \mrm{L}^2(\Om^0)$ is self-adjoint and has compact resolvent. The dashed lines in Figure \ref{spectreHI} represent the approximation of the ten eigenvalues of smallest modulus of the limit operator $\mA^{0}$. The numerical experiments suggest that the spectrum of $\mA^{\delta}$ converges to the spectrum of $\mA^{0}$ as $\delta\to0$. Actually, this can be established. However, since the method is the same as the one we carry out in this paper, in a situation easier to handle, we have chosen not to present the proof.\\
\newline
\noindent\textsc{$\bullet$ Inside the critical interval (-1;-1/3)}\quad\\ 
In Figure \ref{spectreDI}, we display the ten eigenvalues of smallest modulus of Problem (\ref{Pb h}) with respect to $-\ln\delta$ for a contrast $\kappa_{\sigma}=\sigma_{-}/\sigma_{+}=-1+10^{-4}\in(-1;-1/3)$. We observe that the spectrum of $\mA^{\delta}$ depends on $\delta$ even for small $\delta$. In other words, it does not converge to the spectrum of some operator independent of $\delta$. The dashed lines correspond to the expected values of $\delta=\delta^n$ (see (\ref{sequence delta n})), computed explicitly using separation of variables, for which $\mA^{\delta}$ fails to be injective, or equivalently, for which zero belongs to the spectrum of $\mA^{\delta}$. Notice that the spectrum computed numerically indeed passes through zero for these values of $\delta$. Figures \ref{premierevpPosDI} and \ref{premierevpNegDI} represent respectively the approximation of the first positive and the first negative eigenvalue of $\mA^{\delta}$ with respect to $-\ln\delta$. Remark the periodic behaviour. This is consistent with what we proved in Theorem \ref{thmMajor}. Note that in the particular geometry considered here, one can check that Assumption \ref{assumption1} appearing in the statement of Theorem \ref{thmMajor} holds for all $\kappa_{\sigma}\in (-1;-1/3)$ by means of explicit computations using separation of variables.\\ 
Observe that we work here with a contrast very close to $-1$. This may seem surprising because for $\kappa_{\sigma}=-1$, the operators $\mA^{\delta}$ are not of Fredholm type, due to the presence of singularities all over the interface \cite[Thm.\,6.2]{BoCC12}. However, this allows us to obtain several periods in Figures \ref{premierevpPosDI} and \ref{premierevpNegDI} without being obliged to use a very refined mesh. Indeed, in Remark \ref{spectrumPeriodic}, we obtained that asymptotically, the eigenvalues are $\pi/\mu$-periodic in $\ln\delta-$scale, where $\mu$ is defined in (\ref{SingExp}). According to Remark \ref{PropertiesSingu}, we know that for $\kappa_{\sigma}\in(-1;-1/3)$ such that $\kappa_{\sigma}\to -1^{+}$, there holds $\mu\to+\infty$. In our case where $\kappa_{\sigma}=-1+10^{-4}$, the coefficient $|\pi/\mu|=|\pi^2/(2\,\mrm{acosh}(\frac{1-\kappa_{\sigma}}{2(1+\kappa_{\sigma})}))|$ (see (\ref{sequence delta n})) is approximately equal to $-0.5$. From a numerical point of view, it only requires to use meshes which are locally symmetric with respect to the interface to avoid instability phenomena (see \cite{ChCiAc}).\\
\newline
In Figure \ref{numSourceTerm}, we consider the source term problem 
\begin{equation}\label{Pb h source}
\begin{array}{|l}
\textrm{Find}\; u^{\delta}_h\in \mV^{\delta}_h\textrm{ such that}\\[5pt]
(\sigma^{\delta} \nabla u^{\delta}_h,\nabla v^{\delta}_h)_{\Om^{\delta}} = (f, v^{\delta}_h)_{\Om^{\delta}},\quad
\forall v^{\delta}_h\in\mV^{\delta}_h.
\end{array}
\end{equation}
We choose $f$ such that $f(x,y)=100$ if $x<-0.5$ and $f(x,y)=0$ if $x\ge-0.5$. Moreover, we impose $\sigma_+=1$ and $\sigma_+=-1+10^{-4}$. We display the variation of 
$\|u^{\delta}_h\|_{\mH_{0}^{1}(\Om^{\delta})}$ with respect to $1-\delta$. 
We observe peaks which correspond to the values $\delta=\delta^n$ for which $\mA^{\delta}$ fails to be injective. Here, we can do explicit computations to prove this result. For a general geometry where separation of variables does not work, we know from Theorem \ref{thmMajor} that a similar behaviour should be observed. Indeed, $\spec(\mA^{\delta})$ behaves asymptotically as   $\spec(\mathfrak{A}^{\delta})$ as $\delta$ goes to zero, and periodically in $\ln\delta$-scale, $\spec(\mathfrak{A}^{\delta})$ contains the value $0$. Notice that for small values of $\delta$, 
it is very expensive to use a mesh adapted to the geometry. Therefore,  the mesh size is chosen more or less constant with 
respect to $\delta$. This explains why peaks do not appear for small values of $\delta$. 

\begin{figure}[!ht]
\centering
\includegraphics[scale=0.42]{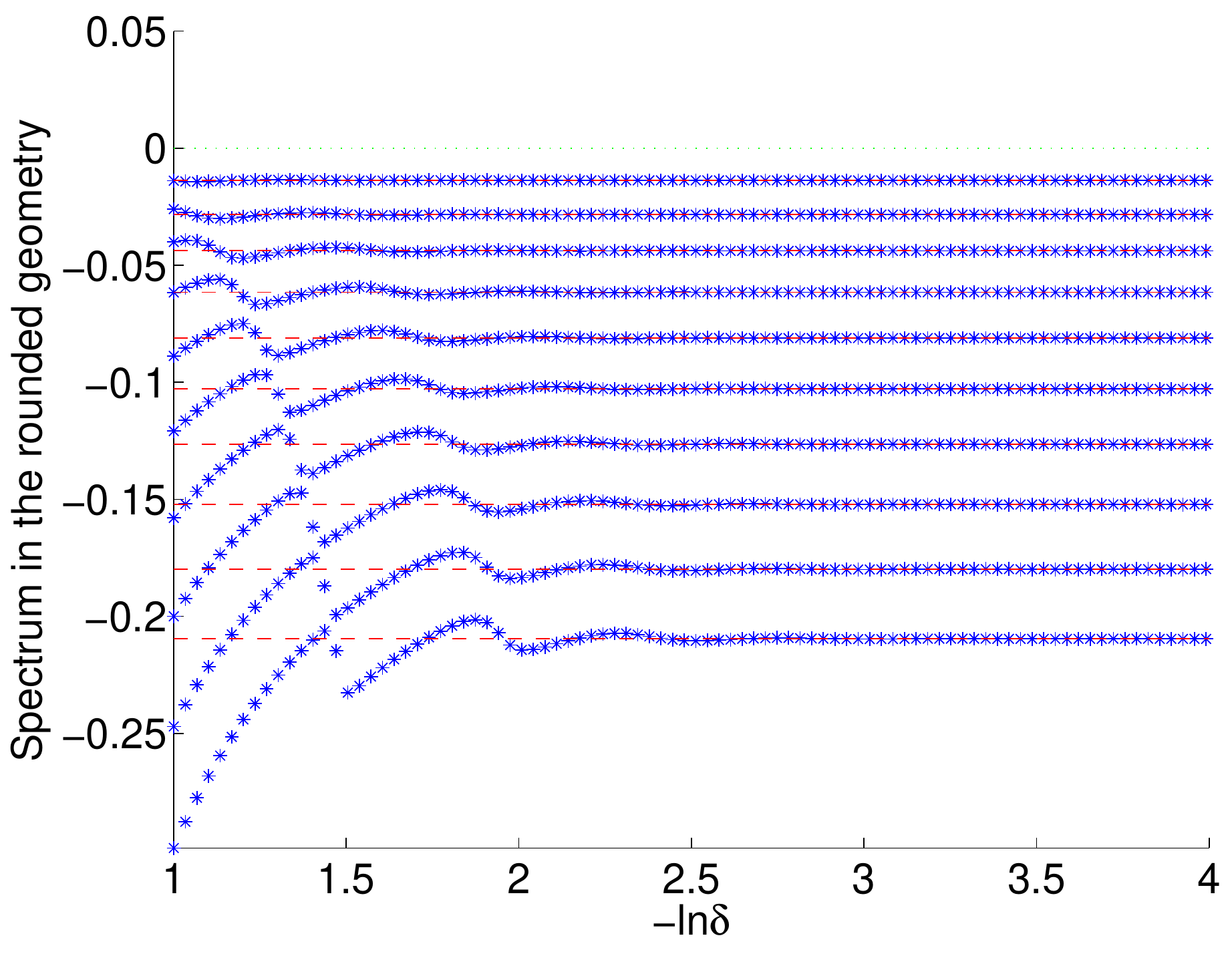}
\caption{For a given $\delta\in(0;1)$, we approximate the ten eigenvalues of smallest modulus of the operator $\mA^{\delta}$. Then, we make $\delta$ tend to zero. The figure represents the approximation of the spectrum of $\mA^{\delta}$ with respect to $-\ln\delta$. The horizontal dashed lines correspond to the approximation of the ten eigenvalues of smallest modulus of the limit operator $\mA^0$.}
\label{spectreHI}
\end{figure}

\begin{figure}[!ht]
\centering
\includegraphics[scale=0.42]{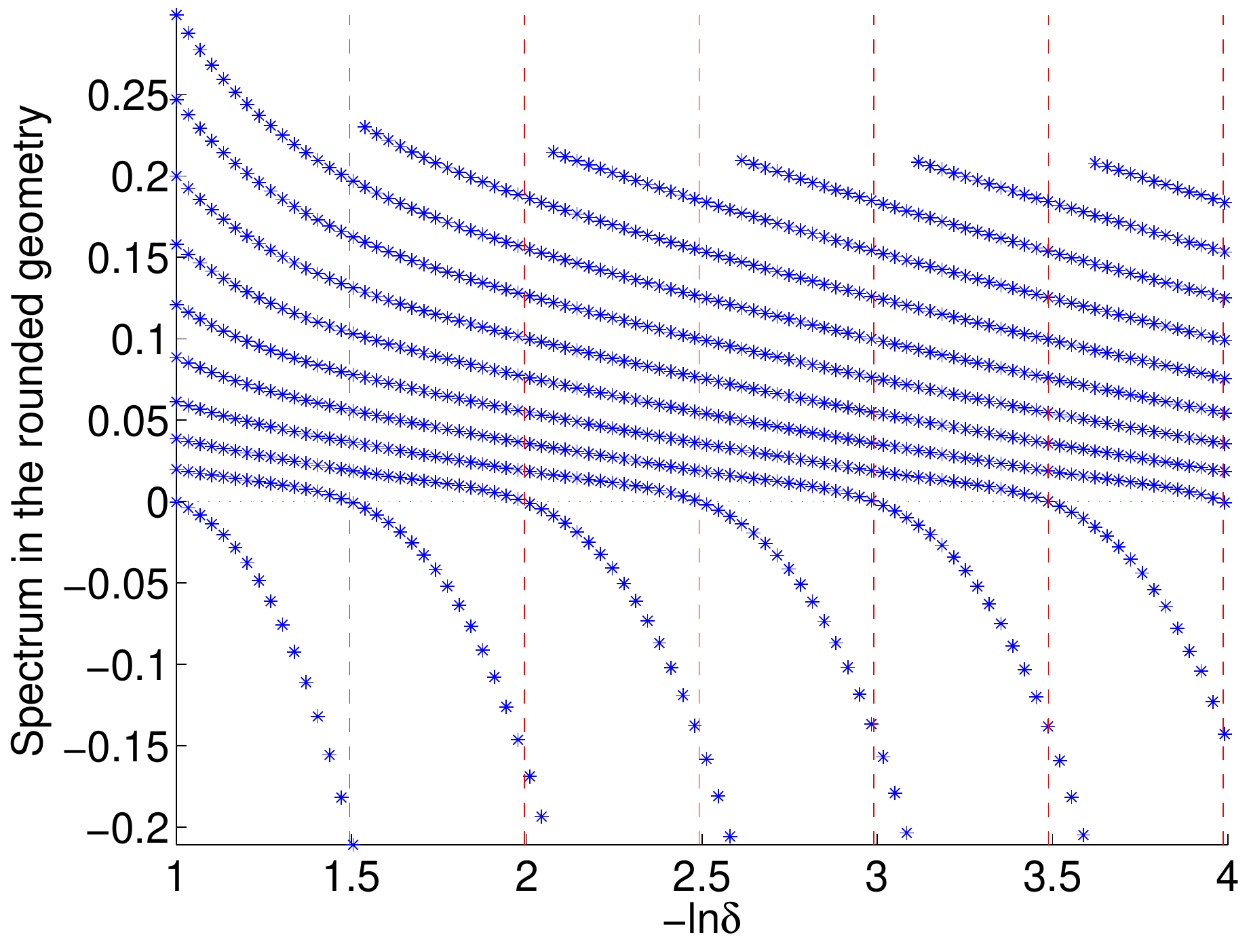}
\caption{For a given $\delta\in(0;1)$, we approximate the ten eigenvalues of smallest modulus of the operator $\mA^{\delta}$. Then, we make $\delta$ tend to zero. The figure represents the approximation of the spectrum of $\mA^{\delta}$ with respect to $-\ln\delta$. The vertical dashed lines correspond to the expected values of $\delta=\delta^n$ for which $\mA^{\delta}$ fails to be injective (see (\ref{sequence delta n})). }
\label{spectreDI}
\end{figure}

\begin{figure}[!ht]
\centering
\includegraphics[scale=0.42]{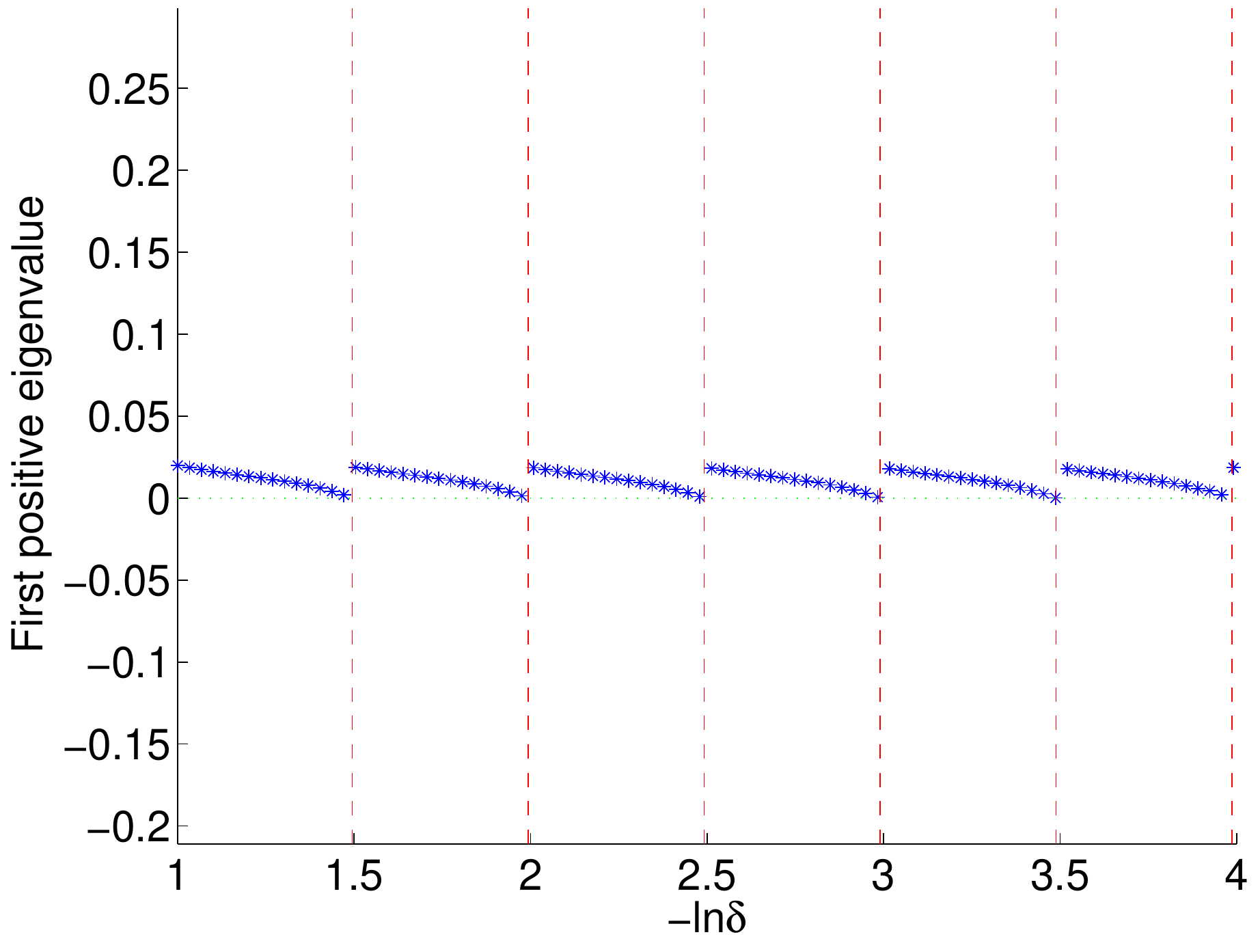}
\caption{Approximation of the first positive eigenvalue of $\mA^{\delta}$ with respect to $-\ln\delta$. The vertical dashed lines correspond to the expected values of $\delta=\delta^n$ for which $\mA^{\delta}$ fails to be injective (see (\ref{sequence delta n})).}
\label{premierevpPosDI}
\end{figure}

\begin{figure}[!ht]
\centering
\includegraphics[scale=0.42]{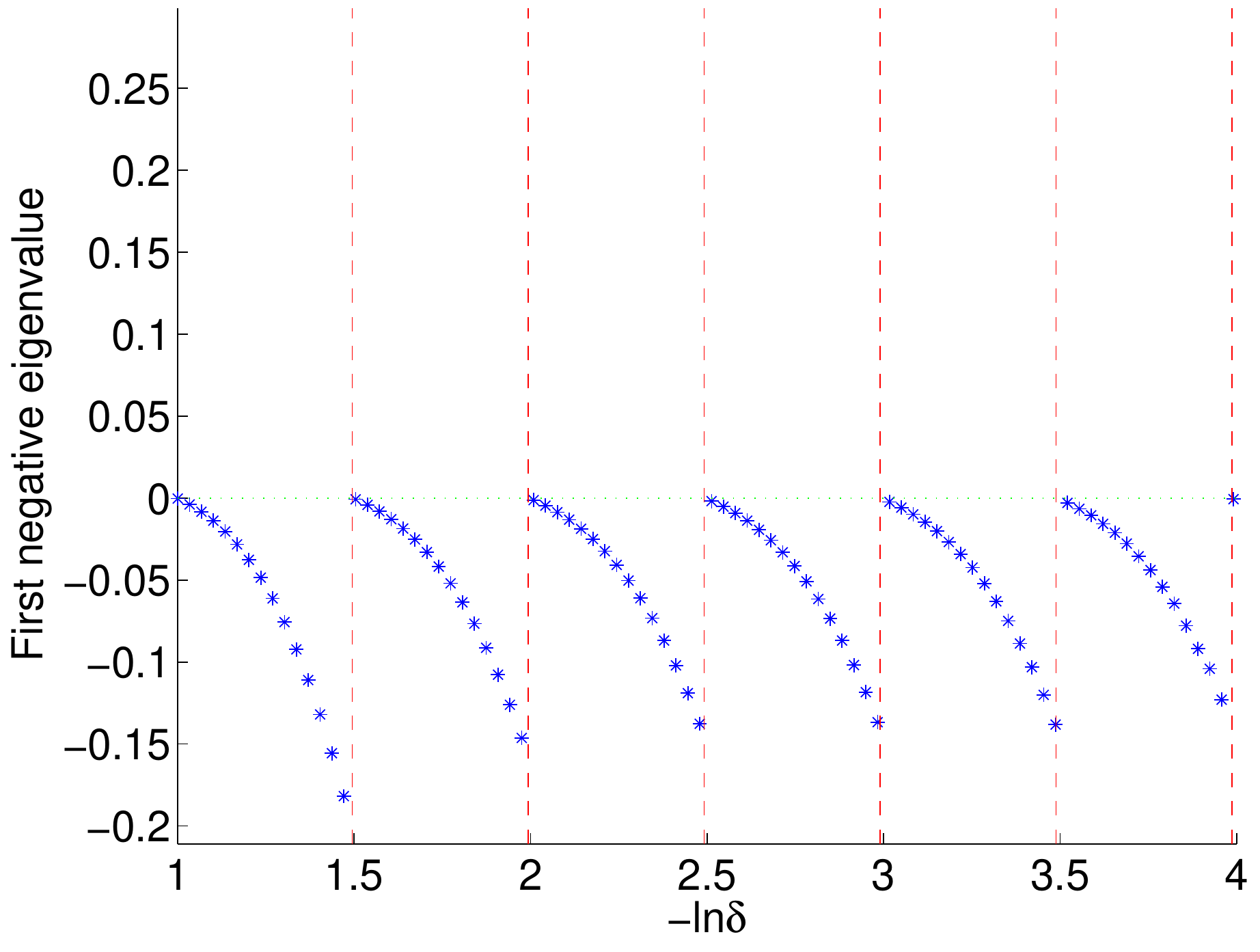}
\caption{Approximation of the first negative eigenvalue of $\mA^{\delta}$ with respect to $-\ln\delta$. The vertical dashed lines correspond to the expected values of $\delta=\delta^n$ for which $\mA^{\delta}$ fails to be injective (see (\ref{sequence delta n})).}
\label{premierevpNegDI}
\end{figure}

\begin{figure}[!ht]
\centering
\includegraphics[scale=0.42]{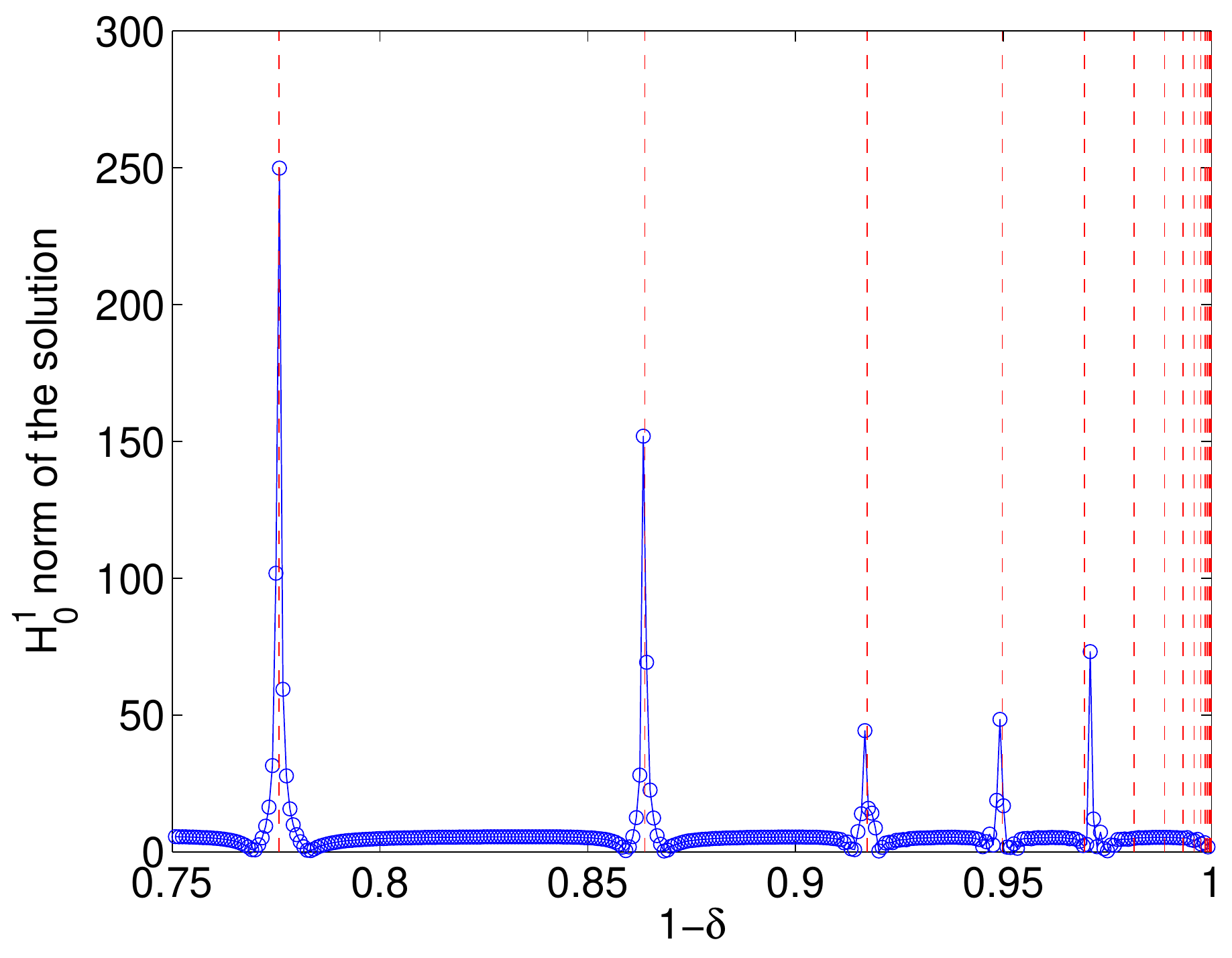}
\caption{Variation of $\|u^{\delta}_h\|_{\mH_{0}^{1}(\Om^{\delta})}$ with respect to $1-\delta$. The vertical dashed lines 
correspond to the expected values of $\delta=\delta^n$ for which $\mA^{\delta}$ fails to be injective (see (\ref{sequence delta n})).\label{numSourceTerm}}
\end{figure}

\newpage

\section*{Appendix}
\noindent In this appendix, first we briefly recall an elementary result of spectral theory that we used in this article in order to estimate the distance of a number to the spectrum of an operator. We provide a proof for the sake of completeness. Then, we use it to complete the demonstration of Proposition \ref{SpectralConvergence}.

\begin{lemma}\label{EstimEigenVal}
Let $\mrm{H}$, equipped with the inner product $(\cdot,\cdot)_{\mrm{H}}$ and the norm $\Vert \cdot\Vert_{\mrm{H}}$, be a Hilbert space. 
For any (\textit{a priori} unbounded) normal linear operator $\mrm{A}: D(\mrm{A})\subset \mrm{H}
\to\mrm{H}$ we have 
$$
\inf_{\lambda\in \mathfrak{S}(\mrm{A})}\vert \lambda - \mu\vert\leq \inf_{v\in D(\mA)\setminus\{0\}}
\frac{\Vert \mA v-\mu v \Vert_{\mrm{H}}}{\Vert v\Vert_{\mrm{H}}},\qquad\forall \mu\in \mathbb{C}.
$$
\end{lemma}
\begin{proof}
Since $\mA$ is normal then, according to the spectral theorem \cite[Thm.\,6.6.1]{BiSo87}, it admits a spectral 
decomposition $\mA = \int_{\mathfrak{S}(\mA)}\zeta d\mrm{E}(\zeta)$ where $\mrm{E}(\zeta)$ refers to a spectral 
measure on $\mrm{H}$. Let $d\mrm{E}_{v,v}$ refer to the measure associated with $\zeta\mapsto 
(\mrm{E}(\zeta)v,v)_{\mrm{H}}$. A spectral decomposition of  $\mA-\mu\mrm{Id}$ is given by
$\mA-\mu\mrm{Id} = \int_{\mathfrak{S}(\mA)}(\zeta - \mu) d\mrm{E}(\zeta)$. Moreover the formula 
$\Vert \mA v-\mu v\Vert_{\mrm{H}}^{2} = \int_{\mathfrak{S}(\mA)}\vert\zeta - \mu\vert^{2} d\mrm{E}_{v,v}(\zeta)$
holds for any $v\in D(\mA)$. As a consequence, we have  
$$
\Vert v\Vert_{\mrm{H}}^{2}\inf_{\lambda\in \mathfrak{S}(\mrm{A})}\vert \lambda - \mu\vert^{2} = 
\inf_{\lambda\in \mathfrak{S}(\mrm{A})}\vert \lambda - \mu\vert^{2}\int_{\mathfrak{S}(\mA)}d\mrm{E}_{v,v}(\zeta)\leq 
\int_{\mathfrak{S}(\mA)}\vert \zeta - \mu\vert^{2}d\mrm{E}_{v,v}(\zeta) = \Vert \mA v-\mu v\Vert_{\mrm{H}}^{2}.
$$
Since this holds for any $v\in D(\mA)$, we can divide by $\Vert v\Vert_{\mrm{H}}^{2}$ and take the $\inf$ in the 
right hand side of the estimate above, which yields the desired inequality.
\end{proof}

\noindent\textbf{Proof of Proposition \ref{SpectralConvergence}.} First, let us show that 
\begin{equation}\label{distanceSpectraProof}
\sup_{\eta\in \mathfrak{S}(\mathfrak{A}^{\delta})}\mathop{\inf\phantom{p}}_{\lambda\in \mathfrak{S}(\mA^{\delta})}\Big\vert \frac{1}{\lambda+i}-\frac{1}{\eta+i}\Big\vert\;\leq\; 
C_{\eps}\;\delta^{1-\eps},\quad\forall \delta\in(0;1].
\end{equation}
Pick some $\eta\in \mathfrak{S}(\mathfrak{A}^{\delta})$. Clearly, we have 
\[
\mathop{\inf\phantom{p}}_{\lambda\in \mathfrak{S}(\mA^{\delta})}\Big\vert \frac{1}{\lambda+i}-\frac{1}{\eta+i}\Big\vert = \mathop{\inf\phantom{p}}_{\tilde{\lambda}\in \mathfrak{S}((\mA^{\delta}+i\mrm{Id})^{-1})}\Big\vert \tilde{\lambda}-\frac{1}{\eta+i}\Big\vert. 
\]
Note that $(\mA^{\delta}+i\mrm{Id})^{-1}$ is a normal operator such that $D((\mA^{\delta}+i\mrm{Id})^{-1})=\mrm{L}^2(\Om)$. As a consequence, according to Lemma \ref{EstimEigenVal} above, we deduce 
\begin{equation}\label{eqIntermSpectral}
\mathop{\inf\phantom{p}}_{\tilde{\lambda}\in \mathfrak{S}((\mA^{\delta}+i\mrm{Id})^{-1})}\Big\vert \tilde{\lambda}-\frac{1}{\eta+i}\Big\vert\;\leq\;\inf_{v\in \mrm{L}^2(\Om)\setminus\{0\}}
\frac{\Vert (\mA^{\delta}+i\mrm{Id})^{-1}v-(\eta+i)^{-1} v \Vert_{\mrm{L}^2(\Om)}}{\Vert v\Vert_{\mrm{L}^2(\Om)}}.
\end{equation}
Since $\eta\in \mathfrak{S}(\mathfrak{A}^{\delta})$, there is some $v\in\mrm{L}^2(\Om)\setminus\{0\}$ such that $(\mathfrak{A}^{\delta}+i\mrm{Id})^{-1}v=(\eta+i)^{-1} v$. From (\ref{eqIntermSpectral}) and (\ref{error estimate result L2}), we infer 
\[
\mathop{\inf\phantom{p}}_{\tilde{\lambda}\in \mathfrak{S}((\mA^{\delta}+i\mrm{Id})^{-1})}\Big\vert \tilde{\lambda}-\frac{1}{\eta+i}\Big\vert\;\leq\;
\frac{\Vert (\mA^{\delta}+i\mrm{Id})^{-1}v-(\mathfrak{A}^{\delta}+i\mrm{Id})^{-1}v \Vert_{\mrm{L}^2(\Om)}}{\Vert v\Vert_{\mrm{L}^2(\Om)}} \;\leq \;C_{\eps}\,\delta^{1-\eps}.
\]
Taking the supremum over all $\eta\in \mathfrak{S}(\mathfrak{A}^{\delta})$, we obtain (\ref{distanceSpectraProof}). Since the roles of 
$\mA^{\delta}$ and $\mathfrak{A}^{\delta}$ are symmetric,  (\ref{distanceSpectra}) is proved.\hfill$\square$

\section*{Acknowledgments} 
The research of L. C. and X. C. is supported by the ANR project METAMATH, grant ANR-11-MONU-016 of the French Agence Nationale de la Recherche. The research of L. C. is supported by the FMJH through the grant ANR-10-CAMP-0151-02 in the ``Programme des Investissements
d'Avenir''. The research of S.A. N. is supported by the Russian Foundation for Basic Research, grant No. 15-01-02175.

\bibliography{Bibli}
\bibliographystyle{plain}

\end{document}